\documentclass[11pt,twoside]{preprint}
\usepackage[a4paper,innermargin=1.2in,outermargin=1.2in,
bottom=1.5in,marginparwidth=1in,marginparsep
=3mm]{geometry}
\usepackage{amsmath,amsthm,amssymb,enumerate}
\usepackage{hyperref}
\usepackage{mathrsfs}
\usepackage{breakurl}
\usepackage[nameinlink,capitalize]{cleveref}

\usepackage[osf]{newtxtext}
\usepackage[bigdelims,vvarbb,cmintegrals]{newtxmath}

\usepackage{xcolor}
\usepackage{mhequ}
\usepackage{comment}
\usepackage{microtype}
\usepackage{pifont}

\colorlet{darkblue}{blue!90!black}
\colorlet{darkred}{red!90!black}
\colorlet{darkgreen}{green!60!black}

\newtheorem*{assumption}{Assumption}

\newtheorem{theorem}{Theorem}[section]
\newtheorem{lemma}[theorem]{Lemma}
\newtheorem{proposition}[theorem]{Proposition}
\newtheorem{corollary}[theorem]{Corollary}

\theoremstyle{definition}
\newtheorem{definition}[theorem]{Definition}

\theoremstyle{remark}
\newtheorem{remark}[theorem]{Remark}

\newcommand{\vn}[1]{{\vert\kern-0.23ex\vert\kern-0.23ex\vert #1 
    \vert\kern-0.23ex\vert\kern-0.23ex\vert}}
\usepackage{mathtools}

\DeclarePairedDelimiter\floor{\lfloor}{\rfloor}

\allowdisplaybreaks

\newcommand{\I}{\mathbb{I}}

\newcommand\bone{\mathbf{1}}

\newcommand\cR{\mathcal{R}}

\newcommand\cD{\mathcal{D}}

\newcommand\cA{\mathcal{A}}

\newcommand\cC{\mathcal{C}}

\newcommand\cN{\mathcal{N}}

\newcommand\cI{\mathcal{I}}

\newcommand\cQ{\mathcal{Q}}
\newcommand\cZ{\mathcal{Z}}

\newcommand\cP{\mathcal{P}}
\newcommand\cF{\mathcal{F}}

\newcommand\scC{\mathscr{C}}

\newcommand\scD{\mathscr{D}}

\newcommand{\bg}{\mathbf{g}}

\def\eps{\varepsilon}

\newcommand{\R}{{\mathbb{R}}}

\renewcommand{\P}{\mathbb{P}}

\def\Q{\mathbb{Q}}

\newcommand{\E}{\mathbb{E}}

\newcommand{\bB}{\mathbf{B}}
\newcommand{\bS}{\mathbb{S}}
\newcommand{\bbB}{\mathbb{B}}
\newcommand{\bF}{\mathbb{F}}

\newcommand{\N}{\mathbb{N}}

\def\what{\widehat}

\def\d{\partial}

\def\path{\mathrm{path}}

\def\Ito{\mathrm{It}\hat{\mathrm{o}}}

\def\({\left(}
\def\){\right)}

\begin{document}
\title{Path-by-path regularisation through multiplicative noise in rough, Young, and ordinary differential equations}
\author{Konstaninos Dareiotis and M\'at\'e Gerencs\'er}
\maketitle

\begin{abstract}
Differential equations perturbed by multiplicative fractional Brownian motions are considered. Depending on the value of the Hurst parameter $H$, the resulting equation is pathwise viewed as an ODE, YDE, or RDE. In all three regimes we show regularisation by noise phenomena by proving the strongest kind of well-posedness with irregular drift: strong existence and path-by-path uniqueness. In the Young and smooth regime $H>1/2$ the condition on the drift coefficient is optimal in the sense that it agrees with the one known for the additive case \cite{CG16, Mate}.
In the rough regime $H\in(1/3,1/2)$ we assume positive but arbitrarily small drift regularity for strong well-posedness, while for distributional drift we obtain weak existence.
\end{abstract}
\tableofcontents

\section{Introduction}\label{sec:intro}
\subsection{Introduction}
The regularising effects of fractional Brownian motions as additive perturbations of differential equations with irregular coefficients, in the sense of restoring well-posedness, are very well studied, see e.g. \cite{NO1, NO2, CG16, BNP, HP, Khoa, GH20, GHM, GGNoiseless, ART21, Mate, BH22}.
These works provide far-reaching extensions of results on standard Brownian-driven stochastic differential equations,
without a large part of the Markovian toolbox (It\^o's formula, Kolmogorov equation, Zvonkin transformation, martingale problem, etc.) available.

One main shortcoming compared to the standard Brownian case is that while therein most results are valid with multiplicative noise thanks to It\^o calculus (for a small but diverse sample see \cite{Veretennikov, Zhang, CHM, FHL}), in the non-Markovian setting no well-posedness  with irregular drift and multiplicative noise has been yet established. 

The difficulty can be best illustrated in the most irregular case, which in the present paper will be the case of Hurst parameter $H\in(1/3,1/2)$.
A prototypical equation with multiplicative noise then reads as
\begin{equ}\label{eq:main}
X_t=x_0+\int_{s_0}^t b(X_s)\,ds+\int_{s_0}^t \sigma(X_s)\,dB^H_s.
\end{equ}
Here $d,d_0\in\N$, $B^H$ is a $d_0$-dimensional fractional Brownian motion with Hurst parameter $H$, $s_0\in[0,1)$, $x_0\in\R^d$, and $b:\R^d\to\R^d$, $\sigma:\R^d\to\R^{d\times d_0}$.
The driving noise in \eqref{eq:main} is both a friend and an enemy: the former side provides the regularisation of the first integral, and the latter side makes the second integral ill-defined. By lifting $B^H$ to a rough path \cite{Lyons}, one can define this as a rough integral, but working pathwise seems to also go in the opposite direction of the probabilistic aspects of regularisation by noise. 
The goal of this work is to reconcile these two sides of the noise and show strong well-posedness of \eqref{eq:main} with irregular $b$.
We state the main result here in a somewhat informal way, the precise formulation is given in Theorems \ref{thm:main} and \ref{thm:weak} below after introducing all necessary concepts and assumptions.

\begin{theorem}\label{thm:main-informal}
If $H\in(1/3,\infty)\setminus \N$, $\alpha>(1-1/(2H))\vee 0$, $b\in \cC^\alpha$,  $\sigma\in \cC^{\floor{1/H}+1}$, and $\sigma$ is uniformly elliptic, then
strong existence and path-by-path uniqueness holds for \eqref{eq:main}.
If $H\in(1/3,1/2]$, $\alpha>1/2-1/(2H)$, $b\in \cC^\alpha$, $\sigma\in \cC^{2}$, and $\sigma$ is uniformly elliptic, then weak existence holds for \eqref{eq:main}.
\end{theorem}


\textbf{Acknowledgements.}
MG thanks Lucio Galeati for various fruitful discussions during the development of a parallel article \cite{Lucio-Mate}, particularly for suggesting the path-by-path approach of \cite{shap}.
MG was funded by the Austrian Science Fund (FWF) Stand-Alone programme P 34992.
The article was finalised during a Research in Groups stay at ICMS, Edinburgh. The authors thank the institute for their excellent support and hospitality.

\subsection{Setup}\label{sec:setup}
\emph{Probabilistic setup}.
We fix a probability space $(\Omega,\cF,\P)$ with a complete filtration $\bF=(\cF_t)_{t\in\R}$ carrying a two-sided $d_0$-dimensional $\bF$-Wiener process $W$.
For $H\in(0,1)$ we define $B^H$ through the Mandelbrot -- van Ness representation
\begin{equ}
B^H_t=\int_{-\infty}^0 \big(|t-s|^{H-1/2}-|s|^{H-1/2}\big)\,dW_s+\int_0^t|t-s|^{H-1/2}\,dW_s.
\end{equ}     \label{eq:mandelbrot}
For $H\in(1,\infty)\setminus\N$ we set
\begin{equ}
B^H_t=\int_{0\leq r_1\leq\cdots\leq r_{\lfloor H\rfloor}\leq t}B^{H-\lfloor H\rfloor}_{r_1}\,dr_1\cdots \,dr_{\lfloor H\rfloor}.
\end{equ}
Conditional expectation given $\cF_s$ is denoted by $\E^s$.
Denoting $[0,1]_{\leq}^2:=\{(s,t)\in[0,1]^2:\,s\leq t\}$ (and similarly $[0,1]_{\leq}^3:=\{(s,u,t)\in[0,1]^3:\,s\leq u\leq t\}$), we have that for all $(s,t)\in[0,1]^2_\leq$, $B^H_t-\E^s B^H_t$ is independent of $\cF_s$ and has Gaussian distribution with mean $0$ and covariance matrix $c(H)|t-s|^{2H}\I$, where $\I$ is the identity matrix (here in $d_0\times d_0$ dimension) and $c(H)$ is a positive constant depending only $H$.
As a consequence, for any measurable bounded function $f:\R^{d_0}\to\R$, $0\leq s\leq t$, and any $\cF_s$-measurable $\R^{d_0}$-valued random vector $X$ one has a.s.
\begin{equation}\label{eq:very-basic}
\E^s f(B^H_t+X)=\cP_{c(H)|t-s|^{2H}} f(\E^s B^H_t+X),
\end{equation}
where for $r>0$, $\cP_r$ is the heat semigroup on $\R^{d_0}$ given by $\cP_r f=p_r\ast f$, with
\begin{equ}
p_r(x)=\frac{1}{(2\pi r)^{d_0/2}}\exp\Big(-\frac{|x|^2}{2r}\Big)
\end{equ}
for $x\in\R^{d_0}$.
More generally, suppose that $f:\R^d\to\R$ is a bounded measurable function, $0\leq s\leq t$, $X$ is an $\cF_s$-measurable $\R^d$-valued random vector and $Y$ is a $\cF_s$-measurable $\R^{d\times d_0}$-valued random matrix $Y$ such that $Y Y^\top$ is a.s. positive definite.
Then we have a.s.
\begin{equation}\label{eq:very-basic2}
\E^s f(Y B^H_t+X)=\cP_{c(H)|t-s|^{2H}YY^\top} f(Y\E^s B^H_t+X),
\end{equation}
where for a positive definite matrix $\Gamma\in\R^{d\times d}$, $\cP_\Gamma$ is given by $\cP_\Gamma f=p_\Gamma\ast f$, with
\begin{equ}\label{eq:HKdef}
p_\Gamma(x)=\frac{1}{(2\pi)^{d/2}(\det\Gamma)^{1/2}}\exp\Big(-\frac{x^\top\Gamma^{-1}x}{2}\Big)
\end{equ}
for $x\in\R^d$. The natural ordering on positive definite matrices is denoted by $\preceq$.
The extension of the above observations to (finite dimensional) vector-valued $f$ is straightforward by working coordinate-wise.

\emph{Function spaces}.
For details on the path spaces and integration concepts used in the sequel we refer the reader to the monograph \cite{Friz-Hairer}.
For two bounds (\eqref{eq:sum_partition} and \eqref{eq:quadratic-controled-estimate}) which are very slightly nonstandard, we provide brief proofs in the Appendix.

For $\alpha\in(0,1)$ and a Borel subset $Q$ of $\R^n$, $n\in\N$, we denote by $\cC^\alpha(Q)$ the set of functions $f:Q\to\R$ such that
\begin{equ}
\|f\|_{\cC^\alpha(Q)}:=\sup_{x\in Q}|f(x)|+[f]_{ \cC^\alpha(Q)}:=\sup_{x\in Q}|f(x)|+\sup_{x\neq y\in Q}\frac{|f(x)-f(y)|}{|x-y|^\alpha}<\infty.
\end{equ}
With a slightly unconventional notation we set $\cC^0(Q)$ to be the set of bounded measurable functions equipped with the supremum norm (by convention, the seminorm on $\cC^0$ is set to be identically $0$).
These definitions easily extend to spaces of functions with values in an arbitrary normed space $V$. If $V$ is a finite dimensional vector space and is obvious from the context, we do not include it in the notation, while if $V=L^p(\Omega)$, we use the notation $\scC^\alpha_p$.
Similarly, $Q$ is dropped from the notation when it is understood from the context.
Finally, for $\alpha\in[1,\infty)$ we denote by $\cC^\alpha(Q)$ the set of functions whose weak derivatives of order $0,1,\ldots,\lfloor \alpha\rfloor$ all have representatives belonging to $\cC^{\alpha-\lfloor\alpha\rfloor}(Q)$.
Their norm is inherited from the norms of their derivatives of order $0,1,\ldots,\lfloor \alpha\rfloor$, while their seminorm is inherited from from the norms of their derivatives of order $1,\ldots,\lfloor \alpha\rfloor$.
Note that if $F\in \cC^1(\R^{n_0},\R^{n_1})$ and $f\in \cC^\alpha(Q,\R^{n_0})$ with $\alpha\in(0,1]$, then $F(f)\in \cC^\alpha(Q,\R^{n_1})$. Quantitatively, one has
\begin{equ}
\,[F(f)]_{\cC^\alpha}\leq [F]_{\cC^1}[f]_{\cC^\alpha},\qquad [F(f)-F(g)]_{C^\alpha}\leq[F]_{\cC^2}(1+[f]_{\cC^\alpha})\|f-g\|_{\cC^\alpha}.
\end{equ} 
For $\alpha\in(-1,0)$ we only ever use $Q=\R^d$, and we understand $\cC^\alpha(\R^d)$ to be the inhomogeneous H\"older-Besov space $B^\alpha_{\infty,\infty}$. For $Q$ as above and $\alpha\in(0,1]$ we define $\cC^{\alpha}_2(Q^2)$ as the set of continuous functions $f:Q^2\to\R$ such that
\begin{equ}
\|f\|_{\cC^\alpha_2(Q^2)}:=\sup_{x,y\in Q}|f(x,y)|+[f]_{ \cC^\alpha_2(Q^2)}:=\sup_{x,y\in Q}|f(x,y)|+\sup_{x\neq y\in Q}\frac{|f(x,y)|}{|x-y|^\alpha}<\infty.
\end{equ}

For $\alpha\in(1/3,1/2]$ we denote by $\cR^\alpha([0,1],\R^{d_0})$ the set of $\alpha$-H\"older rough paths on $[0,1]$: the subset of $\cC^\alpha([0,1],\R^{d_0})\times \cC^{2\alpha}_2([0,1]^2,\R^{d_0\times d_0})$ constrained by the nonlinear relation (\emph{Chen's identity}) postulating that any $(g,\bg)\in\cR^\alpha$ satisfies
\begin{equ}
\bg_{s,t}-\bg_{s,u}-\bg_{u,t}=(g_t-g_s)\otimes (g_s-g_u)
\end{equ}
for all $(s,u,t)\in[0,1]^3_\leq$.
Such a $(g,\bg)\in\cR^\alpha$ is also called a \emph{lift} of $g$.

Given $g\in \cC^\alpha([0,1],\R^{d_0})$ and\footnote{Most statements in \cite{Friz-Hairer} are stated with $\gamma=2\alpha$, but their generalisation is straightforward, see \cite[Exercise~4.7]{Friz-Hairer}} $\gamma\in(\alpha,2\alpha]$, we denote by $\cD^\gamma_g([0,1],\R^n)$ the set of controlled paths: functions $(f,f'):[0,1]\to \R^{n}\times\R^{n\times d_0}$ such that
\begin{equ}
\,[(f,f')]_{ \cD^\gamma_g([0,1])}:=\sup_{0\leq s<t\leq 1}\frac{|f_t-f_s-f'_s(g_t-g_s)|}{|t-s|^\gamma}+[f']_{\cC^{\gamma-\alpha}([0,1])}<\infty.
\end{equ}
As before, the target space $\R^n$ is dropped from the notation when it does not cause confusion.
For $(s,t)\in[0,1]^2_\leq$ the seminorms $[(f,f')]_{ \cD^\gamma_g([s,t])}$ are defined in the obvious way, and we further set $\|(f,f')\|_{\cD^\gamma_g}:=[(f,f')]_{\cD^\gamma_g}+\|f'\|_{\cC^0}+\|f\|_{\cC^0}$.
A function $f'$ such that $(f,f')\in \cD^\gamma_g([0,1])$ is also called a Gubinelli derivative of $f$.
Note the trivial embedding $\cC^\gamma([0,1])\subset \cD^\gamma_g([0,1])$ by setting the Gubinelli derivative to be $0$.
An embedding in an opposite direction is given by
$\cD^\gamma_g([0,1])\subset\cC^\alpha([0,1])$ by dropping the Gubinelli derivative, along with the estimate
\begin{equ}\label{eq:triv-controll-embedding}
\,[f]_{\cC^\alpha([s,t])}\leq [(f,f')]_{ \cD^\gamma_g([s,t])}+ \|f'\|_{\cC^0([s, t])}[g]_{\cC^\alpha([s,t])}\big).
\end{equ} 
Here and in all other bounds below, $(s,t)\in[0,1]^2_\leq$.
Note also the two elementary inequalities: if $\gamma' \in (\gamma, 2\alpha]$, then
\begin{equs}    \label{eq:controll-increasing}
\, [ (f,f')]_{\cD^{\gamma}_g([s, t])} \leq |t-s|^{\gamma'-\gamma}   [ (f,f')]_{\cD^{\gamma'}_g([s, t])},
\end{equs}
while  for $k\geq1$, $0 \leq u_0\leq u_1\leq\cdots\leq u_k$ one has
\begin{equs}      \label{eq:sum_partition}
\, [(f,f')]_{ \cD^\gamma_g([u_0, u_k])}   \leq 
\sum_{i=1}^k [(f,f')]_{ \cD^\gamma_g([u_{i-1},u_i])} +\sum_{i=1}^{k-1}
[ f' ]_{\cC^{\gamma-\alpha}([u_{i-1},u_i])} [ g]_{\cC^\alpha([u_i,u_k])}.
\end{equs}
If $(f,f')\in \cD^\gamma_g([0,1],\R^{n_0})$ and $\cF\in \cC^2(\R^{n_0},\R^{n_1})$, then $(F(f),\nabla F(f)f')\in \cD^\gamma_g([0,1],\R^{n_1})$.   Moreover, 
one has a bound on the growth of this composition operation, which we formulate in a somewhat nonstandard form (compare with e.g. \cite[Lemma~7.3]{Friz-Hairer}), but the proof is similarly elementary (see the Apendix). If $\beta\in[0,1]$ satisfies $\beta\alpha+\alpha\geq \gamma$, then
\begin{equs}    
\, \big[ &(F(f),\nabla F(f)f')\big]_{\cD^\gamma_g([s, t])} \leq 
 \|F\|_{\cC^1}(1+\|f'\|_{\cC^0([s, t])})[(f,f')]_{\cD^\gamma_g([s, t]) }\label{eq:quadratic-controled-estimate}
\\
&+\|F\|_{\cC^2}\|f'\|_{\cC^0([s, t])}\big(1+[g]_{\cC^\alpha([s, t])}\big)\Big([(f,f')]_{\cD^\gamma_g([s, t]) }^\beta +[g]_{\cC^\alpha([s, t])}^{\beta}\|f'\|_{\cC^0([s,t])}^\beta \Big)     .      
\end{equs}
The reason for this convoluted form is that when $(f,f')$ is a rough integral (see below), then $\|f'\|_{\cC^0}$ can be controlled easily via bounds on the coefficients, so the quadratic bound \eqref{eq:quadratic-controled-estimate} becomes practically linear.
Concerning the 
continuity of the composition it suffices to recall \cite[Theorem~7.6]{Friz-Hairer}: if for some $K<\infty$, one has $\|(f,f')\|_{\cD^\gamma_g([s, t])}$, $\|(\bar f,\bar f')\|_{\cD^\gamma_g([s, t])}\leq K$, then the bound
\begin{equs}
\big[&\big(F(f)-F(\bar f),\nabla F(f)f'-\nabla F(\bar f)\bar f'\big)\big]_{\cD^\gamma_g([s, t])}
\\
&\leq N\big(|f_s-\bar f_s|+|f'_s-\bar f'_s|+[(f-\bar f,f'-\bar f')]_{\cD^\gamma_g([s, t])}\big),\label{eq:controlled-comp-stability}
\end{equs}
holds with a constant $N=N(K,\|F\|_{\cC^3},\|g\|_{\cC^\alpha([s,t])})$.

\emph{Integration}.
Let $n\in\N$ and $\alpha,\beta\in(0,1)$ such that $\alpha+\beta>1$.
Then for $f\in \cC^\beta([0,1],\R^n)$ and $g\in \cC^\alpha([0,1],\R^n)$ the \emph{Young integral}
\begin{equ}
h_t=\int_0^t f_s\,dg_s
\end{equ}
is well-defined as the limit as $m\to\infty$ of the Riemann sums
\begin{equ}
\sum_{i=0}^{m} f_{t_i^m}\cdot(g_{t_{i+1}^m\wedge t}-g_{t_i^m\wedge t}),
\end{equ}
where $(\pi_m)_{m\in\N}=\big(\{0=t_0^m\leq \cdots\leq t_{m+1}^m=1\}\big)_{m\in\N}$ is any sequence of partitions such that $|\pi_m|:=\max_{i}|t_{i+1}^m-t_i^m|\to 0$.
For some $N$ depending only on $\alpha+\beta$ the Young integral satisfies the estimate
\begin{equ}\label{eq:Young remainder}
|h_t-h_s-f_s(g_t-g_s)|\leq N |t-s|^{\alpha+\beta}[f]_{\cC^\beta([s,t])}[g]_{\cC^\alpha([s,t])}.
\end{equ}
for all $(s,t)\in[0,1]^2_\leq$.
This of course implies that $h\in \cC^\alpha([0,1],\R)$ as well as the bound
\begin{equ}\label{eq:Young continuity}
\,[h]_{\cC^\alpha([s,t])}\leq N \|f\|_{\cC^\beta([s,t])}[g]_{\cC^\alpha([s,t])}.
\end{equ}
For $s=0$ on the left-hand side of \eqref{eq:Young continuity} one may of course take norm in place of seminorm.
For $n_0>1$, the extension to $\R^{n_0\times n}$-valued integrands, resulting in $\R^{n_0}$-valued  integrals, is straightforward.

Let $\alpha\in(1/3,1/2]$ and $\gamma\in(\alpha,2\alpha]$ be such that $\alpha+\gamma>1$.
Let $(g,\bg)\in \cR^\alpha([0,1],\R^n)$ and $(f,f')\in \cD^\gamma_g([0,1],\R^n)$.
Then the \emph{rough integral}
\begin{equ}
h_t=\int_0^t (f,f')_s\,d(g,\bg)_s
\end{equ}
is well-defined as the limit as $m\to\infty$ of the Riemann sums
\begin{equ}
\sum_{i=1}^m \Big(f_{t_i^m}\cdot(g_{t_{i+1}^m\wedge t}-g_{t_i^m\wedge t})+f'_{t_i^m}\cdot\bg_{t_i^m\wedge t,t_{i+1}^m\wedge t}\Big),
\end{equ}
where $(\pi_m)_{m\in\N}=\big(\{0=t_0^m\leq \cdots\leq t_{m+1}^m=1\}\big)_{m\in\N}$ is any sequence of partitions such that $|\pi_m|\to 0$.
The product of the matrices $f',\bg\in\R^{n\times n}$ is understood to be the Frobenius product.
When both $f'$ and $\bg$ are clear from the context, we will usually write by some abuse of notation
\begin{equ}
h_t=\int_0^t f_s\,dg_s.
\end{equ}
Note that in the trivial case $(f,0)$ the rough and Young integrals coincide.
The rough analogue of the remainder estimate \eqref{eq:Young remainder}
reads as
\begin{equ}\label{eq:rough remainder}
|h_t-h_s-f_s(g_t-g_s)-f'_s\cdot\bg_{s,t}|\leq N |t-s|^{\alpha+\gamma}[(f,f')]_{\cD^{\gamma}_g([s,t])}[(g,\bg)]_{\cR^\alpha([s,t])}
\end{equ}
for all $(s,t)\in[0,1]_\leq^2$,
where $N$ is a constant depending only on $\alpha+\gamma$.
This of course implies that $(h,f)\in \cD^{2\alpha}_g([0,1],\R)$ as well as the bound
\begin{equ}\label{eq:rough continuity}
\,[(h,f)]_{ \cD^{2\alpha}_g([s,t])}\leq N [(f,f')]_{\cD^{\gamma}_g([s,t])}[(g,\bg)]_{\cR^\alpha([s,t])}.
\end{equ}
As before, for $n_0>1$, the extension to $\R^{n_0\times n}$-valued integrands, resulting in $\R^{n_0}$-valued  integrals, is straightforward.


\emph{Conventions}. Whenever for some $s_0\in[0,1)$, a function $X$ is given on the interval $[s_0,1]$, it is understood to be extended as a constant function $f_r\equiv f_{s_0}$ for $r\in[0,s_0]$. On finite dimensional vector spaces we always use the Euclidean norm\footnote{This is not always the optimal choice, especially for matrices. Using their operator norm would allow some intermediate estimates to be dimension-independent, but we do not pursue these subtleties.}.
When a statement contains an estimate with a constant $N$ depending on a certain set of parameters, in the proof we do not carry the constants from line to line. Rather, we write $a\lesssim b$ to denote the existence of a constant $N'$ depending on the same set of parameters such that $a\leq N' b$.

\subsection{Main results}

Our assumptions in the three regimes are as follows.

\begin{assumption}[$\mathbf{A}_{\mathrm{smooth}}$]
For some $M\geq0$, $\lambda>0$, and $\alpha>1-1/(2H)$, one has $\|b\|_{\cC^\alpha}\leq M$, $\|\sigma\|_{\cC^1}\leq M$, and $\sigma\sigma^\top\succeq\lambda\I$.
\end{assumption}

\begin{assumption}[$\mathbf{A}_{\mathrm{Young}}$]
For some $M\geq0$, $\lambda>0$, and $\alpha>1-1/(2H)$, one has $\|b\|_{\cC^\alpha}\leq M$, $\|\sigma\|_{\cC^2}\leq M$, and $\sigma\sigma^\top\succeq\lambda\I$.
\end{assumption}
\begin{assumption}[$\mathbf{A}_{\mathrm{rough}}$]
(i) For some $M\geq0$, $\lambda>0$, and $\alpha>0$, one has $\|b\|_{\cC^\alpha}\leq M$, $\|\sigma\|_{\cC^3}\leq M$, and $\sigma\sigma^\top\succeq\lambda\I$.

(ii) For all $(s,t)\in[0,1]_\leq^2$, $\bB_{s,t}^H$ is $\cF_t$-measurable and the event $\Omega_H=\{\forall \eps>0:\,(B^H,\bB^H)\in \cR^{H-\eps}([0,1])\}$ has probability $1$.
\end{assumption}

For $H>1/2$ we shall also use the event $\Omega_H=\{\forall\eps>0:\,B^H\in \cC^{H-\eps}([0,1])\}$, but in this case one has $\P(\Omega_H)=1$ trivially.
For $H\in(1/3,1/2]$ an example of a rough path lift of $B^H$ that satisfies $\mathbf{A}_{\mathrm{rough}}$ (ii) is the Gaussian rough path constructed in e.g. \cite[Thm.~10.4]{Friz-Hairer}.
For $H=1/2$ this construction gives the Stratonovich lift, another example satisfying $\mathbf{A}_{\mathrm{rough}}$ (ii) would be the It\^o lift.

With these preparations we can define what we mean by a solution. Note that the conditions below in particular imply that all integrals are pathwise well-defined.
\begin{definition}\label{def:soln}
\begin{enumerate}[(i)]
\item Let $H\in(1,\infty)\setminus \N$ and
let $\mathbf{A}_{\mathrm{smooth}}$ hold.
Given $\omega\in\Omega_H$, $s_0\in[0,1)$, and $x_0\in\R^d$, we say that a function $Y:[s_0,1]\to\R^d$ is a \emph{solution} to \eqref{eq:main} if
$Y\in \cC^\beta([s_0,1])$ for some $\beta>0$ and it satisfies the equality \eqref{eq:main} for all $t\in[s_0,1]$.

\item Let $H\in(1/2,1)$ and
let $\mathbf{A}_{\mathrm{Young}}$ hold.
Given $\omega\in\Omega_H$, $s_0\in[0,1)$, and $x_0\in\R^d$, we say that a function $Y:[s_0,1]\to\R^d$ is a \emph{solution} to \eqref{eq:main} if
$Y\in \cC^\beta([s_0,1])$ for some $\beta>1-H$ and it satisfies the equality \eqref{eq:main} for all $t\in[s_0,1]$.

\item Let $H\in(1/3,1/2]$ and
let $\mathbf{A}_{\mathrm{rough}}$ hold.
Given $\omega\in\Omega_H$, $s_0\in[0,1)$, and $x_0\in\R^d$, we say that a function $Y:[s_0,1]\to\R^d$ is a \emph{solution} to \eqref{eq:main} if
$\big(Y,\sigma(Y)\big)\in \cD^\gamma_{B^H(\omega)}([s_0,1])$ for some $\gamma>1-H$ and it satisfies the equality \eqref{eq:main} for all $t\in[s_0,1]$.

\item In all three cases, given $s_0\in[0,1)$ and $x_0\in\R^d$, we say that a stochastic process $(X_t)_{t\in[s_0,1]}$ is a \emph{strong solution} to \eqref{eq:main} if it is adapted to the filtration $\bF$ and for almost all $\omega\in\Omega_H$, the function $X(\omega):[s_0,1]\to\R^d$ is a solution to \eqref{eq:main}.
\end{enumerate}

\end{definition}

The main results of the article on the strong well-posedness of \eqref{eq:main} are summarised in the following theorem.
\begin{theorem}\label{thm:main}
Assume  $H\in(1,\infty)\setminus \N$ and $\mathbf{A}_{\mathrm{smooth}}$, $H\in(1/2,1)$ and $\mathbf{A}_{\mathrm{Young}}$, or  $H\in(1/3,1/2]$ and $\mathbf{A}_{\mathrm{rough}}$.
Then for any $s_0\in[0,1)$, $x_0\in\R^d$ there exists a strong solution to \eqref{eq:main}.
Moreover,
there exists an event $\what\Omega\subset\Omega_H$ of full probability such that for any $\omega\in\what\Omega$, $s_0\in[0,1)$, $x_0\in\R^d$, any two solutions to \eqref{eq:main} coincide.
\end{theorem}
 The various cases are stated and proved in Theorems \ref{thm:exist-Young}, \ref{thm:PBP-uniq-Young}, \ref{thm:exist-rough}, \ref{PBP-uniq-rough}, and \ref{thm:everythin-smooth}.
The overall strategy consists of the following steps. First we establish strong uniqueness and a little bit more: stability with respect to certain data of the equation. The latter can then be used to establish strong existence and a little bit more: existence of a solution semiflow. Finally, the latter can be used to prove path-by-path uniqueness.
To avoid repeating arguments, we give all details in the Young case and refer back in the other two cases for some of the steps that do not need any significant change.

\begin{remark}
In the special case $H=1/2$ and $\bB^{1/2}=\bB^{\Ito}$ our notion of strong solution is somewhat different from the standard notion for It\^o stochastic differential equations. Definition \ref{def:soln} requires more from the solution $Y$: it needs to be controlled on an event $\Omega_Y\subset\Omega_{1/2}$ of full probability, but it also gives more: the stochastic integral is defined on the same event $\Omega_Y$, instead of an unspecified event of full probability.
\end{remark}

\begin{remark}
Also in the $H=1/2$ case, with the help of It\^o calculus the condition $b\in \cC^\alpha$, $\alpha>0$, can be relaxed to for example bounded measurable $b$ to get either strong uniqueness with multiplicative noise \cite{Veretennikov} or path-by-path uniqueness with additive noise \cite{Davie}.
Path-by-path uniqueness with multiplicative noise is a challenging endpoint case of our analysis. In this direction critical variants of the stochastic sewing lemma \cite{FHL} could be useful.
\end{remark}

\begin{remark}
One would expect to be able to lower the regularity requirements on $\sigma$ in $\mathbf{A}_{\mathrm{Young}}$ and $\mathbf{A}_{\mathrm{rough}}$ to $1/H+\eps$ by standard arguments, which for sake of easing presentation we do not pursue.
A more interesting question regarding $\sigma$ is how to further reduce its regularity by stochastic arguments, this is studied in the very recent work \cite{Toyomu-Nicolas}.
\end{remark}

\begin{remark}
Continuing the discussion on $\sigma$, weakening the uniform ellipticity condition to hypoellipticity condition seems highly nontrivial but not impossible, given that densities of such equations (with regular drift) are well-studied (see among others \cite{HN07, BH07, CF10, CHLT15}). In the Brownian case this has been recently achieved in \cite{CHM}.
Note however that the $H>1$ case can be seen as a special form of degenerate equation: we can equivalently rewrite \eqref{eq:main} as
\begin{equs}[eq:system]
dY^1_t&=dB^{H-\floor{H}}_t\\
dY^2_t&=Y^1_t\,dt\\
&\vdots\\
dX_t&=\big(b(X_t)+\sigma(X_t)Y^{\floor H}_t\big)\,dt.
\end{equs}
\end{remark}

The reader familiar with the regularisation by noise literature and in particular with \cite{CG16} notices that in the regime $H\in(1/3,1/2)$ we do not quite recover the condition $\alpha>1-1/(2H)$ that is known to be sufficient for strong well-posedness in the additive noise case.
The naive guess for this obstacle
would be that for $\alpha<0$ the drift coefficient $b$ is in general a distribution and the meaning of $b(X_s)$ is unclear.
However, this is not the our main obstacle and in a weaker sense we are able to handle distributional drift in  the rough case. This is the content of Theorem \ref{thm:weak}, which is proved in Section \ref{sec:weak}.
The more concrete reason for the threshold $\alpha>0$ arising in the rough case of Theorem \ref{thm:main} is discussed in Remark \ref{rem:obstacle}.
\begin{theorem}\label{thm:weak}
Assume $H\in(1/3,1/2]$, $\alpha>1/2-1/(2H)$, $b\in \cC^\alpha$, $\sigma\in \cC^2$, $s_0\in[0,1)$, and $x_0\in\R^d$.
Then there exists a probability space $(\bar\Omega,\bar\cF,\bar \P)$ with a complete filtration $\bar\bF=(\bar\cF_t)_{t\in\R}$ carrying a two-sided $d_0$-dimensional $\bar\bF$-Wiener process $\bar W$, such that $\bar B^H$ defined by the Mandelbrot -- van Ness representation based on $\bar W$ has a rough path lift $(\bar B^H,\bar \bB^H)$ satisfying $\mathbf{A}_{\mathrm{rough}}$ (ii), and such that there exists an $\bar\bF$-adapted stochastic processes $(\bar X_t)_{t\in[s_0,1]}$ and $(\bar D_t)_{t\in[s_0,1]}$ such that 
\begin{itemize}
\item $\bar \P$-almost surely $(\bar X,\sigma(\bar X))\in \cD^\gamma_{B^H}$ for some $\gamma>1-H$;
\item $\bar \P$-almost surely for all $t\in[s_0,1]$ it holds that
\begin{equ}
\bar X_t=x_0+\bar D_t+\int_{s_0}^t\sigma(\bar X_s)\,d\bar B^H_s,
\end{equ}
with the integral understood in a rough sense;
\item For any sequence $(b^n)_{n\in \N}$ of smooth functions converging to $b$ in $\cC^\alpha$,
one has $\bar \P$-almost surely for all $t\in[s_0,1]$
\begin{equ}
\bar D_t=\lim_{n\to\infty}\int_{s_0}^tb^n(\bar X_s)\,ds.
\end{equ}
\end{itemize}
\end{theorem}

\section{The Young case}\label{sec:Young}
Throughout the section we fix $H\in(1/2,1)$ and $\alpha\in(1-1/(2H),1)$. Since lowering $\alpha$ is never a loss of generality, we will also assume $\alpha<2-1/H$.
Furhermore, we fix exponents
\begin{equ}
1/2<H_-<H^-<H,
\end{equ}
such that
$H_-+H^-> 1+\alpha H$, $1+\alpha H_-+(\alpha-2) H>0$, and $H_--H+\alpha H>0$, all of which are clearly possible\footnote{One could of course give explicit expressions to write $H_-$ and $H^-$ as functions of $\alpha$ and $H$, so in the sequel we do not consider them as additional parameters.} under the above conditions on $\alpha$ and $H$.

\subsection{A priori bound}
First we observe that appropriate stopping times based on the driving process control the seminorm of all solutions simultaneously.
For all $K \in\N$, define the stopping time 
\begin{equs} 
\tau_K= \inf\{ t >0 : [B^H]_{\cC^{H^-}([0, t])}  \geq  K^{H^--H_-}\}.
\end{equs}
Note that $t\mapsto [B^H]_{\cC^{H^-}([0, t])}$ is adapted and almost surely continuous, so $\tau_K$ is indeed a stopping time. Moreover, for all $\omega\in\Omega_H$ there is a $K_0(\omega)$ such that $\tau_K=1$ for all $K\geq K_0$.
\begin{lemma}                                                                         \label{lem:apriori-K}
Let $s_0\in[0,1)$, $x_0\in\R^d$, $b\in \cC^0$, and $\sigma\in \cC^1$.
Then there exists a constant $N=N(\|b\|_{\cC^0},\|\sigma\|_{\cC^1},\alpha,H,d,d_0)$ such that if $X$ is a solution of \eqref{eq:main}, then almost surely
\begin{equ}
\,[X_{ \cdot \wedge \tau_K} ]_{\cC^{H_-}([s_0, 1])} \leq \big(2\| b\|_{\cC^0} +1\big)K. 
\end{equ}
\end{lemma}

\begin{proof}
Let $s_0 \leq U <T \leq 1$. For $(s, t) \in [U,T]_{\leq
}^2$, we have 
\begin{equs}
|X_{t \wedge \tau_K}-X_{s \wedge \tau_K}| \leq \big|D_{t \wedge \tau_K} -D_{s\wedge \tau_K}| +   \big|S_{t \wedge \tau_K} -S_{s \wedge \tau_K}|,
\end{equs}
where $D$ and $S$ are the deterministic and stochastic (or rather, Young) integrals in \eqref{eq:main}, respectively.
Obviously the first term on the right-hand side can be estimated by $\| b\|_{\cC^0} |t-s|$, while for the second one we have from \eqref{eq:Young continuity}
\begin{equs}
 \big|S_{t \wedge \tau_K} -S_{s \wedge \tau_K}|  & \leq N_0 \|  \sigma (X_{\cdot \wedge \tau_k}) \|_{\cC^{H_-}([U, T])} [ B^H_{\cdot \wedge \tau_K}]_{\cC^{H^-}([0 , \tau_K])} |t-s|^{H^-}
\\
& \leq  \big(\| \sigma\|_{\cC^0}  +  \| \sigma\|_{\cC^1} [X_{\cdot \wedge \tau_K}]_{\cC^{H_-}([U, T])}\big)N_0 K^{H^--H_-}|t-s|^{H^-},
\end{equs}
where $N_0=N_0(\alpha,H,d,d_0)$ is the constant from Young integration.
Upon dividing by $|t-s|^{H_-}$ and taking supremum over $s, t$, we get 
\begin{equs}
\, [X_{\cdot \wedge \tau_K}]_{\cC^{H_-}([U, T])} \leq \| b
\|_{\cC^0}  +  \big(\| \sigma\|_{\cC^0}  +  \| \sigma\|_{\cC^1} [X_{\cdot \wedge \tau_K}]_{\cC^{H_-}([U, T])}\big)N_0 K^{H^--H_-} |T-U|^{H^- - H_-}.
\end{equs}
If $|T-U|\leq K^{-1}(2\|\sigma\|_{\cC^1} N_0)^{-1/(H^--H_-)}$, then the inequality buckles and we get 
\begin{equs}
\, [X_{\cdot \wedge \tau_K}]_{\cC^{H_-}([U, T])} \leq 2\| b\|_{\cC^0} +1.
\end{equs}
By covering the $[s_0,1]$ by intervals of size $K^{-1}(2\|\sigma\|_{\cC^1} N_0)^{-1/(H^--H_-)}$ and using the subadditivity of H\"older norm along union of intervals we get the claim.
\end{proof}

\subsection{Strong uniqueness and stability}\label{sec:stabil-Young}
Throughout Section \ref{sec:stabil-Young} we fix $K\in\N$ and use the shorthand $\tau=\tau_K$.
Moreover, if $X$ is a solution of \eqref{eq:main} we define the process $\tilde X$ for $t\in[s_0,1]$ as 
\begin{equs}\label{eq:tildeX}
\tilde{X}_t= x_0+\int_{s_0}^t b(X_{s\wedge \tau})\,ds+\int_{s_0}^t \sigma(X_{s \wedge \tau})\,dB^H_s.
\end{equs}

\begin{remark}
Let us point out that although the stopping time is based on the driving noise, we will only stop the solutions and not the noise itself, in order to maintain the regularising effects.
\end{remark}

Recall that by convention we extend processes before their initial times as constants, i.e. $X_r\equiv x_0$ for $r\in[0,s_0]$, and the same for $\tilde X$.
Finally, set the shorthand
\begin{equ}
\,[B^H]_{\cC^{H^-}|\bF}= [(s, t) \mapsto \E^s B_t]_{C^{H^-}([0, 1]^2)}.
\end{equ}
For the above, notice that $\E^sB^H_t$ has a version which is continuous  in $(s, t)$.  Indeed,  by \eqref{eq:mandelbrot} it follows that for all $s,t$, we have with probability one
\begin{equs}
\E^s B^H_t = B^H_t-\int_{s\wedge t} ^t (t-u)^{H-1/2} \, dW_u. 
\end{equs}

\begin{lemma}             \label{lem:Regularisation}
Let  $p \geq 2$, $s_0\in[0,1)$, $x_0\in\R^d$, $(S,T)\in[s_0,1]^2_\leq$, and let $b$ and $\sigma$ satisfy $\mathbf{A}_{\mathrm{Young}}$. Let  $X$ be a strong solution of \eqref{eq:main}.
Then there exists constants $N= N(\alpha, H, d, d_0, M,\lambda, p, K)$ and $\eps=\eps(\alpha,H)>0$ such that for all $f \in \cC^\alpha$, all adapted stochastic processes $Z$, and all $(s,t)\in[S,T]^2_\leq$, the  following bound holds
\begin{equs}
\Big\| \int_s^t & \big( f(\tilde{X}_r + Z_r) -f(\tilde{X}_r)\big) \, dr \Big\|_{L_p(\Omega)}
   \leq N  \| f\|_{\cC^\alpha}\| Z\|_{\scC^0_p([S,T])} |t-s|^{1/2+\eps}
\\
&\qquad\qquad+ N \|f\|_{\cC^\alpha}\Big(\big\|(1+ [B^H]_{\cC^{H^-}|\bF})Z\big\|_{\scC^0_p([S,T])}
+[Z]_{\scC_p^{1/2}([S,T])}\Big)|t-s|^{1+\eps}.\label{eq:one-big-bound}
\end{equs}
\end{lemma}

\begin{proof}
We may and will assume $\|f\|_{\cC^\alpha}=1$.
For $S\leq s\leq u\leq r\leq t\leq T$, set
\begin{equs}
\Theta_{s, r} & = \int_{s_0}^sb(X_{v  \wedge \tau})  \, d v + b(X_{s \wedge \tau})(r-s),
\\
\Sigma_s &=\int_{s_0}^s \sigma(X_{v\wedge\tau})\,dB^H_v,
\\
\Xi_{s,r}  &= \Theta_{s, r} +\Sigma_s+ \sigma(X_{s\wedge \tau}) ( B^H_r- B^H_s),
\\
 \hat {\Xi}_{s,u, r} &= \Theta_{s, r}  +\Sigma_s+ \sigma(X_{s\wedge \tau}) ( \E^u B^H_r- \ B^H_s).
\end{equs}
We wish to apply stochastic sewing in the form of Lemma \ref{lem:SSL}.
Introducing the shorthand $s_-=s-(t-s)$, recall from Lemma \ref{lem:SSL} the notation $\overline{[S,T]}_\leq^2=\{(s,t)\in[S,T]^2:\,s_-\geq S\}$.
We wish to verify the conditions of the lemma with the processes
\begin{equs}
A_{s, t}&= \E^{s_-} \int_s^t f(\Xi_{s_-, r} +Z_{s_-}) -f(\Xi_{s_-, r}) \, dr,
\\
\cA_t&=\int_S^t \big(f(\tilde{X}_r + Z_r) -f(\tilde{X}_r)\big) \, dr.
\end{equs}
Note that with the notation 
$$
\Gamma(s,u, r)= c(H)(r-u)^{2H} \sigma \sigma^\top (X_{s\wedge \tau}),
$$ 
we have by \eqref{eq:very-basic2} the identity $\E^u f(\Xi_{s,r})=\cP_{\Gamma(s,u,r)}f(\hat\Xi_{s,u,r})$.
Therefore by \eqref{eq:C1}, \eqref{eq:HKKK}, and the lower bound assumed on $\sigma\sigma^\top$, we have almost surely
\begin{equs}
|A_{s, t}|& = \Big|\int_s^t \cP_{\Gamma(s_-,s_-, r)}f (\hat{\Xi}_{s_-,s_-,  r}+Z_{s_-} )- \cP_{\Gamma(s_-,s_-, r)}f (\hat{\Xi}_{s_-,s_-, r}) \, dr\Big| 
\\
&\leq |Z_{s_-}|\int_s^t[\cP_{\Gamma(s_-,s_-, r)}f]_{\cC^1}\,dr 
\\
& \lesssim |Z_{s_-}| \int_s^t (r-s)^{(\alpha-1)H}   \, dr. 
\end{equs}
After taking $L_p(\Omega)$ norms, we get
\begin{equs}
\| A_{s, t} \|_{L_p(\Omega)} \lesssim \| Z\|_{\scC^0_p([S,T])} (t-s)^{1+(\alpha-1)H}.
\end{equs}
Since $1+(\alpha-1)H>1/2$, condition \eqref{eq:SSL-cond1} is satisfied with $N_1=N  \| Z\|_{\scC^0_p([S,T])}$.

Moving on to  \eqref{eq:SSL-cond2}, recall that therein we use the notation $u=(s+t)/2$. 
We can therefore write
\begin{equs}
\E^{s_-} \delta A_{s, u, t}&= \E^{s_-} \int_{s}^u f(\Xi_{s_-,r}+Z_{s_-})-f(\Xi_{s_-,r})-f(\Xi_{s-(u-s),r}+Z_{s-(u-s)})+f(\Xi_{s-(u-s),r})\,dr
\\
&\quad + \E^{s_-} \int_{u}^t f(\Xi_{s_-,r}+Z_{s_-})-f(\Xi_{s_-,r})-f(\Xi_{s,r}+Z_{s})+f(\Xi_{s,r})\,dr=:I^1+I^2.
\end{equs}
The two terms $I^1$ and $I^2$ are treated in a virtually identical way, so we only detail the latter.
By the tower property of conditional expectations and conditional Jensen's inequality we can write
\begin{equ}
\|I^2\|_{L_p(\Omega)}\leq \int_u^t \|J_{r}\|_{L_p(\Omega)}\,dr,
\end{equ}
where
\begin{equs}
J_r&=  \E^s f(\Xi_{s_-, r} +Z_{s_-}) -\E^s f(\Xi_{s_-, r}) \, dr - \E^s f(\Xi_{s, r} +Z_s) +\E^s f(\Xi_{s, r})
\\
&=   \cP_{\Gamma(s_-,s, r)}f (\hat{\Xi}_{s_-,s, r}+Z_{s_-} )- \cP_{\Gamma(s_-,s, r)}f (\hat{\Xi}_{s_-,s, r})
\\&\qquad -  \cP_{\Gamma(s,s, r)}f (\hat{\Xi}_{s,s, r}+Z_s)+ \cP_{\Gamma(s,s, r)}f (\hat{\Xi}_{s,s, r}).
\end{equs}
Therefore we can further decompose as $J_r=J^1_r+J^2_r$, with
\begin{equs}
J^1_{r}&=   \cP_{\Gamma(s_-,s, r)}f (\hat{\Xi}_{s_-,s, r}+Z_{s_-} ) - \cP_{\Gamma(s_-,s, r)}f (\hat{\Xi}_{s_-,s, r})
\\
&\qquad -   \cP_{\Gamma(s_-,s, r)}f (\hat{\Xi}_{s,s, r}+Z_s)+ \cP_{\Gamma(s_-,s, r)}f (\hat{\Xi}_{s,s, r}),\label{eq:def-J1}
\end{equs}
and 
\begin{equs}
J^2_r&=   \cP_{\Gamma(s_-,s, r)}f (\hat{\Xi}_{s,s, r}+Z_s)-  \cP_{\Gamma(s,s, r)}f (\hat{\Xi}_{s,s, r}+Z_s)
\\
&\qquad -\cP_{\Gamma(s_-,s, r)}f (\hat{\Xi}_{s,s, r})+\cP_{\Gamma(s,s, r)}f (\hat{\Xi}_{s,s, r}).\label{eq:def-J2}
\end{equs}
Concerning $J^1$, by \eqref{eq:C2} and \eqref{eq:HKKK} we have
\begin{equs}        
|J^1_r|&\leq|Z_{s-}-Z_s|[\cP_{\Gamma(s_-,s, r)}f]_{\cC^1}+|Z_s||\hat{\Xi}_{s_-,s, r}-\hat{\Xi}_{s,s, r}|[\cP_{\Gamma(s_-,s, r)}f]_{\cC^2}
\\
&\lesssim
|Z_{s-}-Z_s|(t-s)^{(\alpha-1)H}+|Z_s||\hat{\Xi}_{s_-,s, r}-\hat{\Xi}_{s,s, r}|(t-s)^{(\alpha-2)H}.\label{eq:est-F1}
\end{equs}
where\footnote{This is the point where the shifted variant of the stochastic sewing lemma is used, otherwise the power $(\alpha-2)H$ would need to be integrated, which would result in the ``wrong'' condition $\alpha>2-1/H$.} we used that $r-s\geq (t-s)/2$ for $r\geq u$.
Next, notice that 
\begin{equs}
|\hat{\Xi}_{s_-,s, r}-\hat{\Xi}_{s,s, r}| & \leq | \Theta_{s_-,r} - \Theta_{s, r} | 
\\
&\quad +\big| \Sigma_{s_-}+ \sigma(X_{s_-\wedge \tau}) ( \E^s B^H_r- \ B^H_{s_-}) -\Sigma_s- \sigma(X_{s\wedge \tau}) ( \E^s B^H_r- \ B^H_s)\big|
\\
& \leq  | \Theta_{s_-,r} - \Theta_{s, r} | 
+\big| \Sigma_{s_-} -\Sigma_s+ \sigma(X_{s_-\wedge \tau}) ( B^H_s- B^H_{s_-})\big|
\\
 & \quad+\big| \E^s \big( ( \sigma(X_{s_-\wedge \tau})  - \sigma(X_{s\wedge \tau}) ( B^H_r-  B^H_s)\big) \big|.
\end{equs}
For the first term recall that by Lemma \ref{lem:apriori-K} we have $[X_{ \cdot \wedge \tau} ]_{C^{H_-}([0, 1])} \lesssim 1$ and therefore
\begin{equs}
| \Theta_{s,r} -\Theta_{s_-,  r} | & =  \Big| \int_{s_-}^s\big( b(X_{v\wedge\tau})-b(X_{s_-\wedge\tau})\big) \, dv + \big( b(X_{s\wedge\tau})-b(X_{s_-\wedge\tau})\big)(r-s) \,\Big|\\&\lesssim (t-s)^{1+\alpha H_-}.
\end{equs}
For the second term, by  Young integration we have
\begin{equs}
| \Sigma_{s_-} -\Sigma_s+ \sigma(X_{s_-\wedge \tau}) ( B^H_s- B^H_{s_-})|  
& \lesssim [ \sigma(  X_{\cdot \wedge \tau}   )]_{\cC^{H_-}} [ B^H]_{\cC^{H^-} | \bF} (s-s_-)^{H_-+H^-}
\\
& \lesssim  [ B^H]_{\cC^{H^-}|\bF} (t-s)^{H_-+H^-}.
\end{equs}
Finally, for the third term we have 
\begin{equs}
\big| \E^s \big( ( \sigma(X_{s_-\wedge \tau})  - \sigma(X_{s\wedge \tau}) ( B^H_r-  B^H_s)\big) \big| & \lesssim
 [ B^H]_{\cC^{H^-}|\bF}  (t-s)^{H_-+H^-}.
\end{equs}
Recalling that $H_-+H^-\geq 1+\alpha H_-$, we can conclude
\begin{equs}
\|J^1_r\|_{L_p(\Omega)}&\lesssim \big\|(1+ [ B^H]_{\cC^{H^-}| \bF})Z\big\|_{\scC^0_p([S,T])}(t-s)^{1+\alpha H_-+(\alpha-2) H}
\\
&\qquad
+[Z]_{\scC_p^{1/2}([S,T])}(t-s)^{1/2+(\alpha-1)H}.
\end{equs}
For $J^2$, by \eqref{eq:C1} and Propositions \ref{prop:HK-Gamma} and \ref{prop:HK-Gammadiff} we have
\begin{equs}
|J^2_r| &\leq    |Z_s |  \big[ \cP_{\Gamma(s_-,s, r)}f- \cP_{\Gamma(s,s, r)}f\big]_{C^1} 
\\
&\lesssim  |Z_s | \big|\Gamma(s_-,s, r)- \Gamma(s,s, r)\big|\big( \big|\Gamma(s_-,s, r)^{-1}\big| \vee \big|\Gamma(s,s, r)^{-1}\big|\big)^{3/2-\alpha/2 } 
\\
 &\lesssim |Z_s | |X_{s_- \wedge \tau} - X_{s \wedge \tau}| (r-s)^{2H} (r-s)^{-3H+ \alpha H}
 \\
&  \lesssim   |Z_s |  (t-s)^{H_--H+\alpha H},\label{eq:bounding-J2}
\end{equs}
and therefore
\begin{equ}
\|J^2_r\|_{L_p(\Omega)}\lesssim\|Z\|_{\scC^0_p([S,T])}(t-s)^{H_--H+\alpha H}.
\end{equ}
By assumption, all of $1+\alpha H_-+(\alpha-2) H$, $1/2+(\alpha-1)H$, and $H_--H+\alpha H$ are positive, so after integration we get the bound
\begin{equs}
\| \E^s \delta A_{s, u, t} \|_{L_p(\Omega)}
\lesssim \Big(\big\|(1+ [ B^H]_{\cC^{H^-}|\bF})Z\big\|_{\scC^0_p([S,T])}
+[Z]_{\scC_p^{1/2}([S,T])}\Big)(t-s)^{1+\eps_2}
\end{equs}
for some $\eps_2>0$. Therefore condition \eqref{eq:SSL-cond2} is satisfied with $N_2=N\Big(\big\|(1+ [ B^H]_{\cC^{H^-}|\bF})Z\big\|_{\scC^0_p([S,T])}
+[Z]_{\scC_p^{1/2}([S,T])}\Big)$.

It remains to verify \eqref{eq:SSL-cond3}.
Introduce the auxiliary quantity
\begin{equ}
\hat A_{s,t}=(t-s)\big(f(\tilde X_{s_-}+Z_{s_-})-f(\tilde X_{s_-})\big).
\end{equ}
We then trivially have
\begin{equ}
\|\cA_t-\cA_s-\hat A_{s,t}\|_{L_p(\Omega)}\leq\int_s^t\big\|f(\tilde X_r+Z_r)-f(\tilde X_{r})-f(\tilde X_{s_-}+Z_{s_-})+f(\tilde X_{s_-})\big\|_{L_p(\Omega)}\,dr.
\end{equ}
Since $[\tilde X]_{\scC^{H_-}_p}\leq [\tilde X]_{L_p(\Omega, C^{H_-})}\lesssim 1$ and
$[Z]_{\scC^{1/2}_p}<\infty$ (otherwise the right-hand-side of \eqref{eq:one-big-bound} is infinite and the claim is trivial),
by the H\"older continuity of $f$ we get
\begin{equ}\label{eq:identify1-Young}
\|\cA_t-\cA_s-\hat A_{s,t}\|_{L_p(\Omega)}\leq N_3(t-s)^{1+\alpha/2}
\end{equ}
with some constant $N_3\geq0$.
Second, since $\hat A_{s,t}$ is $\cF_{s_-}$-measurable, by conditional Jensen's inequality we can write
\begin{equ}
\|\hat A_{s,t}-A_{s,t}\|_{L_p(\Omega)}\leq\int_s^t\big\|f(\Xi_{s_-, r} +Z_{s_-}) -f(\Xi_{s_-, r})-f(\tilde X_{s_-}+Z_{s_-})+f(\tilde X_{s_-})\big\|_{L_p(\Omega)}\,dr.
\end{equ}
From the trivial bound $\|\Xi_{s,r}-\tilde X_s\|_{L_p(\Omega)}\lesssim |r-s|^H$ we get
\begin{equ}\label{eq:identify2-Young}
\|\hat A_{s,t}-A_{s,t}\|_{L_p(\Omega)}\lesssim (t-s)^{1+\alpha H}.
\end{equ}
By \eqref{eq:identify1-Young}-\eqref{eq:identify2-Young} we get \eqref{eq:SSL-cond3}. Therefore all conditions of Lemma \ref{lem:SSL} are satisfied and \eqref{eq:one-big-bound} follows from \eqref{eq:SSL-conc}.
\end{proof}

As is often the case with uniqueness statements, it is useful to prove a stronger statement including some form of stability of the solutions with respect to changing certain data.
Therefore we consider the auxiliary equation 
\begin{equs} \label{eq:aux}
Y_t=y_0+\int_{s_0}^t \tilde b (Y_s) \, ds + \int_{s_0}^t \sigma (Y_s) \, dB^H_s. 
\end{equs}

\begin{lemma}                              \label{lem:stability}
Let  $p \geq 1$, $s_0\in[0,1)$, $x_0, y_0 \in\R^d$. Let $b$, $\tilde b$, and $\sigma$ satisfy $\mathbf{A}_{\mathrm{Young}}$.
Then there exists a constant $N_1$, depending only on  $\alpha, H, d, d_0, M, \lambda, p,$ and $ K$, as well as a constant $\gamma=\gamma(\alpha,H)\in(0,2)$,  with the following property:
 for any  $X$ and $Y$ strong solutions of  \eqref{eq:main} and \eqref{eq:aux}, respectively, and for all $C\geq 1$,  the following estimate holds
\begin{equs}
\| X_{\cdot \wedge\tau }- Y_{\cdot \wedge\tau }\|_{\scC^{1/2}_p([s_0,1])}  \leq N_1 ^{C^{2-\gamma}} \Big( |x_0-y_0|+  \big( \P ([ B^H ]_{\cC^{H^-}|\bF}\geq C ) \big)^{1/2p} +\|b-\tilde b\|_{C^0(\R^d)} \Big). 
\end{equs}
\end{lemma}

\begin{proof}
Note that one may (and we will) assume $p$ to be sufficiently large.
Take $(U,T)\in[s_0,1]^2_{\leq}$. Denoting $Z=X-Y$, one has 
\begin{equs}
 \| Z_{\cdot \wedge \tau} \|_{\scC^{1/2}_p([U,T])} & \leq \| Z_{S \wedge \tau} \|_{L_p(\Omega)}+  2[ Z_{\cdot \wedge \tau} ]_{\scC^{1/2}_p([U,T])}
 \\
 &\leq \| Z_{S \wedge \tau} \|_{L_p(\Omega)}+2 [ D_{\cdot \wedge \tau} ]_{\scC^{1/2}_p([U,T])}
 \\
 & \qquad + 2 [ S_{\cdot \wedge \tau} ]_{\scC^{1/2}_p([U,T])}      + 2 [ E_{\cdot \wedge \tau} ]_{\scC^{1/2}_p([U,T])}             \label{eq:ready-to-buckle}
\end{equs}
where 
\begin{equs}
D_t=\int_{s_0}^t \big( b(X_r)-b(Y_r)\big)\,dr, \quad   S_t= \int_{s_0}^t \big(\sigma(X_r)-\sigma(Y_r)\big)\,d B^H_r, \quad
E_t= \int_{s_0}^t \big( b(Y_r)-\tilde b(Y_r)\big)\,dr.
\end{equs}
We have the trivial estimate
\begin{equs}
\, [ E_{\cdot \wedge \tau} ]_{\scC^{1/2}_p([U,T])}      \leq     \|b-\tilde b\|_{C^0(\R^d)}.      \label{eq:est-for-E}
\end{equs}
Further, since
\begin{align*}
 [ D_{\cdot \wedge \tau} ]_{\scC^{1/2}_p([U,T])} & =\Big[ \int_{s_0}^{\cdot \wedge \tau}\big( b(\tilde{X}_r)-b(\tilde{X}_r-Z_{r\wedge \tau})\big) \,dr  \Big]_{\scC^{1/2}_p([U,T])}
\\
&\leq\Big[ \int_{s_0}^{\cdot} \big(b(\tilde{X}_r)-b(\tilde{X}_r-Z_{r\wedge \tau})\big)\,dr  \Big]_{\scC^{1/2}_p([U,T])},
\end{align*}
we may apply Lemma \ref{lem:Regularisation} with $f=b$ to get 
\begin{equs}
 {[D_{\cdot \wedge \tau}]}_{\scC^{1/2}_p([U,T])} &\lesssim      \| Z_{\cdot \wedge \tau} \|_{\scC^0_p([U,T])} |T-U|^{\eps}
\\
&\quad+ \Big( \big\|(1+[B^H]_{\cC^{H^{-}}|\bF})Z_{\cdot \wedge \tau}\big\|_{\scC^0_p([U,T])}+[Z_{\cdot \wedge \tau}]_{\scC^{1/2}_p([U,T])} \Big) |T-U|^{1/2+\eps}.
\end{equs}
This in turn implies that, for any $C\geq 1$
\begin{equs}
 {[D^Z_{\cdot \wedge \tau}]} _ {\scC^{1/2}_p([U,T])}& \lesssim   C\| Z_{U \wedge \tau} \|_{L_p(\Omega)}+  C \| Z_{\cdot \wedge \tau}\|_{\scC^{1/2}_p([U,T])}  |T-U|^{1/2+\eps}
 \\
 &\qquad +   \big\| \bone_{[ B^H ]_{\cC^{H^{-}}|\bF}\geq C}  [ B^H ]_{\cC^{H^{-}}|\bF}Z_{\cdot \wedge \tau}\big\|_{\scC^0_p([U,T])}.     \label{eq:idk1}
\end{equs}
Then, notice that
\begin{equs}
 \big\| \bone_{[ B^H ]_{\cC^{H^{-}}|\bF}\geq C}  [ B^H ]_{\cC^{H^{-}}|\bF}Z_{\cdot \wedge \tau}\big\|_{\scC^0_p([U,T])} & \leq \big( \P ([ B^H ]_{\cC^{H^{-}}|\bF}\geq C ) \big)^{1/2p}  \big\| [ B^H ]_{\cC^{H^{-}}|\bF}Z_{\cdot \wedge \tau}\big\|_{\scC^0_{2p}([U,T])}
 \\
 & \lesssim  \big( \P ([ B^H ]_{\cC^{H^{-}}|\bF}\geq C ) \big)^{1/2p}, \label{eq:idk2}
\end{equs}
where for the last inequality we have used the fact that $\| [ B^H ]_{\cC^{H^{-}}|\bF}\|_{L_p(\Omega)}\lesssim 1$,
while for  $ \sup_{t \in [0, 1]} |Z_{\cdot \wedge \tau}|$  we have by Lemma \ref{lem:apriori-K} the almost sure bound 
\begin{equs}
 \sup_{t \in [0, 1]} |Z_{\cdot \wedge \tau}| & \leq | Z_0| +  [Z_{\cdot \wedge \tau}]_{\cC^{H_-}}
  \leq |x_0-y_0|+[ X_{\cdot \wedge \tau}]_{\cC^{H_{-}}}+ [Y_{\cdot \wedge \tau}]_{\cC^{H_{-}}}\lesssim 1.
\end{equs}
Consequently, from \eqref{eq:idk1} we arrive at 
\begin{equs}
{[D_{\cdot \wedge \tau}]} _ {\scC^{1/2}_p([U,T])}   &\lesssim   C\| Z_{U \wedge \tau} \|_{L_p(\Omega)} +   C \| Z_{\cdot \wedge \tau}\|_{\scC^{1/2}_p([U,T])} |T-U|^{1/2+\eps}
\\
& \qquad + \big( \P ([ B^H ]_{\cC^{H^-}|\bF}\geq C ) \big)^{1/2p}. 
  \label{eq:est-for-DZ}
\end{equs}
Next, consider the third term on the right hand side of \eqref{eq:ready-to-buckle}. By \eqref{eq:Young remainder}, upon choosing $\eps_0>0$ small so that $1/2-\eps_0+ H^{-} > 1+\eps_0$,  it follows that 
\begin{equs}
\Big|\int_s^t \sigma (X_{r\wedge \tau}) & -\sigma(Y_{r\wedge \tau}) \, dB^H_r \Big|
\\
 & \lesssim  \|\sigma (X_{ \cdot \wedge \tau})-\sigma(Y_{ \cdot\wedge \tau})\|_{\cC^0} [B^H]_{\cC^{H^{-}}} |t-s|^{H^{-}}
\\
& \qquad +  [\sigma (X_{\cdot \wedge \tau})-\sigma(Y_{ \cdot\wedge \tau})]_{\cC^{1/2-\eps_0}([U, T])} [B^H]_{C^{H^{-}}} |t-s|^{1+\eps_0}
\\
& \lesssim  [B^H]_{\cC^{H^{-}}} \| Z_{ \cdot \wedge \tau} \|_{\cC^0} |t-s|^{H^-}+ [Z_{ \cdot\wedge \tau}]_{C^{1/2-\eps_0}([U, T])} [B^H]_{\cC^{H^{-}}} |t-s|^{1+\eps_0}
\\
& \lesssim [B^H]_{\cC^{H^{-}}|\bF}| Z_{ U \wedge \tau}| |t-s| ^{1/2} +  [Z_{ \cdot\wedge \tau}]_{\cC^{1/2-\eps_0}([U, T])} [B^H]_{\cC^{H^{-}}|\bF} |t-s|^{1+\eps_0}.
\end{equs}
Upon taking $L_p(\Omega)$ norms and using
Kolmogorov's continuity theorem in the form of the inequality $ \| \cdot \|_{ L_p\big(\Omega; \cC^{1/2-\eps_0}([U,T]) \big)} \lesssim \| \cdot \|_{\scC^{1/2}_p([U, T])}$ for continuous processes provided $p$ is sufficiently large, we get 
\begin{equs}
\Big\|\int_s^t \sigma (X_{r\wedge \tau})-\sigma(Y_{r\wedge \tau}) \, dB^H_r \Big\|_{L_p(\Omega)} &  \lesssim \big\| Z_{ U \wedge \tau}[B^H]_{\cC^{H^{-}}|\bF}\big\|_{L_p(\Omega)}  |t-s|^{1/2}
\\
& \qquad + \big[Z_{ \cdot \wedge \tau}[B^H]_{C^{H^{-}}|\bF}\big]_{\scC^{1/2}_p([U, T])}  |t-s|^{1+\eps_0}.
\end{equs}
The latter, upon dividing by $|t-s|^{1/2}$ and taking suprema over $(s, t)$ gives 
\begin{equs}
\, [ S_{\cdot \wedge \tau}]_{\scC^{1/2}_p([U, T])} &  \leq  \Big[ \int_0^\cdot \sigma (X_{r\wedge \tau})-\sigma(Y_{r\wedge \tau}) \, dB^H_r \Big]_{\scC^{1/2}_p([U, T])}
\\
& \lesssim  \big\| Z_{ U \wedge \tau}[B^H]_{\cC^{H^{-}}|\bF}\big\|_{L_p(\Omega)} 
+  \big[Z_{\cdot \wedge \tau}[B^H]_{\cC^{H^{-}}|\bF}\big]_{\scC^{1/2}_p([U, T])}  |T-U|^{1/2+\eps_0}
\\
 & \lesssim    C\| Z_{U \wedge \tau} \|_{L_p(\Omega)} +   C \| Z_{\cdot \wedge \tau}\|_{\scC^{1/2}_p([U,T])} |T-U|^{1/2+\eps_0}
 \\
 &\qquad + \big( \P ([ B^H ]_{\cC^{H^-}|\bF}\geq C ) \big)^{1/2p}, 
 \label{eq:est-for-SZ}
\end{equs}
where we used the argument from \eqref{eq:idk1}-\eqref{eq:idk2} in the last inequality.

Consequently,  by  \eqref{eq:ready-to-buckle}, \eqref{eq:est-for-E},  \eqref{eq:est-for-DZ}, and \eqref{eq:est-for-SZ}, 
we get
\begin{equs}
\, [ Z_{\cdot \wedge \tau}]_{\scC^{1/2}_p([U, T])} &   \lesssim  C\| Z_{U \wedge \tau} \|_{L_p(\Omega)} +   C \| Z_{\cdot \wedge \tau}\|_{\scC^{1/2}_p([U,T])} |T-U|^{1/2+\eps_1}
 \\
 &\qquad \qquad \qquad  + \big( \P ([ B^H ]_{\cC^{H^-}|\bF}\geq C ) \big)^{1/2p}+   \|b-\tilde b\|_{C^0(\R^d))},
\end{equs}
with $\eps_1= \eps_0 \wedge \eps$. 
In other words (after using the trivial inequality $\|\cdot\|_{\scC^{1/2}_p}\leq \|\cdot_U\|_{L_p(\Omega)}+2[\cdot]_{\scC^{1/2}_p}$),
 with a constant $N_0\lesssim 1$ we have
\begin{equs}
\| Z_{\cdot \wedge \tau}\|_{\scC^{1/2}_p([U, T])} &   \leq N_0  C\| Z_{U \wedge \tau} \|_{L_p(\Omega)} +  N_0  C \| Z_{\cdot \wedge \tau}\|_{\scC^{1/2}_p([U,T])} |T-U|^{1/2+\eps_1}
 \\
 &\qquad \qquad \qquad  + N_0 \big( \P ([ B^H ]_{\cC^{H^-}|\bF}\geq C ) \big)^{1/2p}+ N_0  \|b-\tilde b\|_{C^0(\R^d)}.\label{eq:buckling}
\end{equs}
 If $|T-U|^{1/2+\eps_1} \leq (2N_0 C)^{-1}$, then the inequality buckles and we get 
\begin{equ}
 \| Z_{\cdot \wedge \tau}\|_{\scC^{1/2}_p([U,T])}  \leq N  C\| Z_{U \wedge \tau} \|_{L_p(\Omega)}+  N  \big( \P ([ B^H ]_{\cC^{H^-}|\bF}\geq C ) \big)^{1/2p}+N  \|b-\tilde b\|_{C^0(\R^d)}.         \label{eq:iterable}
\end{equ}
for some $N$.  Let us now set $S_0=s_0$ and $S_{n+1}= S_n+ 1/m_1$, where $m_1=(2N_0C)^{2-\gamma_1}$, and $\gamma_1$  is defined by $1/(1/2+\eps_1)=2-\gamma_1$. In particular, $|S_{n+1}-S_n|^{1/2+\eps_1}= (2N_0 C)^{-1}$. We apply \eqref{eq:iterable} with $U=S_n$ and $T=S_{n+1}$ and we iterate, to get 
\begin{equs}
 \| Z_{\cdot \wedge \tau}\|_{\scC^{1/2}_p([S_n,S_{n+1}])} & \leq (N C)^{n+1} |x_0-y_0|+   \sum_{j=0}^n (N C)^j  \big(  \big( \P ([ B^H ]_{\cC^{H^-}|\bF}\geq C ) \big)^{1/2p}+  \|b-\tilde b\|_{C^0(\R^d)} \big)
 \\
 &\leq (N C)^{n+2} \Big(|x_0-y_0|+    \big(  \big( \P ([ B^H ]_{\cC^{H^-}|\bF}\geq C ) \big)^{1/2p}+  \|b-\tilde b\|_{C^0(\R^d)} \Big),
\end{equs}
 for some $N$ (which may change from line to line)
Since $S_k\geq 1$ as soon as $k\geq m_1$, we get
\begin{equs}
 \| Z_{\cdot \wedge \tau}\|_{\scC^{1/2}_p([s_0,1])} & \leq  (NC)^{m_1+2} \Big(   |x_0-y_0|+    \big( \P ([ B^H ]_{\cC^{H^-}|\bF}\geq C ) \big)^{1/2p}+\|b-\tilde b\|_{C^0(\R^d)} \ \Big) .
\end{equs}
From this the claim follows for any  $\gamma < \gamma_1$, with $N_1$ chosen sufficiently large.
\end{proof}

\begin{corollary}\label{cor:initial-stability}
Let  $p \geq 1$, $s_0\in[0,1)$, $x_0,y_0\in\R^d$, and let $b$ and $\sigma$ satisfy $\mathbf{A}_{\mathrm{Young}}$.  For any $\theta \in (0, 1)$, there exists a constant $N=N(\theta,\alpha,H,d,d_0,M,\lambda,p,K)$
with the following property:
if  $X$ and $Y$ are strong solutions of  \eqref{eq:main}  with initial conditions $x_0$ and $y_0$, respectively, then 
\begin{equs}\label{eq:initial-almost-Lipschitz}
\| X_{\cdot \wedge \tau}- Y_{\cdot \wedge \tau}\|_{\scC^{1/2}_p([s_0,1])}  \leq  N |x_0-y_0|^\theta. 
\end{equs}
\end{corollary}
\begin{proof}
Note that due to the a priori bounds on $X$ and $Y$ by Lemma \ref{lem:apriori-K}, it suffices to prove \eqref{eq:initial-almost-Lipschitz} when $|x_0-y_0|$ is sufficiently small.
Recall that by Fernique's theorem there exists $\beta= \beta(\alpha,H,p)>0$ such that 
$$
\big( \P ([ B^H ]_{\cC^{H^-}|\bF}\geq C ) \big)^{1/2p} \leq e^{-\beta C^2}.
$$
If $x_0=y_0$, the above exponential bound together with  Lemma \ref{lem:stability} imply that 
\begin{equs}
\| X_{\cdot \wedge \tau}- Y_{\cdot \wedge \tau}\|_{\scC^{1/2}_p([s_0,1])}  \leq N_1 ^{C^{2-\gamma} } e^{-\beta C^2},
\end{equs}
with the constants $N_1$ and $\gamma$ as in Lemma \ref{lem:stability}, in particular, independent of $C$.  We can let now $C \to \infty$, and since $\gamma>0$, we see that the right-hand side above converges to $0$.  

If $x_0 \neq y_0$, then Lemma \ref{lem:stability} yields 
\begin{equs}
\| X_{\cdot \wedge \tau}- Y_{\cdot \wedge \tau}\|_{\scC^{1/2}_p([s_0,1])}  \leq N_1 ^{C^{2-\gamma} }  \Big( |x_0-y_0|+  \big( \P ([ B^H ]_{\cC^{H^-}|\bF}\geq C ) \big)^{1/2p} \Big). 
\end{equs}
Consequently, by choosing $C= \beta^{-1/2} \sqrt{\log(1/|x_0-y_0|) }$ (which is larger than $1$ if $|x_0-y_0|$ is sufficiently small) we get 
\begin{equs}
\| X_{\cdot \wedge \tau}- Y_{\cdot \wedge \tau}\|_{\scC^{1/2}_p([s_0,1])}  & \leq   \exp\Big(N_3 \big(\log(1/|x_0-y_0|)\big)^{(2-\gamma)/2} \Big)    |x_0-y_0|
\end{equs}
with some new constant $N_3$.
The bound \eqref{eq:initial-almost-Lipschitz} now follows: since $(2-\gamma)/2<1$,  for any $\theta \in (0, 1)$ we can choose $N$ large enough,  depending only on $N_3$, $\theta$, and $\gamma$, such that 
\begin{equs}
  \exp\Big(N_3  \big(\log(1/|x_0-y_0|)\big)^{(2-\gamma)/2} \Big)  & \leq  N    |x_0-y_0|^{\theta-1}.
\end{equs}
\end{proof}

\subsection{Strong existence and the semiflow}
The advantage of the generality allowing the stability bounds in Lemma \ref{lem:stability} is that one can easily deduce strong existence of solutions as well as the existence of a certain notion of solution semiflow.
\begin{theorem}\label{thm:exist-Young}
Let   $s_0\in[0,1)$, $x_0\in\R^d$, and let $b$ and $\sigma$ satisfy $\mathbf{A}_{\mathrm{Young}}$. Then there exists a strong solution $X^{s_0,x_0}$ of \eqref{eq:main}.
Moreover, the random field $X$ given by
\begin{equ}
\big(\mathbb{Q}\cap[0,1)\big)\times\R^d\times[0,1]\ni (s_0,x_0,t)\mapsto X^{s_0,x_0}_t\in\R^d
\end{equ}
has a modification $\what X$ such that with an event $\what \Omega\subset\Omega_H$ of full probability one has for all $\omega\in\what \Omega$:
\begin{enumerate}[(i)]
\item For all $s_0\in[0,1)\cap\Q$, $x_0\in\R^d$, the map $[s_0,1]\ni t\mapsto \what X^{s_0,x_0}_t(\omega)$ is a solution of \eqref{eq:main};
\item For all $s_0,u_0\in[0,1)\cap\Q$, $s_0\leq u_0$, $t\in[u_0,1]$, $x_0\in\R^d$, one has the identity $\what X^{s_0,x_0}_t(\omega)=\what X^{u_0,\what X^{s_0,x_0}_{u_0}(\omega)}_t(\omega)$;
\item For all $\theta\in(0,1)$ and $R<\infty$ there exists a constant $N(\omega,\theta,R)$ such that for all $s_0\in[0,1)\cap\Q$, $t\in[s_0,1]$, $x_0,y_0\in\R^d$, $|x_0|,|y_0|\leq R$, one has $|\what X^{s_0,x_0}_t(\omega)-\what X^{s_0,y_0}_t(\omega)|\leq N |x_0-y_0|^\theta$.
\end{enumerate}
\end{theorem}
\begin{proof}
We start with the first statement, so let us fix $s_0\in[0,1)$, $x_0\in\R^d$.
Let $\beta \in(1-H^-,1/2)$ and let us take $p\geq1$ large enough  so that 
\begin{equs}   \label{eq:embed}
\|\cdot\|_{L_p\big(\Omega; \cC^\beta([0,1])\big)}\lesssim \|\cdot\|_{\scC^{1/2}_p([0,1])} 
\end{equs}
holds for continuous processes.
Since $b \in \cC^\alpha$ and $\alpha>0$, there exists a sequence $(b^n)_{n=1}^\infty  \subset  \cC^1$  such that   $b^n \to b$ uniformly  as $n \to \infty$.  Let $X^n$ denote the strong solution of
\begin{equs}      \label{eq:approx-Xn}
X^n_t =x_0 + \int_{s_0}^t  b^n (X^n_s ) \, ds + \int_{s_0}^t \sigma(X^n_s) \, dB^H_s.
\end{equs}
The existence and uniqueness of $X^n$ is well-known.
By Lemma \ref{lem:stability}, it follows that for each $K\in\N$,  the sequence $(X^n_{\cdot  \wedge \tau_K})_{n=1}^\infty$ is Cauchy in $\scC^{1/2}_p([s_0,1])$,
and hence by \eqref{eq:embed}, in $L_p\big(\Omega; \cC^\beta([s_0,1])\big)$ as well.
  Consequently, there exists an adapted process $X$ such that,  for each $K\in\N$,  $X_{\cdot \wedge \tau_K} \in L_p\big(\Omega; \cC^\beta([s_0,1])\big)$ and $\lim_{n \to \infty} \| X_{\cdot \wedge \tau_K}-X^n_{\cdot \wedge \tau_K}\| _{L_p\big(\Omega; \cC^\beta([s_0,1])\big)}=0$.  Notice that this and the fact that almost surely, $\tau_K=1$ for $K$ large enough, implies that $\| X^n-X\|_{\cC^\beta([s_0, 1])} \to 0$ in probability, as $n \to \infty$. 
 By the uniform convergence of $b_n$ to $b$ we get immediately
 \begin{equs}
 \Big\| \int_{s_0}^\cdot  b^n (X^n_s ) \, ds- \int_{s_0}^\cdot  b (X_s ) \, ds \Big\|_{\cC^0([s_0, 1])}\to 0
 \end{equs}
 in probability.
 By  \eqref{eq:Young continuity} and the regularity of $\sigma$ we see that 
 \begin{equs}
  \Big\|  \int_{s_0}^\cdot  \sigma (X^n_s ) \, dB^H_s- \int_{s_0}^\cdot  \sigma  (X_s ) \, dB^H_s \Big\|_{\cC^{H^{-}}([s_0, 1])} &  \lesssim \|  \sigma (X^n_\cdot )-\sigma  (X_\cdot ) \|_{\cC^\beta([s_0, 1])} [B^H]_{ \cC^{H^{-}}([s_0, 1])}
  \\
&   \lesssim \|X^n-X  \|_{\cC^\beta([s_0, 1])} [B^H]_{ \cC^{H^{-}}([s_0, 1])}\to 0
 \end{equs}
in probability.
Hence we can pass to the limit in \eqref{eq:approx-Xn} and get that $X$ is a strong solution.

Now we vary $s_0$ and $x_0$ so let us also include them in the notation. Fix $\theta\in(0,1)$ and take $p$ large enough so that $\bar \theta:=\theta+2/p<1$. From Corollary \ref{cor:initial-stability} and \eqref{eq:embed} we get that for all $x_0,y_0\in\R^d$ and $K\in\N$ one has
\begin{equ}
\big\|\|X^{s_0,x_0}_{\cdot\wedge\tau_K}-X^{s_0,y_0}_{\cdot\wedge\tau_K}\|_{\cC^\beta([s_0,1])}\big\|_{L_p(\Omega)}\lesssim |x_0-y_0|^{\bar\theta}.
\end{equ}
Therefore, by (again) applying Kolmogorov's continuity criterion, we get a modification $\what X$ of $X$: a random field $(\what X^{s_0,x_0})_{x_0\in\R^d}$ that takes values in $\cC^\beta([s_0,1])$ and satisfies for all $K$ and $R$
\begin{equ}\label{eq:semiflow bound}
\Bigg\|\sup_{\substack{x_0\neq y_0\\|x_0|,|y_0|\leq R}}\frac{\|\what X^{s_0,x_0}_{\cdot\wedge\tau_K}-\what X^{s_0,y_0}_{\cdot\wedge\tau_K}\|_{\cC^\beta([s_0,1])}}{|x_0-y_0|^\theta}\Bigg\|_{L_p(\Omega)}\lesssim 1.
\end{equ}
We claim that this modification satisfies (i)-(iii). By continuity of $\what X$ as well as the regularity of the coefficients, to verify (i) it suffices to check that for any $s_0 \in \mathbb{Q},x_0\in \mathbb{Q}^d$, $t\mapsto \what X^{s_0,x_0}_t$ is a.s. a solution, which is immediate. The bound \eqref{eq:semiflow bound} implies (iii). As for (ii), by countability it is sufficient to check for any two fixed $s_0, u_0$, and by continuity, for any fixed $x_0$, the claimed identity holds a.s. for all $t\in[u_0,1]$. This follows from the limiting procedure: the solution to \eqref{eq:approx-Xn} is path-by-path unique, and since both $X^{n;s_0,x_0}_\cdot$ and $X^{n;u_0, X^{n;s_0,x_0}_{u_0}}_\cdot$ are a.s. solutions of the approximate equation on $[u_0,1]$ with initial time $u_0$ and initial condition $X^{n;s_0,x_0}_{u_0}$, they coincide a.s. for all $t\in[u_0,1]$.  Next, notice that from Corollary \ref{cor:initial-stability} and \eqref{eq:embed} we get for all $R$, $K$, and $p$, 
\begin{equs}
\sup_n \sup_{\substack{x_0\neq y_0\\|x_0|,|y_0|\leq R}}\frac{\| X^{n;u_0,x_0}_{\cdot\wedge\tau_K}- X^{n; u_0,y_0}_{\cdot\wedge\tau_K}\|_{L_p(\Omega; \cC^\beta([u_0,1]))}}{|x_0-y_0|^\theta} < \infty .
\end{equs}
The above, combined with the pointwise convergence  $\lim_{n \to \infty} \| X^{n; u_0,x_0}_{\cdot\wedge\tau_K}- \what X^{u_0,x_0}_{\cdot\wedge\tau_K}\|_{L_p(\Omega; \cC^\beta([u_0,1]))}=0$, for $x_0 \in \R^d$, implies that  $X^{n; u_0, \cdot}_{\cdot\wedge\tau_K} \to  \what X^{u_0,\cdot }_{\cdot\wedge\tau_K} $ in the space $\cC^{\theta/2} \big( B_R; L_p(\Omega; \cC^\beta([u_0,1]))\big)$,  where $B_R$ is the ball of radius $R$ in $\R^d$.  Again by Kolmogorov's continuity criterion, we can conclude that the convergence also takes place in the space $L_p\big(\Omega; \cC^{\theta/3} (B_R; \cC^\beta([u_0,1]))\big)$.  Since $R, K$ were arbitrary, this combined with the fact that $X^{n;s_0,x_0}_{u_0}$  have moments of any order uniformly bounded in $n \in \mathbb{N}$,  implies that 
$X^{n;u_0, X^{n;s_0,x_0}_{u_0}}_{\cdot} \to X^{u_0, X^{s_0,x_0}_{u_0}}_{\cdot }$   in $\cC^\beta([u_0,1])$, in probability. Recall that also $X^{n;s_0,x_0}_\cdot \to X^{s_0,x_0}_\cdot$ in $\cC^\beta([s_0,1])$, in probability.  Since $X^{n;s_0,x_0}_\cdot$ and $X_\cdot^{n;u_0, X^{n;s_0,x_0}_{u_0}}$ coincide, the claim follows. 
\end{proof}

\subsection{Path-by-path uniqueness}
Now we see the advantage of the more general existence formulation in Theorem \ref{thm:exist-Young}: one can deduce the path-by-path uniqueness of solutions.

\begin{theorem}\label{thm:PBP-uniq-Young}
Let $b$ and $\sigma$ satisfy $\mathbf{A}_{\mathrm{Young}}$ and take $\what\Omega\subset\Omega_H$ obtained from Theorem \ref{thm:exist-Young}.
Let $\omega\in \what\Omega$, $s_0\in[0,1)$, $x_0\in\R^d$, and let $Y,Y'$ be solutions of \eqref{eq:main}. Then $Y=Y'$.
\end{theorem}
\begin{proof}
The proof is inspired by \cite{shap}.
By continuity, it suffices to show that if $s_0<t\in\Q$, then $Y_t=\what X^{s_0,x_0}_t(\omega)$ on rationals. Indeed, extending to arbitrary $t$ is trivial, while for arbitrary $s_0$ take a sequence $(s_0^n)_{n\in\N}\subset\Q$ such that $s_0^n\searrow s_0$ and let $n\to\infty$ in the bound
\begin{equ}
|Y_t-Y'_t|=|\what X_t^{s_0^n,Y_{s_0^n}}(\omega)-\what X_t^{s_0^n,Y_{s_0^n}'}(\omega)|\lesssim|Y_{s_0^n}-Y'_{s_0^n}|^\theta.
\end{equ} 
Since $\omega$ is fixed throughout the proof, we drop it from the notation. For $s\in[s_0,t]\cap\Q$ define $f(s)=\what X^{s,Y_s}_t$. Clearly $f(s_0)=\what X^{s_0,x_0}_t$ and $f(t)=Y_t$, so it suffices to show that $f$ is constant. Using properties (ii) and (iii) of $\what X$, we can write for any $u,r\in[s_0,t]\cap\Q$, $u\leq r$, $\theta\in(0,1)$:
\begin{equ}\label{eq:PBP-1}
|f(r)-f(u)|=\big|\what X^{r,Y_r}_t-\what X^{u,Y_u}_t\big|=\big|\what X^{r,Y_r}_t-\what X^{r,y}_t|_{y=\what X^{u,Y_u}_r}\big|\lesssim |Y_r-\what X^{u,Y_u}_r|^\theta.
\end{equ}
Here and below the proportionality constant in $\lesssim$ may depend on anything other than the time parameters.
By property (i) of $\what X$ and the assumption on $Y$, both $r\mapsto\what X^{u,Y_u}$ and $r\mapsto Y_r$ are solutions of \eqref{eq:main} with $s_0=u$ and $x_0=Y_u$.
Therefore we have for all $v\in[u,r]$
\begin{equs}
Y_v-\what X^{u,Y_u}_v&=D_{u,v}+S_{u,v},
\end{equs}
where
\begin{equs}
D_{u,v}&=\int_u^v b(Y_s)-b(\what X^{u,Y_u}_s)\,ds,
\\
S_{u,v}&=\int_u^v \sigma(Y_s)-\sigma(\what X^{u,Y_u}_s)\,dB^H_s
\\
&=
\int_u^v \sigma(Y_s)-\sigma(\what X^{u,Y_u}_s)\,dB^H_s-\big(\sigma(Y_u)-\sigma(\what X^{u,Y_u}_u)\big)(B^H_v-B^H_u).
\end{equs}
Note that by the definition of a solution, both of the processes $Y_{\cdot}$ and $\what X^{u,Y_u}_{\cdot}$ are regular enough to make the Young integral $S_{u,v}$ well-defined: they belong to $\cC^\beta$ for some $\beta>1-H$.
So choosing $\bar H\in(1-\beta,H)$, we have $B^H\in \cC^{\bar H}$ and 
from Young integration we immediately get $|S_{u,v}|\lesssim |v-u|^{\beta+\bar H}$.
Furthermore, the boundedness of $b$ implies $|D_{u,v}|\lesssim |v-u|$.
In particular we get $|Y_s-\what X^{u,Y_u}_s|\lesssim |s-u|$, which, combined with the $\alpha$-H\"older continuity of $b$, we can use to upgrade the bound for $D$ to
$
|D_{u,v}|\lesssim |v-u|^{1+\alpha}.
$
Using these bounds with $r$ in place of $v$ and substituting in \eqref{eq:PBP-1}, we get
\begin{equ}
|f(r)-f(u)|\lesssim |r-u|^{\theta((1+\alpha)\wedge (\beta+\bar H))}.
\end{equ}
Since $\alpha>0$, $\beta+\bar H>1$, and $\theta$ can be chosen arbitrarily close to $1$, the exponent on the right-hand side can be made to be greater than $1$, and hence $f$ is indeed constant.
\end{proof}

\section{The rough case}
Throughout the section we fix $H\in(1/3,1/2]$ and assume that the random rough path lift $(B^H,\bB^H)$ of $B^H$ satisfies $\mathbf{A}_{\mathrm{rough}}$ (ii).
In addition, up until Section \ref{sec:weak} we fix $\alpha\in(0,1)$.
\subsection{A priori bound}
Fix $$1/3<H_- < H^-< H$$ (in particular, $2H_-+H^- >1$) such that $2H_--2H+\alpha H>0$ and $3H_--H+\alpha H>1/2$ (as before we consider $H_-, H^-$ as functions of $\alpha,H$).

For all $K \in\N$, define the stopping time 
\begin{equs} 
\tau_K= \inf\{ t >0 : \big[(B^H,\bB^H)\big]_{\cR^{H^-}([0, t])} \geq K  \}. 
\end{equs}
Note that by assumption $t\mapsto\big[(B^H,\bB^H)\big]_{\cR^{H^-}([0, t])}$ is adapted and continuous, so $\tau_K$ is indeed a stopping time. Moreover, for all $\omega\in\Omega_H$ there is a $K_0(\omega)$ such that $\tau_K=1$ for all $K\geq K_0$. To (slightly) ease the notation, we set 
\begin{equs}\label{eq:cD}
\scD^\gamma_K([S, T]):= \cD^\gamma_{B^H_{\cdot \wedge \tau_K}}([S, T]).
\end{equs}
\begin{lemma}                                                                         \label{lem:apriori-rough}
Let $s_0\in[0,1)$, $x_0\in\R^d$, $b\in \cC^0$ and $\sigma\in \cC^2$.
Then,   there exists a constant $N_0=N_0(K, \| b\|_{\cC^0}, \|\sigma\|_{\cC^2},\alpha,H,d, d_0)$,  such that if $X$ is a solution of \eqref{eq:main} then almost surely
\begin{equ}
\,[(X_{ \cdot \wedge \tau_K},  \sigma(X_{ \cdot \wedge \tau_K})]_{\scD^{2H_-}_K([s_0, 1])} \leq N_0.
\end{equ}

\end{lemma}

\begin{proof}
Let $s_0 \leq U <T \leq 1$ and as before, denote by $D$ and $S$ the first and second integrals in \eqref{eq:main}, respectively. For $(s, t) \in [U, T]_{\leq
}$, we have the trivial estimate 
\begin{equs} 
\, [( D_{\cdot \wedge \tau_K}, 0)]_{\scD^{2H_-}_K([U,T])} = [ D_{\cdot \wedge \tau_K}]_{C^{2H_-}([U, T])} \leq \| b\|_{\cC^0} .     \label{eq:est-drift}
\end{equs}
Further, by \eqref{eq:controll-increasing} we have 
\begin{equs}
\, [( S_{\cdot \wedge \tau_K}, \sigma(X_{ \cdot \wedge \tau_K})]_{\scD^{2H_-}_K([U, T])}&   \leq [( S_{\cdot \wedge \tau_K}, \sigma(X_{ \cdot \wedge \tau_K})]_{\scD^{2H^-}_K([U, T])}  |T-U|^{2(H^--H_-)}
\end{equs}
By  \eqref{eq:rough continuity} applied with $\alpha= H^-$ and $\gamma= 2H^-$,  and by  the definition of $\tau_K$,  we get 
\begin{equs} 
 \,  [( S_{\cdot \wedge \tau_K}, \sigma(X_{ \cdot \wedge \tau_K})]_{\scD^{2H^-}_K([U, T])}  & \lesssim  [( \sigma(X_{ \cdot \wedge \tau_K}), \nabla\sigma\sigma(X_{ \cdot \wedge \tau_K}) ]_{\scD^{2H_-}_K([U, T])} 
 \end{equs}
By  \eqref{eq:quadratic-controled-estimate} with $\beta=1$, the bounds on $\sigma$, and the definition of $\tau_K$,  we obtain 
 \begin{equs}
\,  [&( \sigma(X_{ \cdot \wedge \tau_K}), \nabla\sigma\sigma(X_{ \cdot \wedge \tau_K}) ]_{\scD^{2H_-}_K([U, T])} 
  \lesssim \big(  [ (X_{ \cdot \wedge \tau_K},  \sigma(X_{ \cdot \wedge \tau_K})]_{\scD^{2H_-}_K([U, T])} +1 \big) 
\end{equs}
Consequently, from the three estimates above we obtain
\begin{equs}
\,  [&( S_{\cdot \wedge \tau_K}, \sigma(X_{ \cdot \wedge \tau_K})]_{\scD^{2H_-}_K([U, T])}
\\  &\leq N    \big(   [ (X_{ \cdot \wedge \tau_K},  \sigma(X_{ \cdot \wedge \tau_K})]_{\scD^{2H_-}_K([U, T])} +  1  )  |T-U |^{2(H^--H_-)},      \label{eq:est-stoch}
\end{equs}
with $N=N(K,\| \sigma\|_{\cC^2}, H, d, d_0)$.  By \eqref{eq:est-drift} and \eqref{eq:est-stoch}, we get 
\begin{equs}
\,  [&(X_{ \cdot \wedge \tau_K},  \sigma(X_{ \cdot \wedge \tau_K})]_{\scD^{2H_-}_K([U, T])}
\\
  &\leq \| b\|_{\cC^0} +N   \big(   [ (X_{ \cdot \wedge \tau_K},  \sigma(X_{ \cdot \wedge \tau_K})]_{\scD^{2H_-}_K([U, T])}+  1  \big)  |T-U |^{2(H^--H_-)}.
\end{equs}
If $|T-U| \leq (2N)^{1/(H^--H_-)}$, then the estimate buckles and we get 
\begin{equs}
\, [(X_{ \cdot \wedge \tau_K},  \sigma(X_{ \cdot \wedge \tau_K})]_{\scD^{2H_-}_K([U, T])} \leq  2 \| b\|_{\cC^0}. 
\end{equs}
Let us now choose $M \in \mathbb{N}$  so that $1/M \leq  (2N)^{1/(H^--H_-)}$ and let us set $S_0=0$ and $S_{n+1} = S_n+ 1/M$ for $n \in {0, \ldots, M-1}$.  Then by  \eqref{eq:sum_partition} and the above inequality applied with $U$ and $T$ replaced by $S_n$ and $S_{n+1}$, we get  
\begin{equs}
\,[&(X_{ \cdot \wedge \tau_K},  \sigma(X_{ \cdot \wedge \tau_K})]_{\scD^{2H_-}_K([s_0,1])}
\\
& \leq \big(([ B^H_{\cdot \wedge \tau_K}]_{\cC^{H^{-}([0, 1])}} +1  \big)^M\sum_{n=0}^{M-1}[(X_{ \cdot \wedge \tau_K},  \sigma(X_{ \cdot \wedge \tau_K})]_{\scD^{2H_-}_K([S_n,  S_{n+1}])}
\\
& \leq (K+1)^M M 2 \| b\|_{\cC^0}  =: N_0. 
\end{equs}
This finishes the proof. 
\end{proof}

\begin{remark}
Notice that  since $2H_-+H^->1$ and $H^-< H_- < 1/2$, it follows that $H^-< 2H_-$.  This in turn implies that $[X_{\cdot \wedge \tau_K} ]_{\cC^{H^-}([s_0, 1])} \leq [(X_{ \cdot \wedge \tau_K},  \sigma(X_{ \cdot \wedge \tau_K})]_{\scD^{2H_-}_K([s_0,1])}(1+[B^H_{ \cdot \wedge \tau_K}])$. In particular, by the previous lemma we get that $[X_{\cdot \wedge \tau_K} ]_{\cC^{H^-}([s_0, 1])} \leq N_1$ for some $N_1= N_1(K, \| b\|_{\cC^0}, \|\sigma\|_{\cC^2},\alpha,H,d, d_0)$. 
\end{remark} 
\subsection{Strong uniqueness and stability}\label{sec:stabil-rough}
Throughout Section \ref{sec:stabil-rough} we fix $K\in\N$ and use the shorthand $\tau=\tau_K$.
Moreover, if $X$ is a solution of \eqref{eq:main} we define the process $\tilde X$ for $t\in[s_0,1]$ as in \eqref{eq:tildeX}.
Let us highlight that the second integral in \eqref{eq:tildeX} is best rewritten as
\begin{equs}\label{eq:stopped-integration}
\int_{s_0}^t\sigma(X_{s\wedge\tau})\,dB^H_s&=\int_{s_0}^{s_0\vee(t\wedge\tau)}\sigma(X_{s\wedge\tau})\,dB^H_{s\wedge\tau}+\sigma(X_\tau)\big(B^H_{t}-B^H_{s_0\vee(t\wedge\tau)}\big)
\\
&=\int_{s_0}^{s_0\vee(t\wedge\tau)}\sigma(X_{s})\,dB^H_{s}+\sigma(X_\tau)\big(B^H_{t}-B^H_{s_0\vee(t\wedge\tau)}\big),
\end{equs}
where the integral on the right-hand sides can be understood in the rough sense.

 \begin{lemma}             \label{lem:Regularisation-rough}
Let  $p \geq 2$, $s_0\in[0,1)$, $x_0\in\R^d$, $(S,T)\in[s_0,1]^2_\leq$, and let $b$ and $\sigma$ satisfy $\mathbf{A}_{\mathrm{rough}}$. Let  $X$ be a strong solution of \eqref{eq:main}.
Then there exists constants $N= N(\alpha, H, d, d_0, M,\lambda, p, K)$ and $\eps=\eps(\alpha,H)>0$ such that for all $f \in \cC^\alpha$, all adapted stochastic processes $Z$, and all $(s,t)\in[S,T]^2_\leq$, the  following bound holds 
\begin{equs}
\Big\| & \int_s^t  \big( f(\tilde{X}_r + Z_r) -f(\tilde{X}_r)\big) \, dr \Big\|_{L_p(\Omega)}
   \leq N  \| f\|_{\cC^\alpha}\| Z\|_{\scC^0_p([S,T])} |t-s|^{1+(\alpha-1)H}
\\
&\quad + N \|f\|_{\cC^\alpha}\Big(\big\|(1+ [ B^H]_{\cC^{H^-}|\bF})Z\big\|_{\scC^0_p([S,T])}
+[Z]_{\scC_p^{H^-}([S,T])}
\Big)|t-s|^{1+2H_--2H+\alpha H}.\label{eq:one-big-bound-rough}
\end{equs}
\end{lemma}
\begin{remark}
It is a convenient feature of \eqref{eq:one-big-bound-rough} that only the norm of $B^H$, and not of $\mathbf{B}^H$ appears on the right-hand side of \eqref{eq:one-big-bound-rough}. One has much better exponential estimates for the former, which plays a role in the proof of Lemma \ref{lem:stability-rough} below.
\end{remark}
\begin{proof}
We may and will assume $\|f\|_{\cC^\alpha}=1$.
For $S\leq s\leq u\leq r\leq t\leq T$, set
\begin{equs}
\Theta_{s} & = \int_{s_0}^sb(X_{v  \wedge \tau})  \, d v ,
\\
\Sigma_s &=\int_{s_0}^s \sigma(X_{v\wedge\tau})\,dB^H_v,
\\
\Xi_{s,r}  &= \Theta_{s} +\Sigma_s+ \sigma(X_{s\wedge \tau}) ( B^H_r- B^H_s),
\\
 \hat {\Xi}_{s,u, r} &= \Theta_{s}  +\Sigma_s+ \sigma(X_{s\wedge \tau}) ( \E^u B^H_r- \ B^H_s).
\end{equs}
We wish to apply stochastic sewing, i.e. Lemma \ref{lem:SSL}. A slight simplification compared to the proof Lemma \ref{lem:Regularisation} is that the shifting is not necessary, so the original stochastic sewing lemma \cite{Khoa} suffices.
We wish to verify the conditions of the lemma with the processes
\begin{equs}
A_{s, t}&= \E^{s} \int_s^t f(\Xi_{s, r} +Z_{s}) -f(\Xi_{s, r}) \, dr,
\\
\cA_t&=\int_S^t \big(f(\tilde{X}_r + Z_r) -f(\tilde{X}_r)\big) \, dr.
\end{equs}
Recall that as before, with the notation
$$
\Gamma(s,u, r)= c(H)(r-u)^{2H} \sigma \sigma^\top (X_{s\wedge \tau}),
$$ 
we have by \eqref{eq:very-basic2} the identity $\E^u f(\Xi_{s,r})=\cP_{\Gamma(s,u,r)}f(\hat\Xi_{s,u,r})$.
Therefore by \eqref{eq:C1}, \eqref{eq:HKKK}, and the lower bound assumed on $\sigma\sigma^\top$ we have almost surely
\begin{equs}
|A_{s, t}|& = \Big|\int_s^t \cP_{\Gamma(s,s, r)}f (\hat{\Xi}_{s,s,  r}+Z_{s} )- \cP_{\Gamma(s,s, r)}f (\hat{\Xi}_{s,s, r}) \, dr\Big| 
\\
&\leq |Z_{s}|\int_s^t[\cP_{\Gamma(s,s, r)}f]_{C^1}\,dr 
\\
& \lesssim |Z_{s}| \int_s^t (r-s)^{(\alpha-1)H}   \, dr. 
\end{equs}
After taking $L_p(\Omega)$ norms, we get
\begin{equs}
\| A_{s, t} \|_{L_p(\Omega)} \lesssim \| Z\|_{\scC^0_p([S,T])} (t-s)^{1+(\alpha-1)H}.
\end{equs}
Since $1+(\alpha-1)H>1/2$, condition \eqref{eq:SSL-cond1} is satisfied with $N_1=N  \| Z\|_{\scC^0_p([S,T])}$.

Moving on to  \eqref{eq:SSL-cond2}, without the shifts in the definition of $A$ we have a simpler expression:
\begin{equs}
\E^{s} \delta A_{s, u, t}&= 
\E^{s}\E^u 
\int_{u}^t f(\Xi_{s,r}+Z_{s})-f(\Xi_{s,r})-f(\Xi_{u,r}+Z_{u})+f(\Xi_{u,r})\,dr.
\end{equs}
As before, by conditional Jensen's inequality we may drop the first conditioning and write
\begin{equ}
\|\E^{s} \delta A_{s, u, t}\|_{L_p(\Omega)}\leq\int_u^t\|J_r^1\|_{L_p(\Omega)}+\|J_r^2\|_{L_p(\Omega)}\,dr,
\end{equ}
where $J^1$ and $J^2$ are slightly simplified versions of \eqref{eq:def-J1}-\eqref{eq:def-J2}:
\begin{equs}
J^1_{r}&=   \cP_{\Gamma(s,u, r)}f (\hat{\Xi}_{s,u, r}+Z_{s} ) - \cP_{\Gamma(s,u, r)}f (\hat{\Xi}_{s,u, r})
\\
&\qquad -   \cP_{\Gamma(s,u, r)}f (\hat{\Xi}_{u,u, r}+Z_u)+ \cP_{\Gamma(s,u, r)}f (\hat{\Xi}_{u,u, r}),
\\
J^2_r&=   \cP_{\Gamma(s,u, r)}f (\hat{\Xi}_{u,u, r}+Z_u)-  \cP_{\Gamma(u,u, r)}f (\hat{\Xi}_{u,u, r}+Z_u)
\\
&\qquad -\cP_{\Gamma(s,u, r)}f (\hat{\Xi}_{u,u, r})+\cP_{\Gamma(u,u, r)}f (\hat{\Xi}_{u,u, r}).
\end{equs}
Concerning $J^1$, by \eqref{eq:C2} and \eqref{eq:HKKK} we have
\begin{equs}        
|J^1_r|
&\lesssim
|Z_{s}-Z_u|(r-u)^{(\alpha-1)H}+|Z_u||\hat{\Xi}_{s,u, r}-\hat{\Xi}_{u,u, r}|(r-u)^{(\alpha-2)H}.\label{eq:J1again}
\end{equs}
We have
\begin{equs}
|\hat{\Xi}_{s,u, r}-\hat{\Xi}_{u,u, r}| & \leq | \Theta_{s} - \Theta_{u} |
+ \big| \E^u \big( ( \sigma(X_{s\wedge \tau})  - \sigma(X_{u\wedge \tau}) ( B^H_r-  B^H_u)\big) \big|
\\
&\qquad+\big| \Sigma_{s} -\Sigma_u+ \sigma(X_{s\wedge \tau}) ( B^H_u- B^H_{s})\big|
\\
&\lesssim (s-u)+(s-u)^{H^-}(r-u)^{H^-}[B^H]_{\cC^{H^-}|\bF}
\\
&\qquad+\big| \Sigma_{s} -\Sigma_u+ \sigma(X_{s\wedge \tau}) ( B^H_u- B^H_{s})\big|.
\end{equs}
The last term is a rough remainder, but we have to be a little careful: as mentioned, we do not want the norm of the (non-stopped) $\mathbf{B}^H$ to appear.
As in \eqref{eq:stopped-integration}, we can write
\begin{equs}
\Sigma_{u} -\Sigma_s- \sigma(X_{s\wedge \tau}) ( B^H_u- B^H_{s})&=\Big(\int_s^{s\vee(u\wedge\tau)}\sigma(X_{v\wedge\tau})\,dB^H_{v\wedge\tau}-\sigma(X_{s\wedge\tau})\big(B^H_{s\vee(u\wedge\tau)}-B^H_s\big)\Big)
\\
&\qquad-\big(\sigma(X_{s\vee(u\wedge\tau)})	-\sigma(X_{s\wedge\tau})\big)\big(B^H_u-B^H_{s\vee(u\wedge\tau)}\big).\label{eq:idk5}
\end{equs}
In this form it is clear that the first term is a rough remainder of the stopped process against the stopped noise, whose controlled and rough norms are of order $1$ by Lemma \ref{lem:apriori-rough} and by definition, respectively. Therefore the remainder is bounded a.s. by a constant of order $1$ times $|t-s|^{2H_-}$.
The second term in \eqref{eq:idk5} is easily bounded a.s. by $(u-s)^{2H^-}[B^H]_{\cC^{H^-}}$. Combining all of the above gives the almost sure bound
\begin{equ}
|\hat{\Xi}_{s,u, r}-\hat{\Xi}_{u,u, r}| \lesssim (t-s)^{2H_-}\big(1+[B^H]_{\cC^{H^-}|\bF}\big).
\end{equ}
Using this in \eqref{eq:J1again} and taking $L_p(\Omega)$ norms, we can conclude
\begin{equs}
\|J^1_r\|_{L_p(\Omega)}&\lesssim [Z]_{\scC^{H^-}_p([S,T])}(t-s)^{H^-}(r-u)^{(\alpha-1)H}
\\
&\qquad+\big\|Z(1+[B^H]_{\cC^{H^-}|\bF})\big\|_{\scC^{0}_p([S,T])}(t-s)^{2H_-}(r-u)^{(\alpha-2)H}.
\end{equs}
Bounding $J^2$ goes the same way as in \eqref{eq:bounding-J2}
(the only minuscule difference is that here by \eqref{eq:triv-controll-embedding} and Lemma \ref{lem:apriori-rough} even the $H^-$ norm of $X_{\cdot\wedge\tau}$ is bounded), yielding
\begin{equ}
\|J^2_r\|\lesssim  \|Z\|_{\scC^{0}_p([S,T])}(t-s)^{H^-}(r-u)^{2H-3H+\alpha H}.
\end{equ}
Since all powers of $(r-u)$ are integrable and $H^-+\alpha H-H>2H_-+(\alpha-2)H>0$, are positive, we get
\begin{equ}
\|\E^{s} \delta A_{s, u, t}\|_{L_p(\Omega)}\lesssim \Big(\big\|Z(1+[B^H]_{\cC^{H^-}|\bF})\big\|_{\scC^{0}_p([S,T])}+[Z]_{\scC^{H^-}_p([S,T])}\Big)(t-s)^{1+2H_--2H+\alpha H}.
\end{equ}
Therefore \eqref{eq:SSL-cond2} is satisfied with $N_2=\Big(\big\|Z(1+[B^H]_{\cC^{H^-}|\bF})\big\|_{\scC^{0}_p([S,T])}+[Z]_{\scC^{H^-}_p([S,T])}\Big)$.
The verification of \eqref{eq:SSL-cond3} carries through from the proof of Lemma \ref{lem:Regularisation} without any serious change: note that the argument only requires $[\tilde X]_{\scC^\gamma_p}<\infty$ for some $\gamma>0$ (which here is satisfied with $\gamma=H^-$) and $\|\Xi_{s,r}-\tilde X_s\|_{L_p(\Omega)}\lesssim |r-s|^\gamma$ for some $\gamma>0$ (which is also satisfied with $\gamma=H^-$).
Therefore all the conditions of Lemma \ref{lem:SSL} are satisfied and \eqref{eq:one-big-bound-rough} follows from \eqref{eq:SSL-conc}.
\end{proof}

Recall the notation $\scD_K^\gamma$ from \eqref{eq:cD}. Note that $\scD_K^\gamma$ is an $\omega$-dependent path space, and one can define a complete space $L_p(\Omega;\scD_K^{\gamma})$ by defining the norm in the obvious way.
\begin{lemma}                              \label{lem:stability-rough}
Let  $p \geq 1$, $s_0\in[0,1)$, $x_0, y_0 \in\R^d$. Let $b$, $\tilde b$, and $\sigma$ satisfy $\mathbf{A}_{\mathrm{rough}}$.
Then there exist constants $N_1, N_2$, depending only on  $\alpha, H, d, d_0, M, \lambda, p,$ and $ K$, as well as a constant $\gamma=\gamma(\alpha,H)\in(0,2)$,  with the following property:
 for any  $X$ and $Y$ strong solutions of  \eqref{eq:main} and \eqref{eq:aux}, respectively, and for all $C\geq N_2$,  the following estimate holds
\begin{equs}
\big\| \big(X_{\cdot \wedge\tau }- &Y_{\cdot \wedge\tau },\sigma(X_{\cdot\wedge\tau})-\sigma(Y_{\cdot\wedge\tau})\big)\big\|_{L_p(\Omega;\scD_K^{1-H_-})}  
\\
&\leq N_1 ^{C^{2-\gamma}} \Big( |x_0-y_0|+  \big( \P ([ B^H ]_{\cC^{H^-}|\bF}\geq C ) \big)^{1/2p} +\|b-\tilde b\|_{\cC^0(\R^d)} \Big).
\end{equs}
\end{lemma}

\begin{proof}
Note that one may (and we will) assume $p$ to be sufficiently large.
Take $(U,T)\in[s_0,1]^2_{\leq}$. Denoting $Z=X-Y$ and $Z'=\sigma(X)-\sigma(Y)$, one has
\begin{equs}
\big\| (Z_{\cdot \wedge \tau},Z'_{\cdot\wedge\tau}) \big\|_{\scD_K^{1-H_-}([U,T])}
 &\lesssim \| Z_{U \wedge \tau} \|_{L_p(\Omega)}+ [ D_{\cdot \wedge \tau} ]_{\cC^{1-H_-}([U,T])}
 \\
 & \qquad +  \big[ (S_{\cdot \wedge \tau},Z'_{\cdot\wedge\tau}) \big]_{\scD_K^{1-H_-}([U,T])}      +  [ E_{\cdot \wedge \tau} ]_{\cC^1([U,T])}             \label{eq:ready-to-buckle-rough}
\end{equs}
where 
\begin{equs}
D_t=\int_{s_0}^t \big( b(X_r)-b(Y_r)\big)\,dr, \quad   S_t= \int_{s_0}^t \big(\sigma(X_r)-\sigma(Y_r)\big)\,d B^H_r, \quad
E_t= \int_{s_0}^t \big( b(Y_r)-\tilde b(Y_r)\big)\,dr.
\end{equs}
The bound
\begin{equ}\label{eq:roughbuckE}
\,[ E_{\cdot \wedge \tau} ]_{\cC^1([U,T])} \leq \|b-\tilde b\|_{\cC^0(\R^d)}.
\end{equ}
is trivial.
Bounding $S$, by working pathwise, is also easy: by \eqref{eq:rough continuity}, \eqref{eq:controlled-comp-stability}, and the bounds (by definition of $\tau$ and Lemma \ref{lem:apriori-rough}) $[(B^H, \bB^H)]_{\cR^{H^-}([0,\tau])},
[X_{\cdot\wedge\tau},\sigma(X_{\cdot\wedge\tau})]_{\scD^{1-H_-}_K},
[Y_{\cdot\wedge\tau},\sigma(Y_{\cdot\wedge\tau})]_{\scD^{1-H_-}_K},\lesssim 1$, we get
\begin{equs}
\big[ (S_{\cdot \wedge \tau},Z'_{\cdot\wedge\tau}) \big]_{\scD_K^{1-H_-}([U,T])} & \lesssim\big[(\sigma(X_{\cdot\wedge\tau})-\sigma(Y_{\cdot\wedge\tau}),\nabla\sigma\sigma(X_{\cdot\wedge\tau})-\nabla\sigma\sigma(Y_{\cdot\wedge\tau})\big]_{\scD_K^{1-H_-}([U,T])}
\\
&\lesssim |Z_{U\wedge\tau}| + \big[(Z_{\cdot\wedge\tau},Z'_{\cdot\wedge\tau})]_{\scD_K^{1-H_-}([U,T])}.\label{eq:roughbuckS}
\end{equs}
Concerning $D$, first note that for large enough $p$
\begin{equs}
\, [ D_{\cdot \wedge \tau} ]_{L_p(\Omega;\cC^{1-H_-}([U,T]))} 
&\leq\Big[ \int_{s_0}^{\cdot} \big(b(\tilde{X}_r)-b(\tilde{X}_r-Z_{r\wedge \tau})\big)\,dr  \Big]_{L_p(\Omega;\cC^{1-H_-}([U,T]))} 
\\
&\lesssim \Big[ \int_{s_0}^{\cdot} \big(b(\tilde{X}_r)-b(\tilde{X}_r-Z_{r\wedge \tau})\big)\,dr  \Big]_{\scC_p^{1-H}([U,T])} ,
\end{equs}
and the latter term can be bounded via Lemma \ref{lem:Regularisation-rough} with $f=b$.
Indeed, we get
\begin{equs}
\Big[ &\int_{s_0}^{\cdot} \big(b(\tilde{X}_r)-b(\tilde{X}_r-Z_{r\wedge \tau})\big)\,dr  \Big]_{\scC_p^{1-H}([U,T])} \lesssim \|Z_{\cdot\wedge\tau}\|_{\scC^0_p([U,T])}|T-U|^{1+(\alpha-1)H-1+H}
\\
&\quad+ \Big( \big\|(1+[B^H]_{\cC^{H^{-}}|\bF})Z_{\cdot \wedge \tau}\big\|_{\scC^0_p([U,T])}+[Z_{\cdot \wedge \tau}]_{\scC^{H^-}_p([U,T])} \Big) |T-U|^{1+2H_--2H+\alpha H-1+H}.
\end{equs}
Note that the exponents on the right-hand side simplify to $\alpha H$ and $2H_--H+\alpha H$, respectively. In particular, by the trivial inequality $\|Z_{\cdot\wedge\tau}\|_{\scC^0_p([U,T])}\leq\|Z_{U\wedge\tau}\|_{L_p(\Omega)}+[Z_{\cdot\wedge\tau}]_{\scC^{H^-}_p([U,T])}|T-U|^{H^-}$, we can slightly simplify to
\begin{equs}
\Big[ &\int_{s_0}^{\cdot} \big(b(\tilde{X}_r)-b(\tilde{X}_r-Z_{r\wedge \tau})\big)\,dr  \Big]_{\scC_p^{1-H}([S,T])} \lesssim \|Z_{U\wedge\tau}\|_{L_p(\Omega)}|T-U|^{\alpha H}
\\
&\quad+ \Big( \big\|(1+[B^H]_{\cC^{H^{-}}|\bF})Z_{\cdot \wedge \tau}\big\|_{\scC^0_p([U,T])}+[Z_{\cdot \wedge \tau}]_{\scC^{H^-}_p([U,T])} \Big) |T-U|^{2H_--H+\alpha H}.
\end{equs}
To treat the factor $[B^H]_{\cC^{H^{-}}}$ we proceed similarly to Lemma \ref{lem:Regularisation}: on the event $\{[B^H]_{\cC^{H^{-}}}>C\}$ we bound as in  \eqref{eq:idk2}, while on $\{[B^H]_{\cC^{H^{-}}}\leq C\}$ we bound the factor by $C$, and the resulting contribution of $C\|Z_{\cdot\wedge\tau}\|_{\scC^0_p([U,T])}$ as above.
We therefore get
\begin{equs}
\Big[ &\int_{s_0}^{\cdot} \big(b(\tilde{X}_r)-b(\tilde{X}_r-Z_{r\wedge \tau})\big)\,dr  \Big]_{\scC_p^{1-H}([U,T])} \lesssim C\|Z_{U\wedge\tau}\|_{L_p(\Omega)}|T-U|^{\alpha H}
\\
&\quad+ [Z_{\cdot \wedge \tau}]_{\scC^{H^-}_p([U,T])} |T-U|^{2H_--H+\alpha H}
+C [Z_{\cdot \wedge \tau}]_{\scC^{H^-}_p([U,T])} |T-U|^{H^-+2H_--H+\alpha H}
\\
&\quad +\big( \P ([ B^H ]_{\cC^{H^-}|\bF}\geq C ) \big)^{1/2p}|T-U|^{2H_--H+\alpha H}.\label{eq:roughbuckD}
\end{equs}
The only information about the four exponents on the right-hand side that we care about at this point is that they are nonnegative, positive, bigger than $1/2$, and nonnegative, respectively. Therefore in the sequel we bound them from below by $0$, $\eps_0$, $1/2+\eps_0$, and $0$, respectively, where $\eps_0>0$.
Further, by \eqref{eq:triv-controll-embedding} and the definition of $\tau$, clearly $[Z_{\cdot \wedge \tau}]_{\scC^{H^-}_p([U,T])}\lesssim\big\| (Z_{\cdot \wedge \tau},Z'_{\cdot\wedge\tau}) \big\|_{\cD_K^{1-H_-}([U,T])}$.
Putting \eqref{eq:roughbuckE}-\eqref{eq:roughbuckS}-\eqref{eq:roughbuckD} together we get that with some $N_0\lesssim 1$
\begin{equs}
\big\| (Z_{\cdot \wedge \tau},Z'_{\cdot\wedge\tau}) \big\|_{L_p(\Omega;\scD_K^{1-H_-}([U,T]))}&\leq N_0 C\|Z_{U\wedge\tau}\|_{L_p(\Omega)}+ 2\|b-\tilde b\|_{\cC^0(\R^d)}
\\
&\quad+N_0\big( \P ([ B^H ]_{\cC^{H^-}|\bF}\geq C ) \big)^{1/2p}
\\
&\quad+N_0\big\| (Z_{\cdot \wedge \tau},Z'_{\cdot\wedge\tau}) \big\|_{L_p(\Omega;\scD_K^{1-H_-}([U,T]))}|T-U|^{\eps_0}
\\&\quad+N_0 C\big\| (Z_{\cdot \wedge \tau},Z'_{\cdot\wedge\tau}) \big\|_{L_p(\Omega;\scD_K^{1-H_-}([U,T]))}|T-U|^{1/2+\eps_0}.
\end{equs}
If $|T-U|^{1/2+\eps_0}\leq(4 N_0 C)^{-1}$, then for large enough $C$ this also implies $|T-U|^{\eps_0}\leq(4 N_0)^{-1}$, and therefore the inequality buckles. The remaining argument from the proof of Lemma \ref{lem:stability-rough} (following \eqref{eq:iterable}) carries through unchanged.
\end{proof}

\begin{remark}\label{rem:obstacle}
The condition $\alpha>0$ can be traced from the above proof. We buckle the difference in $\scD_K^{1-H_-}$, which is the minimal regularity required to make sense of the rough integral. On the other hand, Lemma \ref{lem:Regularisation-rough} provides a bound on drift difference in $\scC_p^{1+(\alpha-1)H}$, leading to the requirement $1+(\alpha-1)H>1-H_-$.
\end{remark}

The stability of solutions with respect to their initial conditions follows from Lemma \ref{lem:stability-rough}, by following the proof of Corollary \ref{cor:initial-stability} without any change.

\begin{corollary}\label{cor:initial-stability-rough}
Let  $p \geq 1$, $s_0\in[0,1)$, $x_0,y_0\in\R^d$, and let $b$ and $\sigma$ satisfy $\mathbf{A}_{\mathrm{rough}}$.  For any $\theta \in (0, 1)$, there exists a constant $N=N(\theta,\alpha,H,d,d_0,M,\lambda,p,K)$
with the following property:
if  $X$ and $Y$ are strong solutions of  \eqref{eq:main}  with initial conditions $x_0$ and $y_0$, respectively, then 
\begin{equs}\label{eq:initial-almost-Lipschitz-rough}
\big\| \big(X_{\cdot \wedge\tau }- &Y_{\cdot \wedge\tau },\sigma(X_{\cdot\wedge\tau})-\sigma(Y_{\cdot\wedge\tau})\big)\big\|_{L_p(\Omega;\scD_K^{1-H_-})}    \leq  N |x_0-y_0|^\theta. 
\end{equs}
\end{corollary}

\subsection{Strong existence and the semiflow}

\begin{theorem}\label{thm:exist-rough}
Let   $s_0\in[0,1)$, $x_0\in\R^d$, and let $b$ and $\sigma$ satisfy $\mathbf{A}_{\mathrm{rough}}$. Then there exists a strong solution $X^{s_0,x_0}$ of \eqref{eq:main}.
Moreover, the random field $X$ given by
\begin{equ}
\big(\mathbb{Q}\cap[0,1)\big)\times\R^d\times[0,1]\ni (s_0,x_0,t)\mapsto X^{s_0,x_0}_t\in\R^d
\end{equ}
has a modification $\what X$ such that with an event $\what \Omega\subset\Omega_H$ of full probability, $\what X$ satisfies (i)-(ii)-(iii) from Theorem \ref{thm:exist-Young} for all $\omega\in\what \Omega$:
\end{theorem}

\begin{remark}
Note that although we require the same properties (i)-(ii)-(iii) from the semiflow $\what X$ in the Young and the rough case, the \emph{meaning} of (i) (namely, what being a solution of \eqref{eq:main} entails) is different, c.f. Definition \ref{def:soln}.
\end{remark}

\begin{proof}
The existence of a strong solution for fixed initial time and initial condition is obtained as in the proof of Theorem \ref{thm:exist-Young}: taking a sequence $(b^n)_{n\in\N}\subset \cC^1$ converging uniformly to $b$, the classically well-posed rough differential equations
\begin{equ}
X^n_t=x_0+\int_{s_0}^t b^n(X^n_s)\,ds+\int_{s_0}^t\sigma(X^n_s)\,dB^H_s,
\end{equ}
and their solutions $X^n$, the application of Lemma \ref{lem:stability-rough} yields that for any $K\in\N$ the sequence $\big(X^n_{\cdot\wedge\tau_K},\sigma(X^n_{\cdot\wedge\tau_K})\big)_{n\in\N}$ is Cauchy with respect to the norm $L_p(\Omega;\scD_K^{1-H_-})$.
In particular, by the completeness of the space, $\big\|\big(X^n- X,\sigma(X^n)-\sigma(X)\big)\big\|_{\cD^{1-H_-}_{B^H}}\to 0$ in probability for some adapted process $X$ such that $\big(X,\sigma(X)\big)\in \cD^{1-H_-}_{B^H}$ a.s. As in the proof of Theorem \ref{thm:exist-Young}, it is straightforward to verify that $X$ is a strong solution to \eqref{eq:main}.

As for the semiflow, we proceed slightly differently, due to the $\omega$-dependence of the solution space $\cD^{1-H_-}_{B^H}$. From Corollary \eqref{cor:initial-stability-rough} we get
that for all $K\in\N$, $\bar\theta\in(0,1)$ one has uniformly over $x_0,y_0\in\R^d$ the bound
\begin{equ}
\big\|\|X^{s_0,x_0}_{\cdot\wedge\tau_K}-X^{s_0,y_0}_{\cdot\wedge\tau_K}\|_{\cC^{H^-}([s_0,1])}\big\|_{L_p(\Omega)}\lesssim |x_0-y_0|^{\bar\theta}.
\end{equ}
Therefore, by applying Kolmogorov's continuity criterion, we get a modification $\what X$ of $X$: a random field $(\what X^{s_0,x_0})_{x_0\in\R^d}$ that takes values in $\cC^{H^-}([s_0,1])$ and such that for any $\theta\in(0,\bar\theta-1/p)$, satisfies for all $K$ and $R$
\begin{equ}\label{eq:semiflow bound rough}
\Bigg\|\sup_{\substack{x_0\neq y_0\\|x_0|,|y_0|\leq R}}\frac{\|\what X^{s_0,x_0}_{\cdot\wedge\tau_K}-\what X^{s_0,y_0}_{\cdot\wedge\tau_K}\|_{\cC^{H^-}([s_0,1])}}{|x_0-y_0|^\theta}\Bigg\|_{L_p(\Omega)}\lesssim 1.
\end{equ}
Verifying (ii) and (iii) goes just as in the proof of Theorem \ref{thm:exist-Young}. Concerning (i), take  the event  of full probability on which for all $x_0\in\Q^d$, $\what X^{s_0,x_0}$ coincides with $X^{s_0,x_0}$ (in particular, is a solution) and the random variable on the left-hand side of \eqref{eq:semiflow bound rough} is finite for all $K$ and $R$.
We then want to show that on this event for any $y_0\in\R^d$, $\what X^{s_0,y_0}$ is also a solution. Let $K>\big[(B^H,\bB^H)\big]_{\cR^{H^-}([0,1])}$, $R>|y_0|+1$, and let $(x_0^n)_{n\in\N}\subset\Q^d$ such that $x_0^n\to y_0$. By Lemma \ref{lem:apriori-rough} and \eqref{eq:semiflow bound rough} we have
\begin{equ}
\sup_{n\in\N}\big\|\big(\what X^{s_0,x_0^n},\sigma(\what X^{s_0,x_0^n})\big)\big\|_{\cD^{2H_-}_{B^H}}\lesssim 1, \qquad \|\what X^{s_0,x_0^n}-\what X^{s_0,y_0}\|_{\cC^{H^-}}\to 0.
\end{equ}
This implies 
\begin{equ}
\big\|\big(\what X^{s_0,x_0^n}-\what X^{s_0,y_0},\sigma(\what X^{s_0,x_0^n})-\sigma(\what X^{s_0,y_0})\big)\big\|_{\cD^{2H_--\eps}_{B^H}}\to 0
\end{equ}
for any $\eps>0$,
which is enough to pass to the limit and conclude that $\what X^{s_0,y_0}$ is a solution.
\end{proof}

\subsection{Path-by-path uniqueness}

\begin{theorem}\label{PBP-uniq-rough}
Let $b$, $\sigma$, and $\bB^H$ satisfy $\mathbf{A}_{\mathrm{rough}}$ and take $\what\Omega\subset\Omega_H$ obtained from Theorem \ref{thm:exist-rough}.
Let $\omega\in \what\Omega$, $s_0\in[0,1)$, $x_0\in\R^d$, and let $Y,Y'$ be solutions of \eqref{eq:main}. Then $Y=Y'$.
\end{theorem}
\begin{proof}
The proof is similar to that of Theorem \ref{thm:PBP-uniq-Young}, with a bit more care due to the rough integrals.
Again it suffices to consider the case $s_0\in\Q$ and to show that $Y=\what X^{s_0,x_0}(\omega)$. We drop $\omega$ from the notation for the rest of the proof.
For an arbitrary $t\in(s_0,1]\cap \Q$ and for $s\in[s_0,t]\cap\Q$ we define $f(s)=\what X^{s,Y_s}_t$, and since $f(s_0)=\what X^{s_0,x_0}_t$ and $f(t)=Y_t$, it suffices to show that $f$ is constant. Precisely as in \eqref{eq:PBP-1}, for arbitrary $\theta\in(0,1)$ we get
\begin{equ}\label{eq:PBP-rough}
|f(r)-f(u)|\lesssim |Y_r-\what X^{u,Y_u}_r|^\theta.
\end{equ}
By the definition of a solution, we have
\begin{equ}
\big(Y,\sigma(Y)\big), \big(\what X^{u,Y_u},\sigma(\what X^{u,Y_u})\big)\in \cD_{B^H}^\gamma([u,r])
\end{equ} 
for some $\gamma>1-H$ and therefore also
\begin{equ}
\big(\sigma (Y),\nabla\sigma\sigma(Y)\big), \big(\sigma(\what X^{u,Y_u}),\nabla\sigma\sigma(\what X^{u,Y_u})\big)\in \cD_{B^H}^\gamma([u,r]).
\end{equ} 
Moreover, the two processes (including their Gubinelli derivatives) coincide at time $u$. Let us fix $\bar H\in(1-\gamma,H)$ and recall that $(B^H, \mathbf{B}^H)\in \cR^{\bar H}$.
Therefore by rough integration we have for all $v\in[u,r]$
\begin{equs}
\Big|\int_u^v \sigma(Y_s)-\sigma(\what X^{u,Y_u}_s)\,dB^H_s\Big|
&=\Big|\int_u^v \sigma(Y_s)-\sigma(\what X^{u,Y_u}_s)\,dB^H_s
\\
&\qquad-\big(\sigma(Y_u)-\sigma(\what X^{u,Y_u}_u)\big)(B^H_v-B^H_u)
\\
&\qquad-\big(\nabla\sigma\sigma(Y_u)-\nabla\sigma\sigma(\what X^{u,Y_u}_u)\big)\cdot\mathbf{B}^H_{u,v}\Big|\lesssim |v-u|^{\gamma+\bar H}.
\end{equs}
From here the argument concludes as the proof of Theorem \ref{thm:PBP-uniq-Young}: the drift difference
\begin{equ}\label{PBP-2}
\Big|\int_u^v b(Y_s)-b(\what X^{u,Y_u}_s)\,ds\Big|
\end{equ}
can first be bounded by $|u-v|$, and since that implies $|Y_s-\what X^{u,Y_u}_s|\lesssim |s-u|$, combined with the $\alpha$-H\"older continuity of $b$, we can in fact bound the drift difference \eqref{PBP-2} by $|v-u|^{1+\alpha}$.
We therefore have $|Y_s-\what X^{u,Y_u}_s|\lesssim |s-u|^{(1+\alpha)\wedge(\gamma+\bar H)}$, and so if one chooses $\theta$ sufficiently close to $1$, \eqref{eq:PBP-rough} shows that $f$ is constant.
\end{proof}

\subsection{Weak existence for distributional drift}\label{sec:weak}

The purpose of this section is to prove Theorem \ref{thm:weak}.
Unlike in the rest of the article, here we do not need to compare solutions from different initial times and initial conditions, so for simplicity
we take $s_0=0$, $x_0=0$.
Let us also recall the main regularity condition in Theorem \ref{thm:weak}:
in this section we assume  that 
\begin{equs}        \label{eq:cond-alpha-weak}
2-\frac{1}{H}>\alpha >\frac{1}{2}- \frac{1}{2H},
\end{equs}
where the first inequality can be freely assumed for convenience since lowering $\alpha$ is never a loss of generality.
Fix
\begin{equ}
1/3<H_-<H^-<H<H^+
\end{equ}
such that $2H_->1+\alpha H^+$, $\alpha H^++2H_--H^->0$, and $1+\alpha H+\alpha H^+-H>0$.
First, we define what the integral of a distribution along certain paths is.
Given a continuous process $X$ and $f\in\cC^1$ let us define the integral process $\cI^X(f)$ by
\begin{equ}
\cI_t^X(f)=\int_0^t f(X_s)\,ds.
\end{equ}
To make sense of the equation, we wish to extend $\cI^X$ to $\cC^\alpha$ for negative $\alpha$.
This relies on the following estimate.
\begin{remark}\label{rem:SSL-mod}
In the proof we will apply stochastic sewing in the following somewhat unusual way.
In \cite{Khoa} and all subsequent variations all bounds are required to hold in some $L_p(\Omega)$ norm with $p\geq 2$. Note however, that this is only really necessary for the martingale part. The remaining terms can be treated with almost sure arguments, which will be quite convenient for our purposes.
\end{remark}
\begin{lemma}\label{lem:weak-integral}
Fix $\lambda>0$, $K<\infty$.
Assume that $D$ is an adapted process belonging to $\cC^{1+\alpha H^+}$ a.s., 
$(S,S')$ is an adapted process belonging to $\cD^{2H^-}_{B^H}$ a.s., and that $\|S'\|\leq K$ and $S'S'^\top\succeq\lambda\I$. Set $X=D+S$.
Then there exists $\eps,\gamma>0$ depending only on $\alpha,H$ such that for any $f\in\cC^1$ there exists a random variable $\xi \in \cap_{p \geq 1} L_p(\Omega)$ such that a.s. for all $(s,t)\in[0,1]^2_\leq$ one has the bound
\begin{equs}
|\cI^X_t&(f)-\cI^X_s(f)|\leq \|f\|_{\cC^\alpha} \xi(t-s)^{1+\alpha H^+}
\\
&+N\|f\|_{\cC^\alpha}
\big([D]_{\cC^{1+\alpha H^+}([s,t])}
+[(S,S')]_{\cD^{2H_-}_{B^H}([s,t])}+[B^H]_{\cC^{H^-}|\bF}^\gamma\big)(t-s)^{1+\eps}.\label{eq:weak-main-bound}
\end{equs}
The constant $N$, as well as any $L_p(\Omega)$ norm of $\xi$, $p\geq 2$, depends only on $\alpha,H,d,d_0,\lambda,K,p$.
\end{lemma}

\begin{proof}
We may and will assume $\|f\|_{\cC^\alpha} =1$.
It is trivial by continuity that it suffices to prove \eqref{eq:weak-main-bound} for rational points.
Somewhat less trivial, but follows from invoking the classical argument in Kolmogorov's continuity theorem (see e.g. \cite[Theorem~3.1]{Friz-Hairer}), that it is in fact sufficient to consider the case
\begin{equ}
(s,t)\in\Pi:=\{\big((i-1)2^{-n},i 2^{-n}\big):\,n\in\N, i=1,\ldots, 2^{n}\}.
\end{equ}
Therefore, take $(U,T)\in\Pi$. For $(s,t)\in[0,1]_{\leq}^2$ we set
\begin{equs}  
A_{s, t}&= \E^s \int _s^t  f \big(X_s + S'_s (B^H_r-B^H_s)\big) \, dr
\\
&=\int _s^t  \cP_{ \Gamma(s,s,r)}f \big( X_s + S_s' (\E^s B^H_r-B^H_s)\big) \, dr,
\end{equs}
with the notation $\Gamma(s, u, r):= S'_s S'^\top_s(r-u)^{2H}$.
On the probability $1$ event where $X\in\cC^{H^-}$ and $[B^H]_{\cC^{H^-}|\bF}<\infty$, $A_{s,t}$ reconstructs $\cI^X(f)$, i.e. for any sequence of partitions $(\pi_n)_{n\in\N}=\big(\{U=t_0^n\leq\cdots \leq t_{k(n)}^n=T\}\big)_{n\in\N}$ of $[U,T]$ such that $|\pi_n|:=\min\{t_{i}-t_{i-1}:\,i=1,\ldots,k(n)\}\to 0$,
one has
\begin{equ}\label{eq:weak-reconstruct}
\cI_T^X(f)-\cI_U^X(f)=\lim_{n\to\infty}\sum_{i=1}^{k(n)} A_{t_{i-1}^n,t_i^n}.
\end{equ}
We may as well simply take $\pi_n$ to be the uniform partition with meshsize $(T-U)2^{-n}$.
By \eqref{eq:weak-reconstruct} we have
\begin{equ}
|\cI_T^X(f)-\cI_U^X(f)|\leq |A_{U,T}|+\big|\sum_{n\in\N}\sum \delta A_{s,u,t}-\E^s\delta A_{s,u,t}\big|+\sum_{n\in\N}\sum\big|\E^s\delta A_{s,u,t}\big|.
\end{equ}
The inner summation is over all $(s, t)$ such that $s<t$  are consecutive points in $\pi_n$ and $u=(t+s)/2$.  For the  sake of not further overcrowding notation,  the index set will be dropped. 
By Proposition \ref{prop:HK-Gamma}, we have
\begin{equs}                                                   \label{eq:est-Ast-extension-of-I}
| A_{s,t} |  \leq \int_s^t \|\cP_{ \Gamma(s,s,r)}f \|_{\cC^0} \, dr \lesssim   (t-s)^{1+\alpha H}.
\end{equs}
Since $\alpha>-1/2H$, the usual stochastic sewing arguments (see \cite[Section~2]{Khoa}) show that the limit
\begin{equ}\label{eq:Zdef}
Z_{U,T}:=\sum_{n\in\N}\sum \delta A_{s,u,t}-\E^s\delta A_{s,u,t}
\end{equ}
exists in $L_p(\Omega)$ for any $p\geq 2$ and satisfies
\begin{equ}\label{eq:firstZ}
\|Z_{U,T}\|_{L_p(\Omega)}\lesssim (T-U)^{1+\alpha H}.
\end{equ}
It is again a consequence of the proof of Kolmogorov's continuity theorem that the fact that \eqref{eq:firstZ} holds for all $p\geq 2$ implies that the random variable
\begin{equ}
\xi_0:=\sup_{(s,t)\in\Pi}\frac{|Z_{s,t}|}{(t-s)^{1+\alpha H^+}}
\end{equ}
is a.s. finite\footnote{Note that at this point we do \emph{not} claim an a.s. $\cC^{1+\alpha H^+}_2$ modification of $Z$. That would require bounding $Z_{s,t}-Z_{s,v}-Z_{u,t}+Z_{u,v}$ for generic time-quartuples \cite[Lemma A.1]{Mate19}, which is far more demanding.} and satisfies $\|\xi\|_{L_p(\Omega)}\lesssim 1$ for all $p\geq 2$.
Denoting $\xi=N+\xi_0$, by definition we have the bound
\begin{equ}\label{eq:weak-intermediate}
|\cI_T^X(f)-\cI_U^X(f)|\leq \xi|T-U|^{1+\alpha H^+}+\sum_{n\in\N}\sum\big|\E^s\delta A_{s,u,t}\big|
\end{equ}
for all $U,T$.
It remains to estimate $\E^s \delta A_{s,u,t}$.
We have 
\begin{equs}
\E^s \delta A_{s,u,t} & = \E^s \int _u^t  \cP_{ \Gamma(s, u, r)}f \big(X_s +  S_s' (\E^u B^H_r-B^H_s)\big) \, dr. 
\\
& \qquad  -  \E^s \int _u^t  \cP_{\Gamma(u, u, r)}f \big(X_u +  S'_u (\E^u B^H_r-B^H_u)\big) \, dr =: F_1+F_2,
\end{equs}
where
\begin{equs}
F_1 = \E^s \int _u^t  \cP_{  \Gamma(s, u, r)}f (X_s +  {S'_s} (\E^u B^H_r-B^H_s)) -  \cP_{ \Gamma(s, u, r)}f (X_u +  {S'_u} (\E^u B^H_r-B^H_u)) \, dr,
\\
F_2 = \E^s \int _u^t   \cP_{ \Gamma(s, u, r)}f (X_u +  {S'_u} (\E^u B^H_r-B^H_u)) -  \cP_{\Gamma(u, u, r)}f (X_u +  {S'_u} (\E^u B^H_r-B^H_u)) \, dr. 
\end{equs}
For $F_1$, first take $\theta=\theta(\alpha,H)\in(0,1)$ such that $H^-+\theta(2H_--H^-)\geq 1+\alpha H^+$.  By Proposition \ref{prop:HK-Gamma} we have
\begin{equs}
|F_1| & \lesssim   \int_u^t  [ \cP_{  \Gamma(s, u, r)}f]_{\cC^1} |D_s-D_u + S_s-S_u+ {S'_s}(B^H_u-B^H_s)+ ({S'_s}-{S'_u}) ( \E^u B^H_r -B_u) |  \, dr 
\\
& \lesssim  (t-u)^{1+(\alpha-1)H}
\Big(
[D]_{\cC^{1+\alpha H^+}([U,T])}(s-u)^{1+\alpha H^+}
+[ (S, S')] _{\cD^{2H_{-}}_{B^H}([U, T])}(s-u)^{2H_-}
\\
&\qquad\qquad+ [ S']_{\cC^{2H_--H^-}([U, T])}^{ \theta}  [ B^H]_{\cC^{H^-}|\bF}(t-s)^{H^-+\theta(2H_--H^-) }\Big)
\end{equs}
where we have used the fact that $(\alpha-1)H>-1$ and that  $\| S'\|_{\cC^0} \lesssim 1$. 
First we note that the smallest power of the time increments in the big brackets is $1+\alpha H^+$.
Next, we use Young's inequality $a^\theta b\lesssim a+b^{(1-\theta)^{-1}}$, so with  $\gamma=(1-\theta)^{-1}$ we get
\begin{equ}
|F_1|\lesssim (t-s)^{2+\alpha H+\alpha H^+-H}\Big([D]_{\cC^{1+\alpha H^+}([U,T])}
+[ (S, S')] _{\cD^{2H_{-}}_{B^H}([U, T])}+ [ B^H]_{\cC^{H^-}|\bF}^\gamma\Big).\label{eq:F1-weak}
\end{equ}
For $F_2$ we have 
\begin{equs}
|F_2|&  \lesssim  \int_u^t \|\Gamma(s, u, r)- \Gamma(u, u, r) \| (r-u)^{-2H+ \alpha H} \, dr 
\\
& \lesssim  \int_u^t \|  S'_s S'^\top_s- S'_u S'^\top_u \| (r-u)^{\alpha H} \, dr 
\\
& \lesssim (t-u)^{1+\alpha H} (u-s)^{2H_--H^-} [ S']_{\cC^{2H_--H^-}([U, T])}
\\
& \lesssim (t-s)^{1+\alpha H +2H_--H^-} [ (S,S')]_{\cD^{2H_-}_{B^H}([U, T])}.\label{eq:F2-weak}
\end{equs}
Here, similarly to before, we have used $\alpha H > -1$ and $\| S'\|_{\cC^0} \lesssim 1$.
It remains to put \eqref{eq:F1-weak} and \eqref{eq:F2-weak} together and notice that by assumption both exponents $2+\alpha H+\alpha H^+-H$ and $1+\alpha H +2H_--H^-$ are greater than $1$.
Therefore, with some $\eps>0$ one has the bound
\begin{equ}
|\E^s\delta A_{s,u,t}|\lesssim (t-s)^{1+\eps}\Big([D]_{\cC^{1+\alpha H^+}([U,T])}
+[ (S, S')] _{\cD^{2H_{-}}_{B^H}([U, T])}+ [ B^H]_{\cC^{H^-}|\bF}^\gamma\Big).
\end{equ}
Using this bound in \eqref{eq:weak-intermediate} in the standard way (let us again refer the reader to \cite{Khoa} for more details) we get the claimed bound \eqref{eq:weak-main-bound}.
\end{proof}

We formulate a few corollaries of Lemma \ref{lem:weak-integral}, the first of which is the claimed extension of $\cI^X$.
Denote by $\overline{C^\alpha}$ the closure of $C^1$ in $C^\alpha$. Recall that for any $\alpha<\alpha'$ one has $C^{\alpha'}\subset\overline{C^\alpha}$
\begin{corollary}\label{cor:distr-integral}
Take $X,D,S$ as in Lemma \ref{lem:weak-integral} and suppose furthermore that for some $p\geq 1$ and $M<\infty$
\begin{equ}\label{eq:wantbound}
\big\|[D]_{\cC^{1+\alpha H^+}([0,1])}
+[(S,S')]_{\cD^{2H_-}_{B^H}([0,1])}\big\|_{L_p(\Omega)}\leq M.
\end{equ}
Then $f\mapsto\cI^X(f)$ extends as a continuous linear map from $\overline{C^\alpha}$ to $\scC_p^{1+\alpha H^+}$ whose norm depends only on $M, \alpha,H,d,d_0,\lambda,K,p$.
\end{corollary}
The other important consequence is that we obtain a priori bounds for solutions to \eqref{eq:main} with drift that is regular, but using only its distributional norm in the bounds.
\begin{corollary}\label{cor:weak-apriori}
Take $X,D,S$ as in Lemma \ref{lem:weak-integral}, $b\in \cC^1$, $\sigma\in\cC^2$, and suppose furthermore that
\begin{equ}
D_t=\int_0^t b(X_s)\,ds,\qquad S_t=\int_0^t \sigma(X_t)\,dB^H_t,\qquad S'_t=\sigma(X_t);
\end{equ}
with the latter integral understood in the rough sense. Then for any $p\geq 2$
there exist a constant $N$ depending on $\alpha,H,d,d_0,\lambda,p,\|b\|_{\cC^\alpha},\|\sigma\|_{\cC^2}$, such that
\begin{equ}\label{eq:havebound}
\big\|[D]_{\cC^{1+\alpha H^+}([s,t])}
+[(S,\sigma(X))]_{\cD^{2H_-}_{B^H}([s,t])}\big\|_{L_p(\Omega)}\leq N.
\end{equ}
\end{corollary}
\begin{proof}
Since $\cI^X(b)=D$, we can apply \eqref{eq:weak-main-bound} and buckle with respect to $D$. 
We get for any $s,t$
\begin{equ}\label{eq:Dbound}
\,[D]_{\cC^{1+\alpha H^+}([s,t])}
\lesssim \xi+[(S,\sigma(X))]_{\cD^{2H_-}_{B^H}([s,t])}+[B^H]_{\cC^{H^-}|\bF}^\gamma.
\end{equ}
Concerning $S$, with by now usual arguments (that is, by \eqref{eq:controll-increasing}, \eqref{eq:rough continuity} keeping in mind that $1+\alpha H^++H^->1)$, \eqref{eq:quadratic-controled-estimate} with $\beta=1$, triangle inequality, and \eqref{eq:Dbound}) we have for any $U,T$ 
\begin{equs}
\,[&(S,\sigma(X))]_{\cD^{2H_-}_{B^H}([U,T])}
\\
&\leq [( S,\sigma(X))]_{\cD^{2H^-}_{ B^H}([U,T])}|T-U|^{2(H^--H_-)}
\\
&\lesssim  [(\sigma(X),\nabla\sigma\sigma(X))]_{\cD^{1+\alpha H^+}_{B^H}([U,T])}[(B^H,\bB^H)]_{\cR^{H^-}}|T-U|^{2(H^--H_-)}
\\
&\lesssim [(X,\sigma(X))]_{\cD^{1+\alpha H^+}_{B^H}([U,T])}\big(1+[(B^H,\bB^H)]_{\cR^{H^-}}^3\big)|T-U|^{2(H^--H_-)}
\\
&\leq \big([D]_{\cC^{1+\alpha H^+}([U,T])}+[( S,\sigma(X))]_{\cD^{2H_-}_{ B^H}([U,T])}\big)\big(1+[(B^H,\bB^H)]_{\cR^{H^-}}^3\big)|T-U|^{2(H^--H_-)}
\\
&\lesssim \big(\xi+[( S,\sigma(X))]_{\cD^{2H_-}_{ B^H}([U,T])}+[B^H]_{\cC^{H^-}|\bF}^\gamma\big)\big(1+[(B^H,\bB^H)]_{\cR^{H^-}}^3\big)|T-U|^{2(H^--H_-)}.
\end{equs}
On small enough intervals the inequality buckles. That is, with $\tilde\gamma=3/(2(H^--H_-))$, we can take an equidistant partition $\{0=U_0<U_1<\ldots U_L=1\}$ of $[0,1]$ such that $L\lesssim 1+[(B^H,\bB^H)]_{\cR^{H^-}}^{\tilde\gamma}$ and such that for each $i=1,\ldots,L$ one has
\begin{equ}
\,[( S,\sigma(X))]_{\cD^{2H_-}_{ B^H}([U_{i-1},U_i])}\lesssim \xi+[B^H]_{\cC^{H^-}|\bF}^\gamma.
\end{equ}
Applying the generalised subadditivity rule \eqref{eq:sum_partition} for the $\cD^{2H_-}_{ B^H}$ norm, we get
\begin{equs}
\,[( S,\sigma(X))]_{\cD^{2H_-}_{ B^H}([0,1])}&\lesssim L\big(\xi+[B^H]_{\cC^{H^-}|\bF}^\gamma\big)\big(1+[B^H]_{\cC^{H^-}}\big)
\\
&\lesssim \big(1+[(B^H,\bB^H)]_{\cR^{H^-}}^{\tilde\gamma}\big)\big(\xi+[B^H]_{\cC^{H^-}|\bF}^\gamma\big)\big(1+[B^H]_{\cC^{H^-}}\big).
\end{equs}
Taking $L_p(\Omega)$ norms in here as well as in \eqref{eq:Dbound} finishes the proof.
\end{proof}

\begin{proof}[Proof of Theorem \ref{thm:weak}]
We will only prove a slightly weaker statement, with $\cR^{H^-}$ instead of $\cap_{\eps>0}\cR^{H-\eps}$, we leave the general case as an exercise.

For the present proof we need yet another two exponents $H_{--}$ and $H^{++}$. We require them to satisfy $H_{--}<H_-$, $H^{++}>H^+$, $\alpha H^{++}+H_{--}>0$, and ........
Since the condition \eqref{eq:cond-alpha-weak} is with a strict inequality, after sacrificing a small amount of regularity we may and will assume that $b\in\overline{\cC^\alpha}$.
In particular, we can take a sequence $(b^n)_{n\in\N}\subset\cC^1$ converging to $b$ in $\cC^\alpha$
and consider the the classically well-posed rough differential equations
\begin{equ}
X^n_t=\int_{0}^t b^n(X^n_s)\,ds+\int_{0}^t\sigma(X^n_s)\,dB^H_s.
\end{equ}
Denote the first and second integral by $D^n$ and $S^n$, respectively.
Applying Corollary \ref{cor:weak-apriori}, we get
\begin{equ}\label{eq:weak-bound-beforelim}
\sup_{n\in\N}\big\|[D^n]_{\cC^{1+\alpha H^+}([0,1])}
+[(S^n,\sigma(X^n))]_{\cD^{2H_-}_{B^H}([0,1])}\big\|_{L_p(\Omega)}\leq N.
\end{equ}
To conclude tightness on a separable space, first note that since $D^n\in\cC^1$, the laws of $D^n$ are tight on the (separable) space $\overline{\cC^{1+\alpha H^{++}}}$.
Next, it follows from \cite[Exercises~4.8-4.9]{Friz-Hairer} that if $\beta\in(1/3,1/2)$, equipping the (nonlinear) space $\cQ^{\beta}$ of quartuples $(\hat B^H,\hat\bB^H,\hat S,\hat \Sigma)\in\cC^{\beta}\times\cC^{2\beta}_2\times\cC^{\beta}\times\cC^{\beta}$,
constrained by Chen's relation between the first two elements and by the relation $(\hat S,\hat\Sigma)\in\cD_{\hat B^H}^{2\beta}$,
with the metric
\begin{equs}
q&\big((\hat B^{H,1},\hat\bB^{H,1},\hat S^1,\hat \Sigma^1),(\hat B^{H,2},\hat\bB^{H,2},\hat S^2,\hat \Sigma^2)\big)
\\
&=[\hat B^{H,1}-\hat B^{H,2}]_{\cC^{\beta}}
+[\hat \bB^{H,1}-\hat \bB^{H,2}]_{\cC^{2\beta}_2}
+[\hat\Sigma^1-\hat\Sigma^2]_{\cC^{\beta}}
\\
&\quad+\sup_{0\leq s<t\leq 1}\frac{\big(\hat S^1_t-\hat S^1_s-\hat\Sigma^1_s (B^{H,1}_t-B^{H,1}_s)\big)-\big(\hat S^2_t-\hat S^2_s-\hat\Sigma^2_s (B^{H,2}_t-B^{H,2}_s)\big)}{|t-s|^{2\beta}},
\end{equs}
the closure $\overline{\cQ^{\beta}}$ of smooth quartuples in $\cQ^{\beta}$ is separable and contains all $\cQ^{\beta'}$, $\beta'>\beta$.
It is also easy to see that these inclusions are compact.
Therefore, we can conclude that the family of the laws of the septuples $Z^n=\big(X^n,D^n,W,B^H,\bB^H,S^n,\sigma(X^n)\big)$ is tight on
\begin{equ}
\cZ:=\overline{\cC^{H_-}}\times \overline{\cC^{1+\alpha H^{++}}}\times\overline{\cC^{H}}\times \overline{\cQ^{H_{--}}}.
\end{equ}
Therefore by Prokhorov's and Skorohod's theorems there exists a subsequence (for which we use the same notation), a probability space $(\bar\Omega,\bar\cF,\bar\P)$, and random variables $\bar Z^n$ such that $\bar Z^n\overset{\mathrm{law}}{=} Z^n$ and $\bar Z^n$ converges $\bar\P$-a.s. in $\cZ$ to some $\bar Z$.
It is clear (even only using pointwise convergence) that $\bar\Sigma=\sigma(\bar X)$ and that $\bar X=\bar D+\bar S$.
By stability of rough integration (\cite[Theorem~4.17]{Friz-Hairer}) and the condition on the exponents we also get from the convergence that 
\begin{equ}
\bar S=\lim_{n\to\infty} \bar S^n
=\lim_{n\to \infty} \int \sigma(\bar X^n_t)\,d\bar B^{H,n}_t
=\int\sigma(\bar X_t)\,d\bar B^{H}_t.
\end{equ}
Define the filtration $\bar\bF$ by setting $\cF_t$ to be the completion of the $\sigma$-algebra generated by $\{\bar X_s,\bar D_s,\bar W_s,\bar\bB^H_{u,s}:\,u,s\leq t\}$.
It is easy to check (see e.g \cite[Section~7.3]{ART21} or \cite[Theorem~8.2]{Lucio-Mate}) that $\bar W$ is a $\bar\bF$-Wiener process and that $\bar B^H$ satisfies the Mandelbrot -- van Ness representation based on $\bar W$.
Moreover, passing to the limit in the bound \eqref{eq:weak-bound-beforelim}, we have the $\bar D,\bar S,\bar S'=\sigma(\bar X)$ satisfy the condition \eqref{eq:wantbound}.
Since all the required adaptedness properties are satisfied by definition, by Corollary \ref{cor:distr-integral} the process $\cI^{\bar X}(b)$ is well-defined.
All that remains to check is that $\bar D=\cI^{\bar X}(b)$, or in other words, $\cI^{\bar X^n}(b^n)\to\cI^{\bar X}(b)$.
To this end, we write
\begin{equ}
\cI^{\bar X^n}(b^n)-\cI^{\bar X}(b)=\big(\cI^{\bar X^n}(b^n)-\cI^{\bar X^n}(b^m)\big)+\big(\cI^{\bar X^n}(b^m)-\cI^{\bar X}(b^m)\big)+\big(\cI^{\bar X}(b^m)-\cI^{\bar X}(b)\big).
\end{equ}
If $n$ and $m$ are large enough, the first and third terms are small in $L_p(\Omega)$ thanks to Corollary \ref{cor:distr-integral} and the fact that 
$\bar D^n,\bar S^n,(\bar S^n)'=\sigma(\bar X^n)$ satisfy the condition \eqref{eq:wantbound} uniformly in $n$. The second term can be made small for any $m$ by choosing $n$ sufficiently large, simply using the convergence $\bar X^n\to\bar X$ in $\cC^{H_-}$ and the Lipschitzness of $b^m$. This finishes the proof.
\end{proof}

\section{The smooth case}
Throughout the section we fix $H\in(1,\infty)\setminus\N$ and $\alpha\in(1-1/(2H),1)$. Furthermore, fix $H^{-}\in(\lfloor H \rfloor,H)$ such that $1+\alpha H^{-}+(\alpha-2)H>0$ and $1+\alpha H^->H^-$.

\subsection{A priori bounds}
As far as the the integration of the stochastic term goes, the classical regime is of course the easiest. However, the a priori bounds are more delicate.
As observed in \cite{Mate}, estimating the processes in classical H\"older spaces is not sufficient, one should rather work with a type of `stochastic regularity' that allows one to control conditional increments. 

Similarly to before, for $K\in\N$ we define the stopping times
\begin{equs} 
\tau_K= \inf\{ t >0 : [B^H]_{\cC^{H^-}([0, t])}  \geq  K\}. 
\end{equs}
Not that by definition of $B^H$ and the $\cC^{H^{-}}$ norm, we have $[B^{H-m}_{\cdot\wedge\tau_K}]_{\cC^{H^--m}([0,1])}\leq K$ for all $m=0,1,\ldots,\lfloor H\rfloor$.
One of course has the trivial estimate $[X_{\cdot\wedge\tau_K}]_{\cC^1}\lesssim 1$ for any solution $X$.
A much higher order priori bound reads as follows. As before, we denote by $D$ and $S$ the first and second integrals in \eqref{eq:main}, respectively.
\begin{lemma}\label{lem:apriori-class}
Let $s_0\in[0,1)$, $x_0\in\R^d$, and let $b$ and $\sigma$ satisfy $\mathbf{A}_{\mathrm{smooth}}$.
Then there exists a constant $N=N(K,\alpha,H,d,d_0,M)$ such that
if $X$ is a strong solution of \eqref{eq:main}, then for all $(s,t)\in[s_0,1]^2_\leq$ one has almost surely
\begin{equ}\label{eq:apriori-class}
\|D_{t\wedge\tau_K}-\E^s D_{t\wedge\tau_K}\|_{L_\infty(\Omega)}\leq N |t-s|^{1+\alpha H^{-}},\qquad \|S_{t\wedge\tau_K}-\E^s S_{t\wedge\tau_K}\|_{L_\infty(\Omega)}\leq N|t-s|^{H^{-}}.
\end{equ}
\end{lemma}
\begin{proof}
Suppose that the bounds \eqref{eq:apriori-class} hold with some $\theta_1,\theta_2\geq 0$ in place of $1+\alpha H^{-}$, $H^{-}$, respectively.
This is certainly the case with $\theta_1=\theta_2=0$ thanks to the boundedness of $b$, $\sigma$, and (due to the stopping) of $B^{H-1}$. 
Recall that for any two bounded random variables $X$, $Y$, if $Y$ is $\cF_s$-measurable then $\|X-\E^s X\|_{L_\infty(\Omega)}\leq 2\|X-Y\|_{L_\infty(\Omega)}$.
Since the random variable $\bone_{s>\tau_K}D_{\tau_K}$ is $\cF_s$-measurable, we can therefore write
\begin{equs}
\|&D_{t\wedge\tau_K}-\E^s D_{t\wedge\tau_K}\|_{L_\infty(\Omega)}
\\
&\leq 
2\Big\|\bone_{s>\tau_K}D_{\tau_K}+\bone_{s\leq\tau_K}D_{t\wedge\tau_K}-\bone_{s>\tau_K}D_{\tau_K}-\bone_{s\leq\tau_K}\Big(\int_{s_0}^s b(X_r)\,dr-\int_s^{t\wedge\tau_K} b(\E^s X_r)\,dr\Big)\Big\|_{L_\infty(\Omega)}
\\
&=2\Big\|\bone_{s\leq\tau_K}\int_s^{t\wedge\tau_K}b(X_{r\wedge\tau_K})- b(\E^s X_{r\wedge\tau_K})\,dr\Big\|_{L_\infty(\Omega)}
\\
&\lesssim |t-s|^{1+\alpha(\theta_1\wedge\theta_2)},
\end{equs}
using the hypothesis on $\theta_1,\theta_2$ in the last step.
As for the stochastic part, we can proceed similarly to first get a bound
\begin{equs}
\|&S_{t\wedge\tau_K}-\E^s S_{t\wedge\tau_K}\|_{L_\infty(\Omega)}
\leq 2\Big\|\bone_{s\leq\tau_K}\int_s^{t\wedge\tau_K}\sigma(X_{r\wedge\tau_K})B^{H-1}_{r\wedge\tau_K}- \sigma(\E^s X_{r\wedge\tau_K})\E^s B^{H-1}_{r\wedge\tau_K}\,dr\Big\|_{L_\infty(\Omega)}.
\end{equs}
Note that 
\begin{equ}
\|B^{H-1}_{r\wedge\tau_K}-\E^s B^{H-1}_{r\wedge\tau_K}\|_{L_\infty(\Omega)}\leq 2|r-s|^{H^{-}-1}\big\|[B_{\cdot\wedge\tau_K}^{H-1}]_{C^{H^--1}([0,1])}\big\|_{L_\infty(\Omega)}\leq |r-s|^{H^{-}-1}K.
\end{equ}
This and the hypothesis on $\theta_1,\theta_2$ therefore yields
\begin{equ}
|S_{t\wedge\tau_K}-\E^s S_{t\wedge\tau_K}|\lesssim|t-s|^{\big(1+\alpha(\theta_1\wedge\theta_2)\big)\wedge H^{-}}.
\end{equ}
It is elementary to see that iterating the mapping $(\theta_1,\theta_2)\mapsto(1+\alpha(\theta_1\wedge\theta_2),\big(1+\alpha(\theta_1\wedge\theta_2)\big)\wedge H^{-})$,
starting from $(0,0)$, one reaches the value $(1+\alpha H^{-},H^{-})$ after finitely many steps (this is true in fact even under the weaker condition $\alpha>1-1/H^{-}$), finishing the proof.
\end{proof}

The other change compared to the $H<1$ case is that when approximating the stochastic integral, simply freezing the diffusion coefficient is not sufficient.
We will need a more complicated approximation, which in
turn leads to a somewhat more involved version of \eqref{eq:very-basic2}.

\begin{proposition}\label{prop:ugly-covariance}
Let $(s,t)\in[0,1]^2_\leq$ and let $Y=(Y_r)_{r\in[s,t]}$ be an $\cF_s$-measurable $\R^{d\times d_0}$-valued random process. Then conditionally on $\cF_s$, the vector
\begin{equ}
Z=\int_s^t Y_r B^{H-1}_r\,dr
\end{equ}
is Gaussian with covariance matrix
\begin{equs}
\Gamma[Y](s,t)&=\E^s\big(\big(Z-\E^s Z)\big(Z-\E^s Z)^\top\big)
\\
&=c'(H)\int_s^t\int_u^t\int_u^tY_vY_{v'}^\top|v-u|^{H-3/2}|v'-u|^{H-3/2}\,dv\,dv'\,du,\label{eq:ugly-covariance}
\end{equs}
where $c'(H)$ is a positive constant depending only on $H$.

Moreover, there exist a constant $N=N(H)$ such that for any two processes $Y^1,Y^2$ as above, one has the a.s. bound 
\begin{equ}\label{eq:ugly-cov-stability}
\big|\Gamma[Y^1](s,t)-\Gamma[Y^2](s,t)\big|\leq N|t-s|^{2H}\big(\|Y_1\|_{\cC^0([s,t])}+\|Y^2\|_{\cC^0([s,t])}\big)\|Y^1-Y^2\|_{\cC^0([s,t])}.
\end{equ}

Let furthermore $\lambda>0$, $\tilde K\in\R$, and $\bar Y$ be an $\cF_s$-measurable $\R^{d\times d_0}$-valued random variable such that
$\bar Y\bar Y^\top\succeq \lambda\I$ and $|Y|\leq \tilde K$ a.s.
Then there exist constants $c''(H)$ and $\delta=\delta(H,d,d_0,\lambda,\tilde K)$ such that if $\|Y-\bar Y\|_{\cC^0}\leq \delta$ a.s., then $\Gamma[Y](s,t)\succeq c''(H)\lambda|t-s|^{2H}\I$ a.s.
\end{proposition}

\begin{proof}
The conditional Gaussianity of $Z$ is trivial.
As for its conditional covariance matrix, we can write
\begin{equs}
Z-\E^s Z&=\int_s^t Y_{u_{\floor{H}}}\int_{s\leq u_0\leq u_1\leq\cdots \leq u_{\floor{H}}}	|u_1-u_0|^{H-\floor{H}-1/2}\,dW_{u_0}\,d u_1\cdots	\,du_{\floor{H}}
\\
&=\int_{s\leq u_0\leq u_1\leq\cdots \leq u_{\floor{H}}\leq t} Y_{u_{\floor{H}}}|u_1-u_0|^{H-\floor{H}-1/2}\,d u_1\cdots	\,du_{\floor{H}}\,dW_{u_0}
\\
&=\sqrt{c'(H)}\int_{s\leq u_0\leq u_{\floor{H}}\leq t}Y_{u_{\floor{H}}}|u_{\floor{H}}-u_0|^{H-3/2}\,du_{\floor{H}}\,dW_{u_0},
\end{equs}
where $c'(H)$ is defined by the last equality.
It\^o's isometry yields \eqref{eq:ugly-covariance}.

The bound \eqref{eq:ugly-cov-stability} follows directly from \eqref{eq:ugly-covariance}.

Concerning the last claim, it follows viewing $\bar Y$ as a constant in time process, we have
\begin{equs}
\Gamma[\bar Y](s,t)&=c'(H)\bar Y\bar Y^\top\int_s^t\int_u^t\int_u^t|v-u|^{H-3/2}|v'-u|^{H-3/2}\,dv\,dv'\,du
\\
&=c'''(H)\bar Y\bar Y^\top|t-s|^{2H}
\end{equs}
with another positive constant $c'''(H)$. This yields the claim for $\bar Y$ in place of $Y$, and to get the claim for $Y$ as stated, it remains to use \eqref{eq:ugly-cov-stability}.
\end{proof}

\subsection{Strong uniqueness and stability}\label{sec:stabil-smooth}
Throughout Section \ref{sec:stabil-smooth} we fix $K\in\N$, use the shorthand $\tau=\tau_K$, and for a solution of $X$ \eqref{eq:main} we define $\tilde X$ as in \eqref{eq:tildeX}.
As in the previous cases, a certain conditional regularity seminorm of $B^H$ appears, so this time we set the shorthand
\begin{equ}
\,[B^H]_{C^1|\bF}=\sup_{s,t\in[0,1]}\big|\E^s B^{H-1}_t\big|.
\end{equ}
\begin{lemma}             \label{lem:Regularisation-smooth}
Let  $p \geq 2$, $s_0\in[0,1)$, $x_0\in\R^d$, $(S,T)\in[s_0,1]^2_\leq$, and let $b$ and $\sigma$ satisfy $\mathbf{A}_{\mathrm{smooth}}$. Let  $X$ be a strong solution of \eqref{eq:main}.
Then there exists constants $N= N(\alpha, H, d,d_0, M,\lambda, p, K)$ and $\eps=\eps(\alpha,H)>0$ such that for all $f \in \cC^\alpha$, all adapted stochastic processes $Z$, and all $(s,t)\in[S,T]^2_\leq$, the  following bound holds
\begin{equs}
\Big\| \int_s^t & \big( f(\tilde{X}_r + Z_r) -f(\tilde{X}_r)\big) \, dr \Big\|_{L_p(\Omega)}
   \leq N  \| f\|_{\cC^\alpha}\| Z\|_{\scC^0_p([S,T])} |t-s|^{1/2+\eps}
\\
&\qquad\qquad+ N \|f\|_{\cC^\alpha}\Big(\big\|(1+ [ B^H]_{\cC^{1}|\bF})Z\big\|_{\scC^0_p([S,T])}
+[Z]_{\scC_p^{1/2}([S,T])}\Big)|t-s|^{1+\eps}.\label{eq:one-big-bound-smooth}
\end{equs}
\end{lemma}

\begin{proof}
We may and will assume $\|f\|_{\cC^\alpha}=1$.
For $S\leq s\leq u\leq r\leq t\leq T$, set
\begin{equs}
\Theta_{s, r} & = \int_{s_0}^rb(\E^s X_{v  \wedge \tau})  \, d v ,
\\
\Sigma_{s,r} &=\int_{s_0}^r \sigma(\E^s X_{v\wedge\tau})B^{H-1}_v\,dv,
\\
\Xi_{s,r}  &= \Theta_{s, r} +\Sigma_{s,r},
\\
 \hat {\Xi}_{s,u, r} &= \Theta_{s, r}  +\Sigma_{s,u}+ \int_u^r\sigma(\E^s X_{v\wedge\tau})\E^uB^{H-1}_v\,dv.
\end{equs}
We wish to apply stochastic sewing in the form of Lemma \ref{lem:SSL} with the processes
\begin{equs}
A_{s, t}&= \E^{s_-} \int_s^t f(\Xi_{s_-, r} +Z_{s_-}) -f(\Xi_{s_-, r}) \, dr,
\\
\cA_t&=\int_S^t \big(f(\tilde{X}_r + Z_r) -f(\tilde{X}_r)\big) \, dr.
\end{equs}
with $s_-=s-(t-s)$ as before.
We shall restrict to pairs $(s,t)$ satisfying $|t-s|\leq \delta$ for some $\delta>0$ to be soon chosen (and depending only on the parameters of $N$),
by Remark \ref{rem:SSL-triv} this is not a restriction.
Let us denote
$$
\hat\Gamma(s,u, r)= \Gamma\big[\sigma(\E^s X_{\cdot\wedge\tau})\big](u,r),
$$ 
where we use the notation $\Gamma[\cdot](\cdot,\cdot)$ from \eqref{eq:ugly-covariance}.
Note that by Proposition \ref{prop:ugly-covariance}, for $s\leq u\leq r$ one has the identity
$
\E^u f(\Xi_{s,r})=\cP_{\hat\Gamma(s,u,r)}f(\hat \Xi_{s,u,r}).
$
Let us furthermore use the last statement of Proposition \ref{prop:ugly-covariance}, with $\bar Y=\sigma(X_{s\wedge\tau})$ and $Y_\cdot=\sigma(\E^s X_{\cdot\wedge\tau})$. The required properties of $\bar Y$ are trivially satisfied thanks to the assumption on $\sigma$.
As for the difference $\|Y-\bar Y\|_{C^0}$, we have for all $v\in[u,r]$
\begin{equ}
|Y_v-\bar Y|\lesssim |\E^s X_{v\wedge\tau}-X_{s\wedge\tau}|\leq|X_{v\wedge\tau}-X_{s\wedge\tau}|\lesssim |r-s|,
\end{equ}
simply using the boundedness of $b$, $\sigma$, and the stopped $B^{H-1}$ in the last step.
Therefore if $|r-s|$ is sufficiently small, then the condition of Proposition \ref{prop:ugly-covariance} is met and we have
\begin{equ}\label{eq:smooth-cov-lowerbound}
\|\hat\Gamma(s,u,r)^{-1}\|\lesssim |r-u|^{2H}.
\end{equ}
We can then proceed similarly to before, the first bound follows easily:
we have a.s.
\begin{equs}
|A_{s, t}|& = \Big|\int_s^t \cP_{\hat\Gamma(s_-,s_-, r)}f (\hat{\Xi}_{s_-,s_-,  r}+Z_{s_-} )- \cP_{\hat\Gamma(s_-,s_-, r)}f (\hat{\Xi}_{s_-,s_-, r}) \, dr\Big| 
\\
&\leq |Z_{s_-}|\int_s^t[\cP_{\hat\Gamma(s_-,s_-, r)}f]_{\cC^1}\,dr 
\\
& \lesssim |Z_{s_-}| \int_s^t (r-s)^{(\alpha-1)H}   \, dr, 
\end{equs}
making use of the bound \eqref{eq:smooth-cov-lowerbound} to get the last line.
After taking $L_p(\Omega)$ norms, we get
\begin{equs}
\| A_{s, t} \|_{L_p(\Omega)} \lesssim \| Z\|_{\scC^0_p([S,T])} (t-s)^{1+(\alpha-1)H}.
\end{equs}
Since $1+(\alpha-1)H>1/2$, condition \eqref{eq:SSL-cond1} is satisfied with $N_1=N  \| Z\|_{\scC^0_p([S,T])}$.

Moving on to \eqref{eq:SSL-cond2}, we proceed similarly to the proof of Lemma \ref{lem:Regularisation} and aim to bound the quantities $J^1$ and $J^2$, introduced in \eqref{eq:def-J1} and \eqref{eq:def-J2}, respectively. 
For $J^1$, we use the bound \eqref{eq:est-F1}, whose second line is now justified by \eqref{eq:smooth-cov-lowerbound}. The only real difference is how we estimate the term $|\hat{\Xi}_{s_-,s, r}-\hat{\Xi}_{s,s, r}|$, whose definition has changed compared to Lemma \ref{lem:Regularisation}: this time we have
\begin{equs}
\hat{\Xi}_{s,s, r}-\hat{\Xi}_{s_-,s, r}=\Theta_{s,r}-\Theta_{s_-,r}&+\Sigma_{s,s}+\int_s^r\sigma(\E^s X_{v\wedge\tau})\E^sB^{H-1}_v\,dv
\\
&\quad-\Sigma_{s_-,s}-\int_s^r\sigma(\E^{s_-}X_{v\wedge\tau})\E^s B^{H-1}_v\,dv.
\end{equs}
First, we have
\begin{equs}
|\Theta_{s,r}-\Theta_{s_-,r}|&=\Big| \int_{s_-}^rb(\E^{s}X_{v\wedge\tau})-b(\E^{s_-}X_{v\wedge\tau})\,dv\Big|
\\
&\lesssim\int_{s_-}^r\big|\E^{s}X_{v\wedge\tau}-\E^{s_-}X_{v\wedge\tau}\big|^\alpha\,dv
\lesssim |t-s|^{1+\alpha H^-},
\end{equs}
using Lemma \ref{lem:apriori-class} in the last step.
Second, we write
\begin{equ}
\big|\Sigma_{s,s}-\Sigma_{s_-,s}|=\Big|\int_{s_-}^s\big(\sigma(X_{v\wedge\tau})-\sigma(\E^{s_-}X_{v\wedge\tau})\big)B^{H-1}_v\,dv\Big|\lesssim |t-s|^{1+H^-}[B^H]_{\cC^1}.
\end{equ}
And finally, we can write
\begin{equ}
\Big|\int_s^r\big(\sigma(\E^s X_{v\wedge\tau})-\sigma(\E^{s_-} X_{v\wedge\tau})\big)\E^{s}B^{H-1}_v\,dv\Big|\lesssim|t-s|^{1+H^-}[B^H]_{\cC^1|\bF}.
\end{equ}
Therefore we get the a.s. bound 
\begin{equ}
|\hat{\Xi}_{s_-,s, r}-\hat{\Xi}_{s,s, r}|\lesssim|t-s|^{1+\alpha H^-}\big(1+[B^H]_{\cC^1|\bF}\big),
\end{equ}
and using this in \eqref{eq:est-F1} we conclude
\begin{equ}\label{eq:J1bound-smooth}
\|J^1_r\|_{L_p(\Omega)}\lesssim\big\|(1+ [ B]_{\cC^{1}|\bF})Z\big\|_{\scC^0_p([S,T])}(t-s)^{1+\alpha H^-+(\alpha-2) H}
+[Z]_{\scC_p^{1/2}([S,T])}(t-s)^{1/2+(\alpha-1)H}.
\end{equ}
For $J^2$ we write 
\begin{equs}
|J^2_r| &\leq    |Z_s |  \big[ \cP_{\hat\Gamma(s_-,s, r)}f- \cP_{\hat\Gamma(s,s, r)}f\big]_{\cC^1} 
\\
&\lesssim  |Z_s | \big|\hat\Gamma(s_-,s, r)- \hat\Gamma(s,s, r)\big| \big(\big| \hat\Gamma(s_-,s, r)^{-1}\big| \vee\big|\hat \Gamma(s,s, r)^{-1}\big|\big)^{3/2-\alpha/2 } 
\\
 &\lesssim |Z_s | \big|\hat\Gamma(s_-,s, r)- \hat\Gamma(s,s, r)\big| (r-s)^{-3H+H \alpha}.
\end{equs}
By the definition of $\hat \Gamma$ and \eqref{eq:ugly-cov-stability} we have
\begin{equs}
\big|\hat\Gamma(s_-,s, r)- \hat\Gamma(s,s, r)\big|&\lesssim|r-s|^{2H}\big\|\sigma(\E^{s_-} X_{\cdot\wedge\tau})-\sigma(\E^s X_{\cdot\wedge\tau})\big\|_{\cC^0([s,r])}
\\
&\lesssim |t-s|^{2H+H^{-}},
\end{equs}
using Lemma \ref{lem:apriori-class} in the last step.
Therefore we can conclude
\begin{equ}\label{eq:J2bound-smooth}
\|J^2_r\|_{L_p(\Omega)}\lesssim \|Z\|_{\scC^0_p([S,T])}|t-s|^{H^--H+\alpha H}.
\end{equ}
By assumption, all of $1+\alpha H^-+(\alpha-2) H$, $1/2+(\alpha-1)H$, and $H^--H+\alpha H$ are positive. Therefore as in the proof of Lemma \ref{lem:Regularisation}, from \eqref{eq:J1bound-smooth}-\eqref{eq:J2bound-smooth} we get the desired bound on $\E^{s_-}\delta A_{s,u,t}$:
\begin{equ}
\| \E^s \delta A_{s, u, t} \|_{L_p(\Omega)}
\lesssim \Big(\big\|(1+ [ B^H]_{\cC^{1}|\bF})Z\big\|_{\scC^0_p([S,T])}
+[Z]_{\scC_p^{1/2}([S,T])}\Big)|t-s|^{1+\eps_2}
\end{equ}
for some $\eps_2>0$. Indeed, this verifies  condition \eqref{eq:SSL-cond2} with $N_2=N\Big(\big\|(1+ [ B^H]_{\cC^{1}|\bF})Z\big\|_{\scC^0_p([S,T])}
+[Z]_{\scC_p^{1/2}([S,T])}\Big)$.

The verification of \eqref{eq:SSL-cond3} carries through from the proof of Lemma \ref{lem:Regularisation} without any serious change: note that the argument only requires $[\tilde X]_{\scC^\gamma_p}<\infty$ for some $\gamma>0$ (which here is satisfied with $\gamma=1$) and $\|\Xi_{s,r}-\tilde X_s\|_{L_p(\Omega)}\lesssim |r-s|^\gamma$ for some $\gamma>0$ (which is also satisfied with $\gamma=1$).
Therefore all the conditions of Lemma \ref{lem:SSL} are satisfied and \eqref{eq:one-big-bound-smooth} follows from \eqref{eq:SSL-conc}.
\end{proof}

\begin{lemma}                              \label{lem:stability-smooth}
Let  $p \geq 1$, $s_0\in[0,1)$, $x_0, y_0 \in\R^d$. Let $b$, $\tilde b$, and $\sigma$ satisfy $\mathbf{A}_{\mathrm{smooth}}$.
Then there exists a constant $N_1$, depending only on  $\alpha, H, d,d_0, M, \lambda, p,$ and $ K$, as well as a constant $\gamma=\gamma(\alpha,H)\in(0,2)$,  with the following property:
 for any  $X$ and $Y$ strong solutions of  \eqref{eq:main} and \eqref{eq:aux}, respectively, and for all $C\geq 1$,  the following estimate holds
\begin{equs}\label{eq:stability-smooth}
\| X_{\cdot \wedge\tau }- Y_{\cdot \wedge\tau }\|_{\scC^{1/2}_p([s_0,1])}  \leq N_1 ^{C^{2-\gamma}} \Big( |x_0-y_0|+  \big( \P ([ B^H ]_{\cC^1|\bF}\geq C ) \big)^{1/2p} +\|b-\tilde b\|_{C^0(\R^d)} \Big).  
\end{equs}
\end{lemma}
\begin{proof}
The proof is in large part identical to that of Lemma \ref{lem:stability}.
As therein, taking $(S,T)\in[s_0,1]^2_{\leq}$ and denoting $Z=X-Y$, one gets (after applying Lemma \ref{lem:Regularisation-smooth})
\begin{equs}
\|Z_{\cdot\wedge\tau}\|_{\scC^{1/2}_p([S,T])}&\lesssim C\| Z_{S \wedge \tau} \|_{L_p(\Omega)} +   C \| Z_{\cdot \wedge \tau}\|_{\scC^{1/2}_p([S,T])} |T-S|^{1/2+\eps}
\\
& \qquad + \big( \P ([ B^H ]_{\cC^1|\bF}\geq C ) \big)^{1/2p} + \|b-\tilde b\|_{\cC^0(\R^d)}
\\
&\qquad+\Big[\int_{s_0}^\cdot \big(\sigma (X_{r\wedge \tau})-\sigma(Y_{r\wedge \tau}) \big)B^{H-1}_r\,dr \Big]_{\scC^{1/2}_p([S,T])}.\label{eq:pre-buckle-smooth}
\end{equs}
The only difference lies in the treatment of the last term: for $(s,t)\in[S,T]^2_\leq$ we can write
\begin{equ}
\Big|\int_{s}^t \big(\sigma (X_{r\wedge \tau})-\sigma(Y_{r\wedge \tau}) \big)B^{H-1}_r\,dr\Big| \lesssim|t-s|\|Z_{\cdot\wedge\tau}\|_{\cC^0([S,T])}[B^{H}]_{\cC^1}.
\end{equ}
Choose $p$ large enough so that $ \| \cdot \|_{ L_p\big(\Omega; \cC^{0}([S,T]) \big)} \lesssim \| \cdot \|_{\scC^{1/4}_p([S, T])}$. Taking $L_p(\Omega)$ norms and taking supremum over $(s,t)\in[S,T]^2_\leq$, we get
\begin{equs}
\Big[\int_{s_0}^\cdot \big(\sigma (X_{r\wedge \tau}) & -\sigma(Y_{r\wedge \tau})\big) B^{H-1}_r\,dr \Big]_{\scC^{1/2}_p([S,T])}\lesssim |T-S|^{1/2}\big\|Z_{\cdot\wedge\tau}[B^{H}]_{C^1}\big\|_{ L_p\big(\Omega; \cC^{0}([S,T]) \big)}
\\
&\lesssim C|T-S|^{1/2}\|Z_{\cdot\wedge\tau}\|_{ L_p\big(\Omega; \cC^{0}([S,T]) \big)}+\big( \P ([ B^H ]_{\cC^1}\geq C ) \big)^{1/2p}
\\
&\lesssim C|T-S|^{1/2}\|Z_{\cdot\wedge\tau}\|_{ \scC^{1/4}_p([S,T])}+\big( \P ([ B^H ]_{\cC^1}\geq C ) \big)^{1/2p}
\\
&\lesssim C\|Z_{S\wedge\tau}\|_{L_p(\Omega)}+C|T-S|^{3/4}\|Z_{\cdot\wedge\tau}\|_{\scC^{1/2}_p([S,T])}+\big( \P ([ B^H ]_{\cC^1}\geq C ) \big)^{1/2p}.
\end{equs}
Using this bound in \eqref{eq:pre-buckle-smooth}, we get the exact analogue of \eqref{eq:buckling}. The proof is then concluded in the exact same way.
\end{proof}

\subsection{Strong existence and path-by-path uniqueness}
Once the estimate \eqref{eq:stability-smooth} is available, the rest of the argument from Section \ref{sec:Young} carries through with minimal changes, all of which are in fact simplifications due to the ``stochastic'' integral being easier to handle.
We therefore only state the main conclusions.
\begin{theorem}\label{thm:everythin-smooth}
Let $b$ and $\sigma$ satisfy $\mathbf{A}_{\mathrm{smooth}}$. Then for any $s_0\in[0,1)$ and $x_0\in\R^d$ there exists a strong solution to \eqref{eq:main}.
Moreover, there exists an event $\what\Omega\subset\Omega_H$ of full probability such that for any $\omega\in\what\Omega$, $s_0\in[0,1)$, $x_0\in\R^d$, and any two solutions $Y$ and $Y'$ of \eqref{eq:main}, one has $Y=Y'$.
\end{theorem}

{\appendix

\section{Technical tools}

We  recall a number of basic ingredients for stochastic sewing.
First let us recall the following version of the stochastic sewing lemma.
We use the notation $\overline{[S,T]}_\leq^2=\{(s,t)\in[S,T]^2:\,s-(t-s)\geq S\}$.
\begin{lemma}[Shifted stochastic sewing lemma \cite{Khoa, Mate}]\label{lem:SSL}
Let $0\leq S<T\leq1$, $p\geq 2$ and let $(A_{s, t})_{(s,t)\in\overline{[S,T]}_{\leq}^2}$ be a family of random variables in $L_p(\Omega,\R^d)$ such that $A_{s,t}$ is $\cF_t$-measurable.
Suppose that for some $\eps_1,\eps_2>0$ and $N_1,N_2\geq0$ the bounds
\begin{align}
&\|A_{s,t}\|_{L_p(\Omega)}  \leq N_1|t-s|^{1/2+\eps_1}\,,\label{eq:SSL-cond1}
\\
&\|\E^{s-(t-s)}\delta A_{s,u,t}\|_{L_p(\Omega)}  \leq N_2 |t-s|^{1+\eps_2}\label{eq:SSL-cond2}
\end{align}
hold for all $(s,t)\in\overline{[S,T]}_\leq^2$ and $u=(s+t)/2$.
Suppose furthermore that there exists a  process $\cA=\{\cA_t:t\in[S,T]\}$ and constants $\eps_3>0$, $N_3\geq 0$, such that for any $(s,t)\in\overline{[S,T]}_\leq^2$  one has
\begin{equation}\label{eq:SSL-cond3}
\|\cA_t-\cA_s-A_{s,t}\|_{L_p(\Omega)}\leq N_3|t-s|^{1+\eps_3}.
\end{equation}
Then there exist constants $K_1,K_2>0$, which depend only on $\eps_1,\eps_2$, $p$, and $d$, such that for any $(s,t)\in[S,T]_\leq^2$ one has the bound
\begin{equation}\label{eq:SSL-conc}
	\|\cA_t-\cA_s\|_{L_p(\Omega)}  \leq  K_1N_1 |t-s|^{1/2+\eps_1}+K_2N_2 |t-s|^{1+\eps_2}.
\end{equation}
\end{lemma}
\begin{remark}\label{rem:SSL-triv}
A trivial extension of Lemma \ref{lem:SSL} is that if for some $\delta>0$ one requires \eqref{eq:SSL-cond1}-\eqref{eq:SSL-cond2}-\eqref{eq:SSL-cond3} to hold only for $(s,t)$ such that $|t-s|\leq \delta$, then \eqref{eq:SSL-conc} still holds for all $(s,t)\in[S,T]_{\leq}^2$, but with the constants $K_1,K_2$ also depending on $\delta$.
\end{remark}

\begin{proposition}
Let $f\in \cC^1(\R^d)$, $g\in \cC^2(\R^d)$, and $x_1,x_2,x_3,x_4\in\R^d$. Then one has the bounds
\begin{equs}    
&|f(x_1)-f(x_2)|\leq |x_1-x_2|[f]_{\cC^1}\,,\label{eq:C1}\\
&|g(x_1)-g(x_2)-g(x_3)+g(x_4)|\leq |x_1-x_2-x_3+x_4|[g]_{\cC^1}+|x_1-x_2||x_1-x_3|[g]_{\cC^2}\,.\label{eq:C2}
\end{equs}
\end{proposition}

For the next two heat kernel bounds recall the notations $p_\Gamma$ and $\cP_\Gamma$ from \eqref{eq:HKdef}. Proposition \ref{prop:HK-Gamma} is very standard, Proposition \ref{prop:HK-Gammadiff} is also folklore, in this exact form it is proved in \cite[Proposition 2.7]{DGL}.
\begin{proposition}\label{prop:HK-Gamma}
Let $K>0$ and let $\Gamma\in\R^{d\times d}$ be a positive definite matrix such that $\|\Gamma\|\leq K$. Then for any $-1<\alpha\leq \beta\leq 2$ there exists a constant $N=N(d,\alpha,\beta,K)$ such that for any $f\in \cC^\alpha$ one has the bound
\begin{equ}\label{eq:HKKK}
\|\cP_\Gamma f\|_{\cC^\beta}\leq N|\Gamma^{-1}|^{(\beta-\alpha)/2}\|f\|_{\cC^\alpha}.
\end{equ}
\end{proposition}
\begin{proposition}\label{prop:HK-Gammadiff}
Let $K>0$ and let $\Gamma,\bar\Gamma\in\R^{d\times d}$ be positive definite  matrices such that $K^{-1}\I\preceq\Gamma\bar\Gamma^{-1}\preceq K\I$. Then there exists a constant $N=N(d,K)$ such that for all $x\in\R^d$ one has the bound
\begin{equs}
|p_\Gamma(x)-p_{\bar\Gamma}(x)|&\leq N 
|\I-\Gamma\bar\Gamma^{-1}|
\big(p_{\Gamma/2}(x)+p_{\bar\Gamma/2}(x)\big).
\label{eq:different-gamma}
\end{equs}
\end{proposition}
\begin{proof}[Proof of \eqref{eq:sum_partition}]
Using the shorthand $f_{s,t}:=f_t-f_s$ for increments, the bound follows from the following telescopic rearrangement:
\begin{equs}
 f_{u_k}-f_{u_0}-f'_{u_0} g_{u_0, u_k}&= \sum_{i=1}^k( f_{u_{i-1} , u_i} -f'_{u_{i-1}}g_{u_{i-1}, u_i}) -f'_{u_0} g_{u_0, u_k} + \sum_{i=1}^kf'_{u_{i-1}}g_{u_{i-1}, u_i}
 \\
 &=\sum_{i=1}^k( f_{u_{i-1} , u_i} -f'_{u_{i-1}}g_{u_{i-1}, u_i})+\sum_{i=1}^{k-1}  f'_{u_{i-1}, u_i}g_{u_i, u_k}.
\end{equs}
\end{proof}
\begin{proof}[Proof of \eqref{eq:quadratic-controled-estimate}]
With the shorthand notation for increments as before, for $v, u \in [s, t]$ we have 
\begin{equs}
|&F(f_u)-F(f_v) - \nabla F(f_v) f'_v g_{v, u}|
\\
&= \Big| \int_0^1 \nabla F(f_v+ \theta f_{v, u} ) f_{v, u}\, d \theta -\nabla F(f_v) f'_v g_{v, u} \Big|
\\
&\leq \Big| \int_0^1 \big( \nabla F(f_v+ \theta f_{v, u} )-\nabla F(f_v) \big) f'_v g_{v, u}\, d \theta\Big|+\|F\|_{\cC^1}[(f,f')]_{\cD^\gamma_g([s,t])}|u-v|^\gamma
\\
&\leq \|F\|_{\cC^{1+\beta}}\|f'\|_{\cC^0}\|g\|_{\cC^\alpha}[f]_{\cC^\alpha}^\beta|v-u|^{\alpha(1+\beta)}+\|F\|_{\cC^1}[(f,f')]_{\cD^\gamma_g([s,t])}|u-v|^\gamma.
\end{equs}
More easily, one has
\begin{equs}
|\nabla F(f_u)f_u'- \nabla F(f_v)f_v'| & \leq \| F\|_{\cC^1}[f']_{\cC^\alpha}|v-u|^\alpha+ \|F\|_{\cC^{1+\beta}}\|f'\|_{\cC^0}[f]_{\cC^\alpha}^\beta|v-u|^{\beta\alpha}. 
\end{equs} 
Combining these two bounds with \eqref{eq:triv-controll-embedding} yields \eqref{eq:quadratic-controled-estimate}.
\end{proof}

}

\bibliographystyle{Martin}
\bibliography{multi-bib}

\begin{thebibliography}{CHLT15}
\expandafter\ifx\csname url\endcsname\relax
  \def\url#1{\texttt{#1}}\fi
\expandafter\ifx\csname urlprefix\endcsname\relax\def\urlprefix{URL }\fi
\expandafter\ifx\csname href\endcsname\relax
  \def\href#1#2{#2}\fi
\expandafter\ifx\csname burlalt\endcsname\relax
  \def\burlalt#1#2{\href{#2}{\texttt{#1}}}\fi

\bibitem[ART21]{ART21}
\textsc{L.~Anzeletti}, \textsc{A.~Rochard}, and \textsc{E.~Tanr\'e}.
\newblock Regularisation by fractional noise for one-dimensional differential
  equations with nonnegative distributional drift.
\newblock \emph{arXiv preprint arXiv:2112.05685} (2021).

\bibitem[BH07]{BH07}
\textsc{F.~Baudoin} and \textsc{M.~Hairer}.
\newblock A version of {H\"o}rmander's theorem for the fractional brownian
  motion.
\newblock \emph{Probability Theory and Related Fields} \textbf{139}, no. 3-4,
  (2007), 373--395.
\newblock
  \burlalt{doi:10.1007/s00440-006-0035-0}{http://dx.doi.org/10.1007/s00440-006-0035-0}.

\bibitem[BH22]{BH22}
\textsc{F.~{Bechtold}} and \textsc{M.~{Hofmanov{\'a}}}.
\newblock {Weak solutions for singular multiplicative SDEs via regularization
  by noise}.
\newblock \emph{arXiv e-prints} (2022).
\newblock \burlalt{arXiv:2203.13745}{http://arxiv.org/abs/2203.13745}.

\bibitem[BNP19]{BNP}
\textsc{D.~Ba{\~{n}}os}, \textsc{T.~Nilssen}, and \textsc{F.~Proske}.
\newblock Strong existence and higher order {F}r{\'{e}}chet differentiability
  of stochastic flows of fractional brownian motion driven {SDEs} with singular
  drift.
\newblock \emph{Journal of Dynamics and Differential Equations} (2019).
\newblock
  \burlalt{doi:10.1007/s10884-019-09789-4}{http://dx.doi.org/10.1007/s10884-019-09789-4}.

\bibitem[CF10]{CF10}
\textsc{T.~Cass} and \textsc{P.~Friz}.
\newblock Densities for rough differential equations under hörmander's
  condition.
\newblock \emph{Annals of Mathematics} \textbf{171}, no.~3, (2010), 2115--2141.
\newblock
  \burlalt{doi:10.4007/annals.2010.171.2115}{http://dx.doi.org/10.4007/annals.2010.171.2115}.

\bibitem[CG16]{CG16}
\textsc{R.~Catellier} and \textsc{M.~Gubinelli}.
\newblock Averaging along irregular curves and regularisation of {ODEs}.
\newblock \emph{Stochastic Processes and their Applications} \textbf{126},
  no.~8, (2016), 2323--2366.
\newblock
  \burlalt{doi:10.1016/j.spa.2016.02.002}{http://dx.doi.org/10.1016/j.spa.2016.02.002}.

\bibitem[CHLT15]{CHLT15}
\textsc{T.~Cass}, \textsc{M.~Hairer}, \textsc{C.~Litterer}, and
  \textsc{S.~Tindel}.
\newblock Smoothness of the density for solutions to gaussian rough
  differential equations.
\newblock \emph{The Annals of Probability} \textbf{43}, no.~1(2015).
\newblock \burlalt{doi:10.1214/13-aop896}{http://dx.doi.org/10.1214/13-aop896}.

\bibitem[CHM18]{CHM}
\textsc{P.-E. {Chaudru de Raynal}}, \textsc{I.~{Honor{\'e}}}, and
  \textsc{S.~{Menozzi}}.
\newblock {Strong regularization by Brownian noise propagating through a weak
  H{\"o}rmander structure}.
\newblock \emph{arXiv e-prints} (2018).
\newblock \burlalt{arXiv:1810.12225}{http://arxiv.org/abs/1810.12225}.

\bibitem[Dav07]{Davie}
\textsc{A.~M. Davie}.
\newblock Uniqueness of solutions of stochastic differential equations.
\newblock \emph{Int. Math. Res. Not. IMRN} , no.~24, (2007), Art. ID rnm124,
  26.
\newblock
  \burlalt{doi:10.1093/imrn/rnm124}{http://dx.doi.org/10.1093/imrn/rnm124}.

\bibitem[DGL21]{DGL}
\textsc{K.~{Dareiotis}}, \textsc{M.~{Gerencs{\'e}r}}, and \textsc{K.~{L{\^e}}}.
\newblock {Quantifying a convergence theorem of Gy{\"o}ngy and Krylov}.
\newblock \emph{arXiv e-prints}  arXiv:2101.12185.
\newblock \burlalt{arXiv:2101.12185}{http://arxiv.org/abs/2101.12185}.

\bibitem[FH20]{Friz-Hairer}
\textsc{P.~K. Friz} and \textsc{M.~Hairer}.
\newblock \emph{A Course on Rough Paths}.
\newblock Springer International Publishing, 2020.

\bibitem[FHL21]{FHL}
\textsc{P.~K. Friz}, \textsc{A.~Hocquet}, and \textsc{K.~L{\^e}}.
\newblock Rough stochastic differential equations.
\newblock \emph{arXiv preprint arXiv:2106.10340} (2021).
\newblock \burlalt{arXiv:2106.10340}{http://arxiv.org/abs/2106.10340}.

\bibitem[Ger19]{Mate19}
\textsc{M.~Gerencs{\'{e}}r}.
\newblock Boundary regularity of stochastic {PDEs}.
\newblock \emph{The Annals of Probability} \textbf{47}, no.~2(2019).
\newblock
  \burlalt{doi:10.1214/18-aop1272}{http://dx.doi.org/10.1214/18-aop1272}.

\bibitem[Ger22]{Mate}
\textsc{M.~Gerencs{\'{e}}r}.
\newblock Regularisation by regular noise.
\newblock \emph{Stochastics and Partial Differential Equations: Analysis and
  Computations} (2022).
\newblock
  \burlalt{doi:10.1007/s40072-022-00242-0}{http://dx.doi.org/10.1007/s40072-022-00242-0}.

\bibitem[GG21]{GGNoiseless}
\textsc{L.~Galeati} and \textsc{M.~Gubinelli}.
\newblock Noiseless regularisation by noise.
\newblock \emph{Revista Matem{\'a}tica Iberoamericana} (2021).

\bibitem[GG22]{Lucio-Mate}
\textsc{L.~Galeati} and \textsc{M.~Gerencs{\'e}r}.
\newblock \emph{In preparation} (2022).

\bibitem[GH20]{GH20}
\textsc{L.~{Galeati}} and \textsc{F.~A. {Harang}}.
\newblock {Regularization of multiplicative SDEs through additive noise}.
\newblock \emph{arXiv e-prints} (2020).
\newblock \burlalt{arXiv:2008.02335}{http://arxiv.org/abs/2008.02335}.

\bibitem[GHM21]{GHM}
\textsc{L.~{Galeati}}, \textsc{F.~A. {Harang}}, and \textsc{A.~{Mayorcas}}.
\newblock {Distribution dependent SDEs driven by additive fractional Brownian
  motion}.
\newblock \emph{arXiv e-prints}  arXiv:2105.14063.
\newblock \burlalt{arXiv:2105.14063}{http://arxiv.org/abs/2105.14063}.

\bibitem[HN07]{HN07}
\textsc{Y.~Hu} and \textsc{D.~Nualart}.
\newblock Differential equations driven by h{\"o}lder continuous functions of
  order greater than 1/2.
\newblock In \textsc{F.~E. Benth}, \textsc{G.~Di~Nunno},
  \textsc{T.~Lindstr{\o}m}, \textsc{B.~{\O}ksendal}, and \textsc{T.~Zhang},
  eds., \emph{Stochastic Analysis and Applications},  399--413. Springer Berlin
  Heidelberg, Berlin, Heidelberg, 2007.

\bibitem[HP21]{HP}
\textsc{F.~A. Harang} and \textsc{N.~Perkowski}.
\newblock C$\infty$- regularization of {ODEs} perturbed by noise.
\newblock \emph{Stochastics and Dynamics} \textbf{21}, no.~08(2021).
\newblock
  \burlalt{doi:10.1142/s0219493721400104}{http://dx.doi.org/10.1142/s0219493721400104}.

\bibitem[L{\^e}20]{Khoa}
\textsc{K.~L{\^e}}.
\newblock A stochastic sewing lemma and applications.
\newblock \emph{Electronic Journal of Probability} \textbf{25}(2020).
\newblock
  \burlalt{doi:doi:10.1214/20-EJP442}{http://dx.doi.org/doi:10.1214/20-EJP442}.

\bibitem[Lyo98]{Lyons}
\textsc{T.~J. Lyons}.
\newblock Differential equations driven by rough signals.
\newblock \emph{Revista Matemática Iberoamericana} \textbf{14}, no.~2, (1998),
  215--310.

\bibitem[MP22]{Toyomu-Nicolas}
\textsc{T.~{Matsuda}} and \textsc{N.~{Perkowski}}.
\newblock {An extension of the stochastic sewing lemma and applications to
  fractional stochastic calculus}.
\newblock \emph{arXiv e-prints} (2022).
\newblock \burlalt{arXiv:2206.01686}{http://arxiv.org/abs/2206.01686}.

\bibitem[NO02]{NO1}
\textsc{D.~Nualart} and \textsc{Y.~Ouknine}.
\newblock Regularization of differential equations by fractional noise.
\newblock \emph{Stochastic Processes and their Applications} \textbf{102},
  no.~1, (2002), 103--116.
\newblock
  \burlalt{doi:10.1016/s0304-4149(02)00155-2}{http://dx.doi.org/10.1016/s0304-4149(02)00155-2}.

\bibitem[NO03]{NO2}
\textsc{D.~Nualart} and \textsc{Y.~Ouknine}.
\newblock Stochastic differential equations with additive fractional noise and
  locally unbounded drift.
\newblock In \emph{Stochastic inequalities and applications}, vol.~56 of
  \emph{Progr. Probab.},  353--365. Birkh\"{a}user, Basel, 2003.

\bibitem[Sha16]{shap}
\textsc{A.~V. Shaposhnikov}.
\newblock Some remarks on {D}avie's uniqueness theorem.
\newblock \emph{Proceedings of the Edinburgh Mathematical Society} \textbf{59},
  no.~4, (2016), 1019--1035.

\bibitem[Ver80]{Veretennikov}
\textsc{A.~J. Veretennikov}.
\newblock Strong solutions and explicit formulas for solutions of stochastic
  integral equations.
\newblock \emph{Mat. Sb. (N.S.)} \textbf{111(153)}, no.~3, (1980), 434--452,
  480.

\bibitem[Zha05]{Zhang}
\textsc{X.~Zhang}.
\newblock Strong solutions of {SDE}s with singular drift and sobolev diffusion
  coefficients.
\newblock \emph{Stochastic Processes and their Applications} \textbf{115},
  no.~11, (2005), 1805--1818.
\newblock
  \burlalt{doi:https://doi.org/10.1016/j.spa.2005.06.003}{http://dx.doi.org/https://doi.org/10.1016/j.spa.2005.06.003}.

\end{thebibliography}
\end{document}